\documentclass[a4paper,11pt]{amsart}
\usepackage[left=2.7cm,right=2.7cm,top=3.5cm,bottom=3cm]{geometry}

\usepackage{amsthm,amssymb,amsmath,amsfonts,mathrsfs,amscd,amsbsy,dsfont,verbatim}
\usepackage{stmaryrd,dsfont}
\usepackage[new]{old-arrows}
\usepackage[latin1]{inputenc}
\usepackage[all,cmtip]{xy}
\usepackage{latexsym}
\usepackage{longtable}
\usepackage{mathtools}
\usepackage{marginnote}
\usepackage{graphicx}
\newcommand*{\Scale}[2][4]{\scalebox{#1}{$#2$}}%

\newcounter{lettera}

\newtheorem{itheorem}[lettera]{Theorem}

\usepackage[pagebackref]{hyperref}

\mathtoolsset{showonlyrefs}

\usepackage{graphicx}

\numberwithin{equation}{section}

\newfont{\cyr}{wncyr10 scaled 1100}
\newfont{\cyrr}{wncyr9 scaled 1000}

\theoremstyle{plain}
\newtheorem{theorem}{Theorem}[section]

\newtheorem{proposition}[theorem]{Proposition}
\newtheorem{lemma}[theorem]{Lemma}
\newtheorem{corollary}[theorem]{Corollary}

\theoremstyle{definition}
\newtheorem{definition}[theorem]{Definition}
\newtheorem{assumption}[theorem]{Assumption}

\theoremstyle{remark}
\newtheorem{remark}[theorem]{Remark}

\newtheorem{remark/notation}[theorem]{Remark/Notation}
\newtheorem{notation/convention}[theorem]{Notation/Convention}

\newcommand{\Q}{\mathds Q}
\newcommand{\N}{\mathds N}
\newcommand{\Z}{\mathds Z}
\newcommand{\R}{\mathds R}
\newcommand{\C}{\mathds C}

\newcommand{\F}{\mathds F}
\newcommand{\T}{\mathds T}

\newcommand{\defeq}{\vcentcolon=}

\DeclareMathOperator{\Spec}{Spec}
\DeclareMathOperator{\Pic}{Pic}
\DeclareMathOperator{\End}{End}
\DeclareMathOperator{\Aut}{Aut}
\DeclareMathOperator{\Frob}{Frob}

\DeclareMathOperator{\Norm}{N}
\DeclareMathOperator{\Hom}{Hom}

\DeclareMathOperator{\Gal}{Gal}
\DeclareMathOperator{\GL}{GL}

\DeclareMathOperator{\SL}{SL}
\DeclareMathOperator{\Sel}{Sel}

\DeclareMathOperator{\M}{M}

\DeclareMathOperator{\CH}{CH}
\DeclareMathOperator{\AJ}{AJ}
\DeclareMathOperator{\Ind}{Ind}

\DeclareMathOperator{\Fil}{Fil}

\DeclareMathOperator{\im}{im}
\DeclareMathOperator{\coker}{coker}
\DeclareMathOperator{\corank}{corank}
\DeclareMathOperator{\rank}{rank}

\DeclareMathOperator{\length}{length}

\newcommand{\res}{\mathrm{res}}
\newcommand{\cores}{\mathrm{cores}}

\newcommand{\cyc}{{\rom{cyc}}}

\newcommand{\ord}{\mathrm{ord}}

\newcommand{\unr}{\mathrm{unr}}

\newcommand{\alg}{\mathrm{alg}}

\newcommand{\tor}{\mathrm{tors}}

\newcommand{\divv}{\mathrm{div}}

\newcommand{\loc}{\mathrm{loc}}

\newcommand{\cchar}{\mathrm{char}}

\newcommand{\Sha}{\mbox{\cyr{X}}}

\usepackage[usenames]{color}
\definecolor{Indigo}{rgb}{0.2,0.1,0.7}
\definecolor{Violet}{rgb}{0.5,0.1,0.7}
\definecolor{White}{rgb}{1,1,1}
\definecolor{Green}{rgb}{0.1,0.9,0.2}

\newcommand{\longmono}{\mbox{\;$\lhook\joinrel\longrightarrow$\;}}

\newcommand{\longepi}{\mbox{\;$\relbar\joinrel\twoheadrightarrow$\;}}

\newcommand{\smallmat}[4]{\bigl(\begin{smallmatrix}#1&#2\\#3&#4\end{smallmatrix}\bigr)}

\newcommand{\dirlim}{\mathop{\varinjlim}\limits}
\newcommand{\invlim}{\mathop{\varprojlim}\limits}

\newcommand{\Tr}{\operatorname{Tr}}

\newfont{\gotip}{eufb10 at 12pt}

\newcommand{\cO}{{\mathcal O}}

\newcommand{\p}{\mathfrak{p}}

\newcommand{\Sym}{\operatorname{Sym}}

\newcommand{\rom}{\mathrm}

\newcommand{\dR}{\mathrm{dR}}
\newcommand{\et}{\text{\'et}}

\makeatletter
\@namedef{subjclassname@2020}{%
  \textup{2020} Mathematics Subject Classification}
\makeatother

\begin{document}

\title[Anticyclotomic main conjectures for modular forms]{Anticyclotomic Iwasawa main conjectures\\for modular forms}
%\today
%\date{}
\author{Matteo Longo, Maria Rosaria Pati and Stefano Vigni}

\thanks{The authors are partially supported by PRIN 2022 ``The arithmetic of motives and $L$-functions'' and by the GNSAGA group of INdAM. The research by the second and third authors is partially supported by the MUR Excellence Department Project awarded to Dipartimento di Matematica, Universit\`a di Genova, CUP D33C23001110001.}

%\thanks{The authors are partially supported by PRIN 2017 ``Geometric, algebraic and analytic methods in arithmetic''.}

\begin{abstract}
Let $f$ be a newform of even weight at least $4$, level $N$ and trivial character. Let $p\nmid N$ be an odd prime number that is ordinary for $f$ and let $K$ be an imaginary quadratic field satisfying a generalized Heegner hypothesis relative to $N$. In this paper, we prove (under mild arithmetic assumptions) Iwasawa main conjectures for $f$ over the anticyclotomic $\Z_p$-extension of $K$ both in the definite setting and in the indefinite setting (in the second case, we prove a main conjecture \emph{\`a la} Perrin-Riou for modular forms). Our strategy of proof follows the approach of Bertolini--Darmon via congruences combined with our previous results on an analogue for $f$ of Kolyvagin's conjecture on the non-triviality of his $p$-adic system of derived Heegner points on elliptic curves. As a second contribution, when $p$ splits in $K$ we prove an Iwasawa--Greenberg main conjecture for the $p$-adic $L$-functions of Bertolini--Darmon--Prasanna and Brooks.
\end{abstract}

\include{thebibliography}

\address{Dipartimento di Matematica, Universit\`a di Padova, Via Trieste 63, 35121 Padova, Italy}
\email{mlongo@math.unipd.it}
\address{Dipartimento di Matematica, Universit\`a di Genova, Via Dodecaneso 35, 16146 Genova, Italy}
\email{mariarosaria.pati@unige.it}
\address{Dipartimento di Matematica, Universit\`a di Genova, Via Dodecaneso 35, 16146 Genova, Italy}
\email{stefano.vigni@unige.it}



\maketitle

%\tableofcontents

\section{Introduction} \label{Intro}

Let $p$ be a prime number. The present paper offers a proof (under standard arithmetic assumptions) of the Iwasawa main conjectures (IMC, for short) over anticyclotomic $\Z_p$-extensions of imaginary quadratic fields for even weight modular forms that are ordinary at $p$, both in the definite setting and in the indefinite setting. In our strategy, we follow the approach via congruences between modular forms that was originally proposed by Bertolini--Darmon in the context of elliptic curves (\cite{BD-IMC}); this method was later developed by, among others, Howard in the abstract setting of bipartite Euler systems (\cite{howard-bipartite}), W. Zhang in his proof of Kolyvagin's conjecture and the $p$-part of the Birch--Swinnerton-Dyer formula for elliptic curves (\cite{Zhang}) and, more recently, Bertolini--Longo--Venerucci (\cite{BLV}); see also work by Burungale--Castella--Kim (\cite{BCK}) and Burungale--B\"uy\"ukboduk--Lei (\cite{BBL}). The anticyclotomic Iwasawa theory of modular forms was previously considered by Chida--Hsieh in the definite case (\cite{ChHs2}) and by Longo--Vigni in the indefinite case (\cite{LV-kyoto}); in \cite{ChHs2} and \cite{LV-kyoto}, only one inclusion in the equality predicted by the main conjecture is proved (namely, the one that is obtained via Euler system arguments) and the purpose of this paper is to complete the proof in both cases. As we have already remarked, our approach is via congruences: it follows mainly \cite{BD-IMC}, \cite{BLV}, \cite{howard-bipartite} and reduces the missing divisibility to a non-triviality result for a certain system of Galois cohomology classes that is an analogue for modular forms of Kolyvagin's conjecture on the $p$-indivisibility of his system of derived Heegner points on elliptic curves: such a higher weight conjecture has been established (under analogous assumptions) in \cite{LPV}. In fact, this is essentially a by-product of converse theorems due to Skinner--Urban (\cite{SU}) and generalized by Fouquet--Wan (\cite{FW}). It is worth remarking that our underlying bipartite Euler system is constructed by means of a slight extension of results of Wang (\cite{wang}), which in turn generalize results in \cite{Zhang} to a higher weight setting.  

To describe our results, let $f\in S_k(\Gamma_0(N))$ be a newform of even weight $k\geq 4$ and square-free level $N$, whose $q$-expansion will be denoted by $f(q)=\sum_{n\geq1}a_n(f)q^n$. Let $K$ be an imaginary quadratic field, write $D_K$ for the discriminant of $K$ and consider a factorization $N=MD$ of $N$ into a product of positive integers $M$ and $D$ with the property that if a prime number divides $M$ (respectively, $D$), then it splits (respectively, is inert) in $K$. We say that we are in the \emph{definite} (respectively, \emph{indefinite}) case if $D$ is a product of an \emph{odd} (respectively, \emph{even}) number of primes (notice that $D=1$ is an admissible option). Now pick a prime number $p$ such that
\begin{itemize}
\item $p\nmid ND_K$;
\item $p$ is ordinary for $f$.
\end{itemize}
Let $F\defeq\Q\bigl(a_n(f)\mid n\geq1\bigr)$ be the Hecke field of $f$, which is a totally real number field. Write $\cO_F$ for the ring of integers of $F$ and let $\p$ be a prime ideal of $\cO_F$ above $p$. Denote by $F_\p$ the completion of $F$ at $\p$ and by $\cO_\p$ its valuation ring, which is equal to the completion of $\cO_F$ at $\p$. Let $T$ be the self-dual twist of the $\p$-adic Galois representation attached to $f$ by Deligne (\cite{Del-Bourbaki}): this is a free $\cO_\p$-module of rank $2$, equipped with an action of $\Gal(\overline{\Q}/\Q)$. Set
$A\defeq T\otimes_{\cO_\p}(F_\p/\cO_\p)$. Moreover, let $K_\infty$ be the anticyclotomic $\Z_p$-extension of $K$, let $G_\infty\defeq\Gal(K_\infty/K)$ and denote by $\Lambda\defeq\cO_\p[\![G_\infty]\!]$ the associated Iwasawa algebra. One may define Selmer groups $\Sel(K_\infty,T)$ and $\Sel(K_\infty,A)$: the former is a compact $\Lambda$-module obtained by taking inverse limits of Selmer groups over the finite layers of $K_\infty/K$ with respect to the corestriction maps in cohomology, whereas the latter is a discrete $\Lambda$-module defined by taking inverse limits of Selmer groups over the finite extensions of $K_\infty/K$ with respect to the restriction maps in cohomology. Let $\Sel(K_\infty,A)^\vee$ be the Pontryagin dual of $\Sel(K_\infty,A)$. These groups are described, from different perspectives, in \S\ref{Selmer-subsec}, \S\ref{selmerlambda} and \S\ref{selcondsec}. 

In the definite setting, following Bertolini--Darmon (\cite{BD96}, \cite{BD-IMC}) and Vatsal (\cite{Vatsal}), Chida--Hsieh studied interpolation properties of $p$-adic $L$-functions $\mathcal{L}_\mathfrak{p}(f)$ constructed from theta elements (\cite{ChHs2}, \cite{ChHs1}); see \S\ref{L-funct} and \S\ref{definite-subsec} for details. 

\begin{itheorem}[Definite IMC] \label{A-thm}
Suppose we are in the definite case and that Assumption $\ref{introass}$ holds. Then $\Sel(K_\infty,A)^\vee$ is a torsion $\Lambda$-module whose characteristic ideal coincides with the ideal generated by $\mathcal{L}_\p(f)$. 
\end{itheorem}

This result corresponds to Theorem \ref{IMC-def}. As noted before, partial results in this direction were obtained in \cite{ChHs2} for one divisibility and in \cite{CKL} for the other divisibility, using a different approach. 

In the indefinite case, taking inverse limits (with respect to corestrictions) of (regularized) Heegner cycles of $p$-power conductor leads to the definition of an element $\kappa_\infty\in\Sel(K_\infty,T)$ (see \S\ref{reg-heeg}, \eqref{Heegnerclass}, \S\ref{indefinite-subsec}).     

\begin{itheorem}[Indefinite IMC] \label{ind-IMC-intro-thm}
Suppose we are in the indefinite case and that Assumption $\ref{introass}$ holds. Moreover, assume that $\kappa_\infty$ is non-torsion. Then $\Sel(K_\infty,A)^\vee$ is a $\Lambda$-module of rank $1$ and the characteristic ideal generated by its torsion $\Lambda$-submodule coincides with the characteristic ideal of $\Sel(K_\infty,T)/\Lambda\cdot\kappa_\infty$.   
\end{itheorem}

This result corresponds to Theorem \ref{IMC}; in line with terminology found in current literature for elliptic curves, we can refer to Theorem \ref{ind-IMC-intro-thm} as a \emph{Perrin-Riou main conjecture} for modular forms. We remark that the non-torsioness of $\kappa_\infty$ is known when $p$ splits in $K$ (\cite{Burungale}, see Proposition \ref{non zero}), so it can be removed from Theorem \ref{ind-IMC-intro-thm} in this case (as is done in Theorem \ref{IMC}). As mentioned earlier, partial results towards this conjecture were obtained (for one divisibility only) in \cite{LV-kyoto}. 

Finally, one can rephrase the indefinite IMC via the approach to $p$-adic $L$-functions due to Bertolini--Darmon--Prasanna. In what follows, we assume that $p$ splits in $K$ and write $p\cO_K=\p\bar{\p}$; unlike what happens with Theorem \ref{ind-IMC-intro-thm}, this splitting condition is essential for the definition of the objects involved, not only 
to ensure the non-torsioness of $\kappa_\infty$. Now set $\Lambda^\unr\defeq\Lambda\otimes_{\cO_\p}\cO_\p^\unr$, where $\cO_\p^\unr$ is the maximal unramified extension of $\cO_\p$; following \cite{BDP}, \cite{Brooks}, \cite{CH} and \cite{Magrone}, we may introduce a BDP $p$-adic $L$-function $\mathcal{L}_\p^\mathrm{BDP}\in\Lambda^\unr$ interpolating twists of $f$ by certain anticyclotomic Hecke characters (see \S\ref{BDP function}). This object is related to generalized Heegner cycles, which in turn can be explicitly linked to the class $\kappa_\infty$. Thus, a relation is established between the $\Lambda$-module $\Sel(K_\infty,T)/\Lambda\cdot\kappa_\infty$ appearing in the indefinite IMC and the BDP $p$-adic $L$-function $\mathcal{L}_\p^\mathrm{BDP}$. On the algebraic side, the $\Lambda$-torsion submodule of the Pontryagin dual of $\Sel(K_\infty,T)$ can be related to a 
\emph{modified} Selmer group $\Sel_{\emptyset,0}(K_\infty,T)$, which is a cotorsion 
$\Lambda$-module defined by imposing no local condition at $\p$ and the trivial local condition at $\bar{\p}$ (see \S\ref{modifiedsection}, where $\Sel_{\emptyset,0}(K_\infty,T)$ is denoted by $\Sel_{\emptyset,0}(K,\mathbf{T}_f)$).

The indefinite Iwasawa main conjecture can then be equivalently stated as follows.

\begin{itheorem}[Indefinite IMC] \label{C-thm}
Suppose we are in the indefinite case and that Assumption \ref{introass} holds. Moreover, assume that $p$ splits in $K$. Then $\Sel_{\emptyset,0}(K_\infty,T)$ is $\Lambda$-torsion and the ideal of $\Lambda^\unr$ generated by its characteristic power series coincides with the ideal of $\Lambda^\unr$ generated by $\mathcal{L}_\p^{\mathrm{BDP}}$. 
\end{itheorem}

This result corresponds to Theorem \ref{IG-thm}. Similar alternative versions of the Perrin-Riou conjecture were studied by Kobayashi (\cite{Kobayashi-Que}; \emph{cf.} also \cite{KO2}, \cite{KO1}) and Castella (\cite{Castella-London}) in the case of elliptic curves. 

Let $\overline{T}=T\pmod{\p}$ denote the residual mod $\p$ representation associated with $T$ and set $p^*\defeq(-1)^{\frac{p-1}{2}}p$; moreover, let $c_f$ stand for the index in $\cO_F$ of the order generated by the Fourier coefficients of $f$. We conclude this introduction by remarking that our main results are established under the following list of technical conditions (\emph{cf.} Assumption \ref{ass1}).

\begin{assumption} \label{introass}  
\begin{enumerate}
\item $p\nmid ND_Kc_f$;
\item $T$ is ordinary at $p$ (\emph{cf.} \eqref{ordinaryfil}); 
\item $p>k+1$ and $\#(\mathbb{F}_p^\times)^{k-1}>5$;  
\item $p$ is non-anomalous (\emph{cf.} Definition \ref{anomalous});
\item the restriction of $\overline{T}$ to $\Gal\bigl(\bar\Q/\Q(\sqrt{p^*})\bigr)$ is absolutely irreducible;
\item $\overline{T}$ is ramified at all primes $q\,|\,M$ with $q\equiv 1\pmod{p}$;
\item $\overline{T}$ is ramified at all primes $q\,|\,D$ with $q\equiv\pm 1\pmod{p}$; 
\item there is a prime $q\,|\,N$ such that $\overline{T}$ is ramified at $q$;
\item the $p$-adic Galois representation attached to $f$ has big image (\emph{cf.} Section \ref{Iwasawa}).
\end{enumerate}
\end{assumption}

In principle, these assumptions could be relaxed; however, some of the results we use (especially those from \cite{wang}, which builds on \cite{ChHs2} and \cite{ChHs1}) are stated under conditions (2)--(8) (called ($\mathrm{CR}^\star$) in \cite[Assumption 1]{wang}), which is why we require them here. As for (1) and (9), the non-divisibility of $c_f$ by $p$ and the big image property are assumed in \cite{LPV}, some of whose results we exploit in this article, and, as a consequence, are imposed here too. 

\subsection*{Notation and conventions.} 

Fix throughout algebraic closures $\overline\Q$ and $\overline\Q_p$ and embeddings $\overline\Q\hookrightarrow \C$ and $\overline{\Q}\hookrightarrow\overline\Q_p$. 

Let $\widehat{\Z}$ denote the profinite completion of $\Z$ and, for any ring $A$, set $\widehat{A}\defeq A\otimes_\Z\widehat{\Z}$. For any $\Q$-algebra $A$, put $A_\infty\defeq A\otimes_\Q\R$ and, for any ring $A$ and any prime number $\ell$, set $A_\ell\defeq A\otimes_\Z\Z_\ell$. 

If $L$ is a number field, then we denote by $\cO_L$ the ring of integers of $L$ and for each place $v$ of $L$ we write $L_v$ for the completion of $L$ at $v$ and $\cO_{L_v}$ for the valuation ring of $L_v$; a similar convention applies to any local field. For a field $L$, we let $G_L\defeq\Gal(\overline{L}/L)$ be its absolute Galois group; if $L$ is a number field and $v$ a finite place of $L$, then $G_v\subset G_L$ is the decomposition group at $v$ and $I_v\subset G_v$ the inertia subgroup. 

%Let $\mathbb{A}$ denote the adele ring of $\Q$; we denote $\mathbb A_f$ or $\mathbb{A}^{(\infty)}$ the ring of finite adeles in $\mathbb{A}$ and $\mathbb{A}_\infty=\R$ the archimedean component of $\mathbb{A}$; for any $x\in\mathbb{A}$ we let $x_f\in\mathbb{A}^{(\infty)}$ and $x_\infty\in \mathbb{A}_\infty$ the projections to the finite and infinite summands, respectively; we write $a$ for the image of $a\in \Q$ in $\mathbb{A}$ via the diagonal embedding, and we still write $a$ for its finite and archimedean components $a_f$ or $a_\infty$, respectively. For a $\Q$-algebra $A$, we let $A_\mathbb A=A\otimes_\Q\mathbb A$, $A_\infty=A\otimes_\Q\R$, and $\widehat{A}=A\otimes_\Q\mathbb{A}^{(\infty)}$. We also denote $\widehat{\Z}$ the profinite completion of $\Z$ and for each ring $A$ we let $\widehat{A}=A\otimes_\Z\widehat{\Z}$. 

Given a variety $V$ defined over a number field $L$, we denote by $\CH^{i}(V/L)$ the abelian group of cycles of codimension $i$ in $V$, rational over $L$, with $\Z$-coefficients, modulo rational equivalence; we also write $\CH_0^{i}(V/L)$ for the submodule of homologically trivial cycles, \emph{i.e.}, the kernel of the cycle class map
\[ c\ell: \CH^{i}(V/L)\longrightarrow H^{2i}_\text{\'et}\bigl(\overline{V},\Q_q(i)\bigr)^{G_L}, \] 
where $\overline{V}\defeq V\otimes_L\overline{L}$ and $q$ is a prime number (this definition does not depend on $q$). Finally, for a ring $A$, let $\CH^{i}(V/L{)}_A\defeq\CH^{i}(V/L)\otimes_\Z A$ and $\CH^{i}_0(V/L{)}_A\defeq\CH^{i}_0(V/L)\otimes_\Z A$. 

\section{Preliminaries on quaternion algebras} \label{sec:preliminaries}

We collect some general background on quaternion algebras over $\Q$, in both the definite and the indefinite settings. 

\subsection{Quaternion algebras} \label{sec:DQA}

Let $M$ and $D$ be square-free coprime integers, and assume that $D$ is square-free. Let $B$ denote the quaternion algebra of discriminant $D$. Then $B$ is definite if $D$ is the product of an odd number of distinct primes, and indefinite otherwise. 
Fix an imaginary quadratic field $K$ of discriminant $-D_K$ such that all primes dividing $D$ are inert in $K$, and all primes dividing $M$ are split in $K$. We also fix a prime number $p\nmid MDD_K$. Fix an embedding of $\Q$-algebras $\psi_K:K\hookrightarrow B$ which allows us to identify $K$ with a subalgebra of $B$ writing $K\subset{B}$ instead of $\psi_K(K)\subset B$. Using the Skolen--Noether 
theorem as in \cite[\S2.2]{ChHs2}, we can choose $j\in B^\times$ satisfying the following conditions:  
\begin{itemize}
\item $B=K\oplus K j$ as $\Q$-vector spaces; 
\item $j^2=\beta$ with $\beta\in\Q^\times$ and $\beta<0$; 
\item $jt=\bar{t}j$ for all $t\in K$, where $t\mapsto \bar{t}$ is the main involution of $B$; 
\item $\beta\in (\Z_q^\times)^2$ for all $q\mid Mp$;
\item $\beta\in \Z_q^\times$ for all $q\mid D_K$.
\end{itemize}
Fix such a $j$ once and for all. Set $\delta_K\defeq\sqrt{-D_K}$ and $D'_K=D_K$ if $D_K$ is odd, $D'_K=D_K/2$ if $D_K$ is even. Define 
$\theta=\frac{D_K'+\delta_K}{2}$. Define the embedding $i_K:B\hookrightarrow \M_2(K)$ by $a+bj=\smallmat {a}{b\beta}{\bar b}{\bar a}$.
For each prime $q$ of $\Q$, let $B_q=B\otimes_\Q\Q_q$ and $R_q=R\otimes_\Z\Z_q$. For each prime $q\nmid D$ we fix isomorphisms
$i_q:B_q\simeq \M_2(\Q_q)$ as follows: 
\begin{itemize}
\item If $q\nmid Mp$, we require that $i_q(\mathcal{O}_K\otimes_\Z\Z_q)\subset \M_2(\Z_q)$; 
\item If $q\mid Mp$, we define $i_q$ by $i_q(\theta)=\smallmat {\theta+\bar\theta}{-\theta\bar\theta}{1}{0}$, 
$i_q(j)=\sqrt{\beta}\smallmat{-1}{\theta+\bar\theta}{0}{1}$. 
\end{itemize}
Let $R\subset B$ be the Eichler order of level $M$ with respect to the isomorphisms $i_q$. 
We denote $\widehat{R}=R\otimes_\Z\widehat{\Z}$ and $\widehat{B}=B\otimes_\Z\widehat{\Z}$, where $\widehat{Z}$ is the profinite completion of $\Z$. We denote $B_\infty=B\otimes_\Q\R$, which is either isomorphic to $\M_2(\R)$ (we say that $B$ is \emph{indefinite} in this case) or isomorphic to the Hamilton quaternions (we say that $B$ is
\emph{definite} in this case). We put $B_\mathbb{A}=\widehat{B}\times B_\infty$ 
and denote $x\mapsto x_f$ and $x\mapsto x_\infty$ the two projections $B_\mathbb{A}\rightarrow \widehat{B}$ and 
$B_\mathbb{A}\rightarrow B_\infty$. 

\subsection{Special elements}\label{sec special elements} 

The \emph{Atkin--Lehner involution} of $B$ of level $M$ is defined to be the element $\tau=(\tau_q)_q\in\widehat{B}^\times$ such that $\tau_q=1$ if $q\nmid MD$, $\tau_q=j$ if $q\mid D$ and $\tau_q=\smallmat 01{-M}0$ if $q\mid M$. We complete this definition with the archimedean place and define an element $\tau\in B_\mathbb{A}^\times$ by requiring $\tau_\infty=j$. The reader is referred, \emph{e.g.}, to \cite[Section 2]{Chida}, \cite[\S3.3]{ChHs1} and \cite[(2.5), (3.3)]{wang} for details. Now we introduce special elements $g_n\in\widehat{B}^\times$ that will be used to define Heegner (respectively, Gross) points on indefinite (respectively, definite) Shimura curves. These points are defined in many places (see, \emph{e.g.}, \cite{CKL}, \cite{CL}, \cite{ChHs2}, \cite{ChHs1}, \cite{LV-MM}) with different normalizations: here we follow the one in \cite[Section 2]{Chida} and \cite[(3.2)]{wang}. Let us write $M\cO_K=\mathfrak{M}\cdot\overline{\mathfrak{M}}$ and for every prime number $q\neq p$ 
define $\xi_q\in B_q^\times$ as follows:
\begin{itemize}
\item if $q\nmid Mp$, then $\xi_q\defeq1$; 
\item if $q\,|\,M$ and $q=\mathfrak{q}\cdot\bar{\mathfrak{q}}$ with $\mathfrak{q}\,|\,\mathfrak{M}$, then $\xi_q$ satisfies $i_q(\xi_q)=\delta_K^{-1}\smallmat \theta{\bar\theta}11$ (observe that $\smallmat\theta{\bar\theta}11\in \GL_2(K_\mathfrak{q})\simeq\GL_2(\Q_q)$ and recall that $\delta_K\in\Z_q^\times$).
\end{itemize}
Furthermore, for every integer $n\geq0$ we introduce $\xi_p^{(n)}\in B_p^\times$ as follows: 
\begin{itemize}
\item if $p$ splits in $K$, then $\xi_p^{(n)}$ satisfies $i_p(\xi_p^{(n)})=\smallmat{1}{p^{-n}}{0}{1}$;
\item if $p$ is inert in $K$, then $\xi_p^{(n)}$ satisfies $i_p(\xi_p^{(n)})=\smallmat{p^n}{0}{0}{1}$. 
\end{itemize}
Finally, set $g_n\defeq\xi_p^{(n)}\cdot\prod_{q\neq p}\xi_q\in \widehat{B}^\times$.

\subsection{Polynomial representations}

Let $k\geq 4$ be an even integer and $A$ be a ring. Put $r=k-2$ and 
$L_r(A)=\Sym^r(A).$
Let $\mathbf{v}_j=X^{r/2-j}Y^{r/2+j}$ for $-r/2\leq j\leq r/2$ be the standard basis elements of $L_r(A)$.
Denote by $\rho_k:\GL_2(A)\rightarrow\Aut_A(L_r(A))$ the representation defined by $\rho_k(g)(P)=\det^{-(r/2)}(g)(P|g)$ where if $P=P(X,Y)$ is a polynomial, then $P|g$ is the polynomial $P|g(X,Y)=P((X,Y)g)$, with $(X,Y)g$ indicating the left matrix multiplication. 
If $p>k-2$ is a prime number and $A$ a $\Z_{(p)}$-algebra, we may define a perfect pairing $\langle\cdot,\cdot\rangle_k:L_r(A)\times L_r(A)\rightarrow A$ by setting
\begin{equation}\label{pairing}
\bigg\langle \sum_{i=1}^ra_i\mathbf{v}_i,\sum_{j=1}^rb_j\mathbf{v}_j\bigg\rangle_{\!k}\defeq\sum_{-\frac{k-2}{2}\leq n\leq \frac{k-2}{2}}a_nb_{-n}(-1)^{r/2+n}\frac{\Gamma(k/2+n)\Gamma(k/2-n)}{\Gamma(k-1)}. \end{equation}
This pairing satisfies the rule 
$\langle\rho_k(g)P_1,\rho_k(g)P_2\rangle_k=\langle P_1,P_2\rangle_k$. 

Composing the embedding $i_K:B\hookrightarrow\M_2(K)$ with the fixed embedding $K\hookrightarrow\C$ 
gives an embbedding $B\hookrightarrow\M_2(\C)$; further composing with  
$\rho_k:\GL_2(\C)\rightarrow \Aut_\C(L_r(\C))$ we obtain a representation $\rho_{k,\infty}:B^\times\rightarrow\Aut_\C(L_r(\C))$. We still denote 
$\rho_{k,\infty}:B^\times_\infty\rightarrow \Aut_\C(L_r(\C))$ the representation obtained by scalar extension, where $B_\infty=B\otimes_\Q\R$. 
We also remark that $\C\cdot \mathbf{v}_m$ is the eigenspace where $\rho_{k,\infty}(t)$ acts via $(t/\bar{t})^m$, for all $t\in (K\otimes_\Q\R)^\times\simeq\C^\times$.

Let $q\nmid D$ be a prime number and let $\mathfrak{q}$ be the prime of $K$ above $q$ corresponding to a fixed embedding $\bar\Q\hookrightarrow\bar\Q_q$.
We may then consider the embedding $i_{K_\mathfrak{q}}:B\hookrightarrow \M_2(K_\mathfrak{q})$ obtained composing 
$i_K:B\hookrightarrow\M_2(K)$ with the embedding $K\hookrightarrow K_\mathfrak{q}$; further composing with  
$\rho_k:\GL_2(K_\mathfrak{q})\rightarrow \Aut_\C(L_r(K_\mathfrak{q}))$ we obtain a representation $\rho_{k,\mathfrak{q}}:B^\times\rightarrow\Aut_\C(L_r(K_\mathfrak{q}))$. We still denote 
$\rho_{k,\mathfrak{q}}:B^\times_q\rightarrow \Aut_\C(L_r(K_\mathfrak{q}))$ the representation obtained by scalar extension.   
If $q\mid Mp$ and $\gamma_\mathfrak{q}=\smallmat {\sqrt{\beta}}{-\sqrt{\beta}\bar\theta}{-1}{\theta}$ then $\rho_{k,\mathfrak{q}}(g)=\rho_k(\gamma_\mathfrak{q} i_q(g)\gamma_\mathfrak{q}^{-1})$ for all $g\in B^\times$, where $i_q$ is fixed for these primes as in \S\ref{sec:DQA}.
Moreover, $\rho_{k,\infty}(g)=\rho_{k,\mathfrak{q}}(g)$ for all $g\in B^\times$. See \cite[\S2.3, \S4.1]{ChHs1} for details. 

\section{Anticyclotomic $p$-adic $L$-functions}

Let us fix a \emph{definite} quaternion algebra $B$ over $\Q$ of discriminant $D$, an integer $M\geq1$ coprime with $D$, a prime number $p\nmid MD$ and an imaginary quadratic field $K$ together with an embedding $\psi_K:K\hookrightarrow B$; our general notation from Section \ref{sec:preliminaries} is in force.

\subsection{Shimura curves} \label{sets} 

Fix an Eichler order $R\subset B$ of level $M$. Let $\Hom(K,B)$ be the set of $\Q$-algebra embeddings $K\hookrightarrow B$. The \emph{Shimura curve} (or \emph{Gross curve}) attached to the data above is the double coset space
\[ X\defeq B^\times\big\backslash\widehat{B}^\times\times\Hom(K,B)\big/\widehat{R}^\times. \] 
We denote $X$ by $X_{M,D}$ if we need to stress dependence on $M$ and $D$. Fix, as before, an integer $c\geq1$ prime to $MD$. If a point $x\in X$ is represented by a pair $(g,\varphi)$ such that $\varphi(K)\cap g^{-1}\widehat{ R}^\times g=\varphi(\cO_{c})$, then we say that $\varphi$ is an \emph{optimal embedding} of $\cO_c$ into $R$ and that $x$ is a \emph{Gross point} of \emph{conductor $c$}. The set of Gross points of conductor $c$ is equipped with an action of $\Gal(K^\mathrm{ab}/K)$ factoring through $\Gal(H_{c}/K)$, which is defined as before: if $\sigma\in\Gal(K^\mathrm{ab}/K)$ is represented by $a\in\hat{K}^\times$ via the geometrically normalized Artin map and $x$ is represented by $(g,\varphi)$, then $x^\sigma\in\mathcal{S}_{M,D}(K)$ is represented by the pair $\bigl(\hat{\varphi}(a)g,\varphi\bigr)$, where the map $\hat\varphi:\hat K\rightarrow \hat{B}$ is obtained from $\varphi$ by extending scalars to $\hat{\Z}$. Then $x^\sigma$ is still a Gross point of conductor $c$. 

For any integer $n\geq0$, we let $x_n=\bigl[(g_n,\psi_K)\bigr]$ be a Gross point of conductor $p^n$. In order to lighten our notation, for $a\in \widehat{K}^\times$ and $\sigma\in\Gal(K^\mathrm{ab}/K)$ corresponding to $a$ under the Artin map, we abbreviate $x_n^\sigma=\bigl[\bigl(\widehat{\psi}_K(a)g_n,\psi_K\bigr)\bigl]$ as $x_n(a)=\bigl[(ag_n,\psi_K)\bigr]$. 

\subsection{Modular forms}\label{mod}

Let $U$ be a compact open subgroup of $\widehat{B}^\times$; in our arguments, $U$ will be either $\widehat{R}^\times$ or $\widehat{R(p)}^\times$, where $R(p)\subset R$ is the standard Eichler order of $B$ of level $p$. Now fix an $\cO_{K}$-algebra $A$ and denote by $\mathbf{S}_k^B(A,U)$ the $A$-module of functions $B^\times\backslash\widehat{B}^\times/U\rightarrow L_k(A)$. 
If $U=\widehat{R}^\times$ (respectively, $U=\widehat{R(p)}^\times$), then the double coset Hecke algebra $\T_{M,D}(A)$ of level $M$ (respectively, $\T_{Mp,D}(A)$ of level $Mp$), generated by operators $T_q$ for $q\nmid MD$ and $U_q$ for $q\,|\,MD$ (respectively, $T_q$ for $q\nmid MDp$ and $U_q$ for $q\,|\,MDp$), acts on $\mathbf{S}_k^B(A,U)$. Set 
\[ \mathrm{level}(U)\defeq\begin{cases}M & \text{if $U=\widehat{R}^\times$}\\[2mm] Mp & \text{if $U=\widehat{R(p)}^\times$}\end{cases} \]
and $\Gamma_g\defeq\Bigl(B^\times\cap (gUg^{-1}\widehat{\Q}^\times)\Bigr)\big/\Q^\times$. The $A$-module $\mathbf{S}_k^B(A,U)$ is equipped with a Petersson product 
\[ {\langle\cdot,\cdot\rangle}_B:\mathbf{S}_k^B(A,U)\times \mathbf{S}_k^B(A,U)\longrightarrow A \]
defined by the rule 
\[ {\langle \phi_1,\phi_2\rangle}_B\defeq\sum_{g\in B^\times\backslash\widehat{B}^\times/\widehat{R}^\times\widehat{\Q}^\times}|\Gamma_g|^{-1}\cdot\big\langle\phi_1(g),\phi_2(g\tau)\big\rangle_k, \]
where $\tau=(\tau)_q\in \widehat{B}^\times$ is the Atkin--Lehner involution of level $\mathrm{level}(U)$ (\emph{cf.} \S \ref{sec special elements}).

From here on, set 
\[ {S}_k(M,D;A)\defeq\mathbf{S}_k^B\bigl(A,\widehat{R}^\times\bigr),\quad{S}_k(Mp,D;A)\defeq\mathbf{S}_k^B\Bigl(A,\widehat{R(p)}^\times\Bigr).
\]
Fix an eigenform $\phi\in{S}_k(M,D;A)$ for all Hecke operators in $\T_{M,D}(A)$; assuming that $A$ contains all the eigenvalues of $\phi$, we shall adopt the same symbol to denote the $A$-algebra homomorphism $\T_{M,D}(A)\rightarrow A$ associated with $\phi$. Moreover, we assume that $\phi$ is \emph{$p$-adically normalized}, \emph{i.e.}, $\phi$ is not $\p$-divisible, and \emph{ordinary}, \emph{i.e.}, $\phi(T_p)$ is a unit in $A$. We denote by $\alpha_p$ the root of the Hecke polynomial $X^2-\phi(T_p)X-p^{k-1}$ that is a unit of $A$. 
Finally, we write $\phi^\sharp$ for the \emph{$p$-stabilization} of $\phi$, which is defined by the formula 
\[ \phi^\sharp(x)\defeq\phi(x)-\frac{p^\frac{k-2}{2}}{\alpha_p}\cdot\phi\biggl(x\begin{pmatrix}p^{-1}&0\\0&1\end{pmatrix}\!\biggr). \]
It turns out that $\phi^\sharp$ lies in $S_k(Mp,D;A)$; furthermore, $\phi^\sharp$ is an eigenform for all the Hecke operators in $\T_{Mp,D}(A)$ that has the same eigenvalues as $\phi$ for all operators different from $U_p$ and satisfies $U_p(\phi^\sharp)=\alpha_p\phi^\sharp$. 

\subsection{Theta elements} \label{theta}

Set $\mathcal{G}_n\defeq\Gal(H_{p^n}/K)$ and $\mathbf{v}_0^*\defeq D_K^\frac{k-2}{2}(XY)^{\frac{k-2}{2}}$. Moreover, for all integers $n\geq0$ put $\mathbf{w}_n\defeq p^\frac{n(k-2)}{2}\mathbf{v}_0^*$. The \emph{theta elements} attached to $\phi$ are 
\[ \tilde\theta_{n}(\phi)=\alpha_p^{-n}\cdot\sum_{\sigma\in\mathcal{G}_{n}}\big\langle
\rho_k^{-1}(\gamma_\p)(\mathbf{w}_n),\phi^\sharp\bigl(x_n(a)\bigr)\big\rangle_k\cdot\sigma\in A[\mathcal{G}_n]. \]
Let $\pi_n:A[\mathcal{G}_{n}]\rightarrow A[G_{n-1}]$ be the projector induced by the splitting $\mathcal{G}_n\simeq G_{n-1}\times\Delta$, where $\Delta=\Gal(H_1/K)$ has order coprime with $p$ by assumption, then define 
\[ \theta_{n-1}(\phi)\defeq\pi_n\bigl(\tilde\theta_{n}(\phi)\bigr)\in A[G_{n-1}] \] 
for all integers $n\geq 1$. The projected theta elements $\theta_n(\phi)$ thus introduced are norm-compatible and we may consider the Iwasawa element 
\begin{equation} \label{theta-infty-eq} 
\theta_\infty(\phi)\defeq \invlim_{n\geq 0}\theta_{n}(\phi)\in A[\![G_\infty]\!],
\end{equation}
which is called the \emph{square-root $p$-adic $L$-function of $\phi$}. When we want to stress dependence on $M$ and $D$, we write $\theta_{M,D,n}(\phi)$ and $\theta_{M,D,\infty}(\phi)$ for $\theta_n(\phi)$ and $\theta_\infty(\phi)$, respectively.

\begin{remark} 
Here we compare the definition of theta elements given above with the one in \cite{ChHs1} when $A=\C$. The reader is advised to keep a copy of \cite{ChHs1} at hand, bearing the following notational differences in mind: the symbol $A_p$ in \cite{ChHs1} corresponds to our $\alpha_p$, while $\alpha_p$ in \cite{ChHs1} corresponds to our $\alpha_pp^{-\frac{k-2}{2}}$. We write $\tilde\theta_n^\mathrm{CH}(\phi)$ for the theta elements in \cite[\S4.2]{ChHs1}: see, especially, \cite[Definition 4.1]{ChHs1}, where $\tilde\theta_n^\mathrm{CH}(\phi)$ is denoted by $\Theta_n^{[0]}\bigl(f_{\pi'}^\dagger\bigr)$. In \cite{ChHs1}, $p$-adic quaternionic modular forms are functions $\psi:B^\times\backslash\widehat{B}^\times\rightarrow L_k(\C_p)$ such that $\psi(gu)=\rho_k(u_p^{-1})\psi(g)$ for $u\in U$, where $U$ is a compact open subgroup of $\widehat{B}^\times$; one can check that the map $\psi\mapsto\phi_\psi$ with $\phi_\psi(g)\defeq\rho_k(g_p)\psi(g)$ establishes an isomorphism between the $\C_p$-vector space of $p$-adic quaternionic modular forms in \cite[\S4.1]{ChHs1} and the $\C_p$-vector space $\mathbf{S}_n(\C_p,U)$ in \S \ref{mod}. We want to show that $\tilde\theta_n^\mathrm{CH}(\psi)=\tilde\theta_n(\phi_\psi)$ for all $n\geq0$. Let $\psi^\sharp$ be the $p$-stabilization of $\psi$ (see, \emph{e.g.}, \cite[\S3.2]{ChHs1}). First of all, there are equalities  
\begin{equation} \label{CHtheta}
\tilde\theta_n^\mathrm{CH}(\psi)=\alpha_p^{-n}p^\frac{n(k-2)}{2}\cdot\sum_{\sigma\in\mathcal{G}_{n}}\big\langle \mathbf{v}_0^*,\psi^\sharp\bigl(x_n(a)\bigr)\big\rangle_k\cdot\sigma
=\alpha_p^{-n}\cdot\sum_{\sigma\in\mathcal{G}_{n}}\big\langle\mathbf{w}_n,\psi^\sharp\bigl(x_n(a)\bigr)\big\rangle_k\cdot\sigma.
\end{equation}
(\emph{cf.} the relevant definitions in \cite{ChHs1}). Now recall the definition of the $p$-adic avatar $\hat\psi^\sharp(g)=\rho_k^{-1}(\gamma_\p)\rho_{k,p}^{-1}(g_p)\psi^\sharp(g)$ in \cite[\S4.1]{ChHs1}, which is a $p$-adic quaternionic modular form in the sense of \cite{ChHs1}; using the relation $\rho_{k,p}(g)=\rho_k(\gamma_\p g_p\gamma_\p^{-1})$, we get 
\[ \psi^\sharp(g)=\rho_k(\gamma_\p)\rho_k(g_p)\hat\psi^\sharp(g)=\rho_k(\gamma_\p) \phi_\psi^\sharp(g). \]
Replacing in \eqref{CHtheta}, we get
\[ \tilde\theta_n^\mathrm{CH}(\psi)=\alpha_p^{-n}\cdot\sum_{\sigma\in\mathcal{G}_{n}}\big\langle\rho_k^{-1}(\gamma_\p) (\mathbf{w}_n),\phi_\psi^\sharp\bigl(x_n(a)\bigr)\big\rangle_k\cdot\sigma=\tilde\theta_n(\phi_\psi),
\]
as was to be shown. 
\end{remark}

\subsection{$p$-adic $L$-functions}\label{L-funct} 

We assume that $A$, here denoted by $\cO$, is the valuation ring of a finite extension of $\Q_p$; take a $p$-adically normalized newform $f\in S_k(\Gamma_0(MD))$ and let $\phi$ be the Jacquet--Langlands lift of $f$. The \emph{$p$-adic $L$-function} of $f$ is  
\[ \mathcal{L}(f)\defeq \theta_\infty(\phi)\cdot\theta_\infty(\phi)^*\in\cO[\![G_\infty]\!], \] 
where $\theta_\infty(\phi)$ is defined as in \eqref{theta-infty-eq} and $\lambda\rightarrow\lambda^*$ is the canonical involution of the algebra $\cO[\![G_\infty]\!]$ behaving like inversion on group-like elements. The reader is referred, \emph{e.g.}, to \cite[Theorem 2.4]{CL}, \cite[Theorem 4.6]{ChHs1} and \cite{Hung} for the interpolation properties of $\mathcal{L}(f)$. 

Now let $L_K(f,\chi,s)$ be the $L$-function of $f$ twisted by a finite order character $\chi$ of $G_\infty$. Let $\Omega_{f,D}$ be the \emph{Gross period} of $f$ (see, \emph{e.g.}, \cite[\S1]{ChHs1} and \cite[\S5.2]{KL1}); recall that $\Omega_{f,D}$ is related to the Petersson norm ${\|f\|}_{\Gamma_0(MD)}$ of $f$ by the formula  
\[ \Omega_{f,D}=\frac{(4\pi)^k\cdot{\|f\|}_{\Gamma_0(MD)}}{\langle\phi,\phi\rangle_B} \]
(\emph{cf.} \cite[(4.3)]{ChHs1}). Given a finite order character $\chi$ of $G_\infty$ factoring through $K_n$, there is an equality 
\[ \chi\bigl(\mathcal{L}(f)\bigr)=\mathrm{Const}\times\frac{L_K(f,\chi,k/2)}{\Omega_{f,D}}; \]
here the constant $\mathrm{Const}=C_nE_p^2$ is a non-zero algebraic number, while
\[ E_p\defeq\begin{cases} 1 & \text{if $n\geq1$}\\[2mm]
\biggl(1-\frac{\chi(\wp)p^\frac{k-2}{2}}{\alpha_p}\biggr)\cdot\biggl(1-\frac{\chi(\bar\wp)p^\frac{k-2}{2}}{\alpha_p}\biggr) & \text{if $n=0$ and $p$ splits in $K$ as $p\cO_K=\wp\cdot\bar{\wp}$}\\[4mm]
1-\frac{p^{k-2}}{\alpha_p^{2}} & \text{if $n=0$ and $p$ is inert in $K$}\end{cases} \]
and
\[ C_n\defeq\biggl(\frac{\#\cO^\times_K}{2}\biggr)^2\cdot\Gamma(k/2)^2\cdot\sqrt{D_K}\cdot D_K^{k-2}\cdot a_p(f)^{-2n}\cdot p^{n(k-1)}. \]
%In particular, note that for $n=1$ both $C_p$ and $E_p$ are $p$-adic units, so $\mathrm{Const}$ is a $p$-adic unit. 
In the interpolation formulas above, it is implicit that the quotient between the special value of the complex $L$-function and the Gross period is algebraic, so it may be compared via the embeddings $\iota_\infty:\bar\Q\hookrightarrow\C$ and $\iota_p:\bar\Q\hookrightarrow\bar\Q_p$ to the value of the theta series at $\chi$. 

\section{Heegner cycles in Iwasawa theory} \label{heegner-sec}

In this section, we fix an \emph{indefinite} quaternion algebra $B$ over $\Q$ of discriminant $D$, an integer $M\geq1$ coprime with $D$, a prime number $p\nmid MD$ and an imaginary quadratic field $K$ together with an embedding $\psi_K:K\hookrightarrow B$; our general notation from Section \ref{sec:preliminaries} is in force. We also set $B_\infty\defeq B\otimes_\Q\R$ and choose an $\R$-algebra isomorphism $i_\infty:B_\infty\xrightarrow\simeq\M_2(\R)$ such that $i_\infty(\theta)=\smallmat {\theta+\bar\theta}{-\theta\bar\theta}{1}{0}$, with $\theta$ as in \S \ref{sec:DQA}.

\subsection{Shimura curves} \label{curves} 

Assume that the square-free integer $D$ is the product of an \emph{even} number of primes; of course, the set of prime divisors of $D$ may be empty, in which case $D=1$ and it will not be restrictive to take $B=\M_2(\Q)$. Choose a maximal order $\cO_B$ in $B$ (with $\mathcal{O}_B=\M_2(\Z)$ when $D=1$) and an Eichler order $R\subset B$ of level $M$ (with $R=\bigl\{\smallmat abcd\in \SL_2(\Z)\mid M\,|\,c\bigr\}$ when $D=1$). We also fix an integer $d\geq 5$ coprime with $MD$. Denote by $\mathcal{H}$ the complex upper half-plane and set $\mathcal{H}^\pm\defeq\C\smallsetminus\R$; we identity $\mathcal H^\pm$ with the set $\Hom(\C,B_\infty)$ of $\R$-algebra homomorphisms $\C\hookrightarrow B_\infty$ by sending $\varphi\in\Hom(\C,B_\infty)$ to $i_\infty\bigl(\varphi(i)\bigr)(i)\in\C$. The \emph{Shimura curve} attached to these data is the double coset space 
\[ \mathrm{Sh}_R\defeq B^\times\big\backslash(\widehat{B}^\times\times\mathcal{H}^\pm)\big/\widehat{R}^\times, \]
where $B^\times$ acts by left translations on $\widehat{B}^\times$ and by fractional linear transformations via $i_\infty$ on $\mathcal{H}^\pm$, whereas $\widehat{R}^\times$ acts by right translations on $\widehat{B}^\times$ and trivially on $\mathcal{H}^\pm$. The set $\mathrm{Sh}_R$ is equipped with a canonical structure of Riemann surface, which is compact if $D>1$ and open (\emph{i.e.}, non-compact) if $D=1$. In what follows, we distinguish these two cases. 

Suppose that $D=1$. Denote by $\mathcal Y_0(M)=\mathcal Y_0(M,1)$ the open modular curve over $\Z[1/M]$, whose set of complex points is $\mathrm{Sh}_R=Y_0(M)(\C)=\Gamma_0(M)\backslash \mathcal{H}$; here $\Gamma_0(M)$ is the usual congruence subgroup of $\SL_2(\Z)$ consisting of matrices that are upper triangular modulo $M$. Furthermore, denote by $Y_0(M)$ the generic fiber of $\mathcal Y_0(M)$ and let $\mathcal X_0(M)=\mathcal X_0(M,1)$ be its canonical (Baily--Borel) compactification, still defined over $\Z[1/M]$; we write $X_0(M)$ for the generic fiber of $\mathcal X_0(M)$. It is well known that $\mathcal Y_0(M)$ coarsely represents the moduli problem that associates with each $\Z[1/M]$-scheme $S$ the set of isomorphism classes of pairs $(E,C)$ such that
\begin{itemize}
\item $E$ is an elliptic curve over $S$;
\item $C$ is a locally cyclic subgroup of $E[M]$ of order $M$. 
\end{itemize}
We also consider the moduli problem associating with each $\Z[1/Md]$-scheme $S$ the set of isomorphism classes of triples $(E,C,\nu_d)$, called \emph{test objects}, where the pair $(E,C)$ is as before and 
\begin{itemize}
\item $\nu_d:(\Z/d\Z)^2\overset\simeq\longrightarrow E[d]$ is an isomorphism of group schemes.
\end{itemize} 
Then this moduli problem is representable by a $\Z[1/Md]$-scheme $\mathcal Y_d(M)=\mathcal Y_d(M,1)$, which is smooth and admits a canonical smooth (Baily--Borel) compactification $\mathcal X_d(M)=\mathcal X_d(M,1)$; as before, write $Y_d(M,1)=Y_d(M)$ and $X_d(M,1)=X_d(M)$ for the generic fibers of $\mathcal Y_d(M)$ and $\mathcal X_d(M)$, respectively. We denote by $\pi_d(M):\boldsymbol{\mathcal{E}}_d(M)\rightarrow \mathcal Y_d(M)$ the universal object (omitting the explicit mention of the other structures) and, similarly, by $\pi_d(M):\boldsymbol{\overline{\mathcal{E}}}_d(M)\rightarrow\mathcal X_d(M)$ the universal generalized elliptic curve. Sometimes we also use the full notation
$\boldsymbol{\mathcal{E}}_d(M,1)=\boldsymbol{\mathcal{E}}_d(M)$, $\boldsymbol{\overline{\mathcal{E}}}_d(M,1)=\boldsymbol{\overline{\mathcal{E}}}_d(M)$, 
and analogously for $\pi_d(M,D)=\pi_d(M)$. Finally, 
$\pi_d(M):\mathcal{E}_d(M)\rightarrow  Y_d(M)$ and $\pi_d(M):\overline{\mathcal{E}}_d(M)\rightarrow  X_d(M)$ denote the generic fibers of 
$\boldsymbol{\mathcal{E}}_d(M)$ and $\boldsymbol{\overline{\mathcal{E}}}_d(M)$, respectively. 

When $D>1$, write $\mathcal X_0(M,D)$ for the compact modular curve over $\Z[1/M]$ whose complex points are $\mathrm{Sh}_R=X_0(M,D)(\C)=R_1^\times\backslash \mathcal{H}$, where $R_1^\times$ is the subgroup of $R^\times$ consisting of elements of norm $1$; the generic fiber of $\mathcal X_0(M,D)$ will be denoted by $X_0(M,D)$. Then $\mathcal X_0(M,D)$ coarsely represents the moduli problem attaching to each $\Z[1/MD]$-scheme $S$ the set of isomorphism classes of triples $(A,\iota,C)$ such that
\begin{itemize}
\item $A$ is a polarized abelian surface over $S$,
\item $\iota:\mathcal{O}_B\hookrightarrow \End_S(A)$ is an embedding,
\item $C$ is a $\mathcal{O}_\mathcal{B}$-stable locally cyclic subgroup of $A[M]$ of order $M^2$. 
\end{itemize} 
We also consider the moduli problem that associates with each $\Z[1/MDd]$-scheme $S$ the set of isomorphism classes of $4$-tuples $T=(A,\iota,C,\nu_d)$ (called again \emph{test objects}) where the triple $(A,\iota,C)$ is as before and 
\begin{itemize}
\item $\nu_d:(\cO_\mathcal{B}/d\cO_\mathcal{B}{)}_S\rightarrow A[d]$ is an isomorphism of $\mathcal{O}_\mathcal{B}$-stable group schemes.
\end{itemize}
This moduli problem is representable by a smooth $\Z[1/MDd]$-scheme $\mathcal X_d(M,D)$ whose generic fiber is denoted 
$X_d(M,D)$. As before, the universal object will be denoted by $\pi_d(M,D):\boldsymbol{\mathcal{A}}_d(M,D)\rightarrow \mathcal X_d(M,D)$ (omitting the other structures); finally, the generic fiber of $\boldsymbol{\mathcal{A}}_d(M,D)$ will be $\pi_d(M,D):\mathcal{A}_d(M,D)\rightarrow X_d(M,D)$. 

\subsection{Bad reduction of Shimura curves}\label{sec4.2}

Let $\mathcal{X}=\mathcal{X}_0(M\ell,D)$ or $\mathcal{X}=\mathcal{X}_d(M\ell,D)$ where $\ell\nmid MD$ is a prime, and $D\geq 1$ is an integer which is a product of an even number of distinct primes; we shall define an integral (non-smooth) model $\mathfrak{X}$ over $\Z_{\ell^2}$ of the $\Q_{\ell^2}$-scheme $\mathcal{X}\otimes\Q_{\ell^2}$ (where the tensor product is over $\Z[1/MDd\ell]$, $\Q_{\ell^2}$ is the quadratic unramified extension of $\Q_\ell$ and $\Z_{\ell^2}$ its valuation ring). For this, consider the moduli problem which associates to each $\Z_\ell$-scheme $S$ the set of isomorphism classes of \emph{refined} test objects $(T,\mu_\ell)$ where $T$ is a text object as before (a pair $(A,C)$ or a triple $(A,C,\nu_d)$ if $D=1$ and $\mathcal{X}=\mathcal{X}_0(M\ell,D)$ or $\mathcal{X}=\mathcal{X}_d(M\ell,D)$, respectively; a triple $(A,\iota,C)$ or a quadruplet $(A,\iota,C,\nu_d)$ if $D>1$ and $\mathcal{X}=\mathcal{X}_0(M\ell,D)$ or $\mathcal{X}=\mathcal{X}_d(M\ell, D)$, respectively) and 
\begin{itemize}
\item $\mu_\ell:A\rightarrow A'$ is an isogeny of degree $\ell$ that commutes with the action of $\mathcal{O}_B$.
\end{itemize} 
When $\mathcal{X}=\mathcal{X}_d(M\ell,D)$, 
this moduli problem is representable over $\Z_{\ell^2}$ by an affine scheme $\mathfrak{Y}_{d}(M\ell,D)$ if $B=\M_2(\Q)$ and 
by a projective scheme $\mathfrak{X}_{d}(M\ell,D)$ otherwise; in the former case we denote $\mathfrak{X}_d(M\ell,D)$ the canonical (Baily--Borel) compactification of $\mathfrak{Y}_d(M\ell,D)$.
The base change of 
$\mathfrak{X}_{d}(M\ell,D)$ to $\Q_\ell$ is canonically isomorphic to $\mathcal{X}_{d}(M\ell,D)\otimes\Q_\ell$ (the isomorphism is defined by the obvious forgetful map on the moduli space, and the tensor product is over $\Z[1/MDd\ell]$). The corrisponding universal object will be denoted by $\pi_d(M\ell,D):\mathfrak{A}_d(M\ell,D)\rightarrow \mathfrak{X}_d(M\ell,D)$, with the understanding that if $D=1$, then it is the generalized elliptic curve equipped with the extra structure at $\ell$ described before. When $\mathcal{X}=\mathcal{X}_0(M\ell,D)$, the moduli problem is coarsely representable over $\Z_{\ell^2}$ by an affine scheme $\mathfrak{Y}_0(M\ell,D)$ if $B=\M_2(\Q)$ and 
by a projective scheme $\mathfrak{X}_0(M\ell,D)$ otherwise; in the former case, we let $\mathfrak{X}_0(M\ell,D)$ be the canonical (Baily--Borel) compactification of $\mathfrak{Y}_d(M\ell,D)$. The base change of $\mathfrak{X}_0(M\ell,D)$ from $\Z[1/MDd\ell]$ to $\Q_{\ell^2}$ is canonically isomorphic to $\mathcal{X}_0(M\ell,D)\otimes\Q_{\ell^2}$, an isomorphism being induced by the obvious forgetful map on the moduli space. 

Fix a prime number $\ell\nmid MD$, then consider the $\Z_{\ell^2}$-schemes $\mathfrak{X}\defeq\mathcal{X}_0(M,D)\otimes\Z_{\ell^2}$ and $\mathfrak{X}_\ell\defeq\mathfrak{X}_0(M\ell,D)$, where the base change is taken over $\Z[1/MD]$. Let $\mathbf{X}\defeq\mathfrak{X}\otimes_{\Z_{\ell^2}}\F_{\ell^2}$ 
and $\mathbf{X}_\ell\defeq\mathfrak{X}_\ell\otimes_{\Z_{\ell^2}}\F_{\ell^2}$ be the special fibers of $\mathfrak{X}$ and $\mathfrak{X}_\ell$, respectively, where $\F_{\ell^2}$ is the field with $\ell^2$ elements. Then a theorem of Buzzard in the quaternionic case (\cite[Theorem 4.7]{Buzzard}), which extends a result of Deligne--Rapoport in the modular curve case (\cite{DR}), shows that $\mathbf{X}_\ell$ consists of two irreducible components, each isomorphic to $\mathbf{X}$, intersecting transversally at supersingular points of $\mathbf{X}$. Set $\overline{\mathbf{X}}\defeq\mathbf{X}\otimes_{\F_{\ell^2}}\overline{\F}_{\ell}$ and $\overline{\mathbf{X}}_\ell\defeq\mathbf{X}_\ell\otimes_{\F_{\ell^2}}\overline{\F}_{\ell}$. 
Adopting the notation in \cite[p. 2309]{wang}, which is used also in \cite{Liu}, we write $\overline{\mathbf{X}}^{(0)}_\ell$ for the disjoint union of connected components $C$ of $\overline{\mathbf{X}}_\ell$ and $\overline{\mathbf{X}}^{(1)}_\ell$ for the disjoint union of the intersection of two distinct connected components $C_1\cap C_2$ of $\overline{\mathbf{X}}_\ell$; let also $a_0:\overline{\mathbf{X}}^{(0)}_\ell\rightarrow \overline{\mathbf{X}}_\ell$ and $a_1:\overline{\mathbf{X}}^{(1)}_\ell\rightarrow \overline{\mathbf{X}}_\ell$
denote the canonical maps. 

Fix a prime number $\ell\nmid MD$ and let $D$ be a product of an odd number of distinct primes, so that $D\ell$ is a product of an even number of distinct primes (notice that the quaternion algebra of discriminant $D\ell$ is not a matrix algebra); again, we shall define an integral (non-smooth) model $\mathfrak{X}_d(M,D\ell)$ over the unramified quadratic extension $\Z_{\ell^2}$ of $\Z_\ell$ of the $\Q_\ell$-scheme $\mathcal{X}_d(M,D\ell)\otimes\Q_{\ell}$, where the base change is taken over $\Z[1/MDd\ell]$. To do this, consider the moduli problem associating with each $\Z_{\ell^2}$-scheme $S$ the set of isomorphism classes of \emph{refined} test objects $(T,\iota)$, where $T$ is now a triple $(A,C,\nu_d)$ as before and 
\begin{itemize}
\item $\iota:\cO_B\hookrightarrow \End_S(A)$ is a special embedding (we refer, \emph{e.g.}, to \cite[Ch. III, \S 3.1, D\'efinition]{BC} for a discussion of special quaternionic actions in the context of Drinfeld theory). 
\end{itemize} 
This moduli problem is representable over $\Z_{\ell^2}$ by a projective scheme $\mathfrak{X}_{d}(M,D\ell)$ whose generic fiber is isomorphic to $\mathcal{X}_d(M,D\ell)\otimes\Q_{\ell^2}$, where $\Q_{\ell^2}$ is the unramified quadratic extension of $\Q_\ell$ and the base change is taken over $\Z[1/MDd\ell]$ (the isomorphism is defined by forgetting that $\iota$ is special). Write $\pi_d(M\ell,D):\mathfrak{A}_d(M,D\ell)\rightarrow \mathfrak{X}_d(M,D\ell)$ for the corresponding universal object. Let $\mathbf{X}_d(M,D\ell)$ denote the special fiber of $\mathfrak{X}_d(M,D\ell)$. It is well known that the $\F_{\ell^2}$-scheme $\mathbf{X}_d(M,D\ell)$ is the (disjoint) union of two copies of a suitable $\mathds{P}^1$-bundle over $B_{D\ell}^\times\backslash\widehat{B}_{D\ell}^\times/\widehat{R}_M^\times$, where $B_{D\ell}$ is the (definite) quaternion algebra of discriminant $D\ell$ and $R_M$ is an Eichler order of $B_{D\ell}$ of level $M$. 

\subsection{Kuga--Sato varieties}

Now let $\mathscr{X}_d(M,D)$ denote one of the Shimura curves that are (compactifications of) the fine moduli spaces that we introduced above; furthermore, let $\mathscr{A}_d(M,D)\rightarrow \mathscr{X}_d(M,D)$ be the corresponding universal object. More precisely, this general notation refers to the following cases: 
\begin{itemize}
\item $\mathcal{A}_d(M,D)\rightarrow {X}_{d}(M,D)\rightarrow\Spec(\Q)$;
\item $\boldsymbol{\mathcal{A}}_d(M,D)\rightarrow\mathcal{X}_{d}(M,D)\rightarrow\Spec\bigl(\Z[1/MDd]\bigr)$;
\item $\mathfrak{A}_d(M,D)\rightarrow \mathfrak{X}_{d}(M,D)\rightarrow\Spec(\Z_{\ell^2})$ for $\ell\,|\,MD$.
\end{itemize}
For any integer $k\geq 4$, we denote by
$\pi_{k,d}:\mathscr W_{k,d}(M,D)\rightarrow \mathscr{X}_d(M,D)$ the Kuga--Sato variety of dimension $k-1$ over $\mathscr{X}_d$ defined as
\begin{itemize}
\item the desingularization of the $(k-2)$-fold product of $\mathscr{A}_d$ over $\mathscr{X}_d$ if $D=1$; 
\item the $(k-2)/2$-fold product of $\mathscr{A}_d(M,D)$ over $\mathscr{X}_d(M,D)$ if $D>1$.
\end{itemize}
Therefore, for $\mathscr{X}_d(M,D)=X_d(M,D)$ we get the $\Q$-scheme $W_{k,d}(M,D)$, for $\mathscr{X}_d(M,D)=\mathcal{X}_{d}(M,D)$ we get the $\Z[1/MDd]$-scheme $\mathcal W_{k,d}(M,D)$ and for $\mathscr{X}_d(M,D)=\mathfrak{X}_d(M,D)$ we get the $\Z_{\ell^2}$-scheme $\mathfrak{W}_{k,d}(M,D)$. When $M$ and $D$ are understood, we omit them from the notation. 

To study the cohomology of Kuga--Sato varieties, we need to introduce projectors that will be used to identify the relevant cohomology groups with cohomology groups of the base Shimura curve. Fix an even integer $k\geq4$ and set $r\defeq k-2$. Fix also a $\Z_p$-algebra $A$; we assume that $A$ is either the valuation ring of a finite extension of $\Q_p$ or a quotient of it.   

We begin with $B=\M_2(\Q)$. In this case, define the sheaf 
\[ \mathcal{F}_A\defeq j_*\Sym^r\bigl(\R^1\pi_{k,d*}A\bigr)(r/2) \] 
on $\mathscr{X}_d(M,D)$, where $j:\mathscr{Y}_d(M)\hookrightarrow \mathscr{X}_d(M)$ is the canonical inclusion of the affine moduli space into its compactification.

When $B\not\simeq\M_2(\Q)$, we need to cut the dimension of the sheaf $\R^1\pi_{d*}\Z_p$ before taking symmetric powers. To do this, observe that $\R^1\pi_{d*}(\Z_p)$ is equipped with an action of $\mathcal{O}_B$; we fix an isomorphism of $\Z_p$-algebras $\mathcal{O}_B\otimes_\Z\Z_p\simeq \M_2(\Z_p)$ and let $e$ be the idempotent in $\mathcal{O}_B\otimes_\Z\Z_p$ that corresponds under this isomorphism to the matrix $\smallmat 1000$. 
Then define the sheaf 
\[\mathcal{F}_A\defeq\Sym^r\bigl(e\R^1\pi_{k,d*}A\bigr)(r/2)\] on $\mathscr{X}_d(M,D).$

\begin{remark} 
We use the same symbol $\mathcal{F}_A$ for sheaves on different Shimura curves, defined over different bases. However, no confusion is likely to arise from this convention. 
\end{remark}
 
\begin{remark} 
When $B\not\simeq \M_2(\Q$, there is the following alternative description of the sheaf $\mathcal{F}_A$. Let us define 
\[\mathbb{L}_2(A)\defeq\bigcap_{b\in B}\ker\Bigl(b-\Norm(b): \R^2\pi_{k,d*}A\longrightarrow \R^2\pi_{k,d*}A\Bigr), \]
which is a $3$-dimensional local system on $X_d$. Now consider the non-degenerate pairing 
\[ (\cdot,\cdot):\mathbb{L}_2(A)\otimes\mathbb{L}_2(A)\longmono \R^2\pi_{k,d*}\Q_p\otimes \R^2\pi_{k,d*}A
\overset\cup\longrightarrow\R^4\pi_{k,d*}A\underset\simeq{\overset{\Tr}\longrightarrow} A(-2), \]
where $\Tr$ is the trace map and the tensor products are taken over $A$. For $r>2$, there is a Laplace operator 
\[ \Delta_{r/2}:\Sym^{r/2}\bigl(\mathbb{L}_2(A)\bigr)\longrightarrow \Sym^{r/2-2}\bigl(\mathbb{L}_2(A)\bigr)(-2) \]
defined by 
\[ \Delta_r(x_1,\dots,x_{r/2})\defeq\sum_{1\leq i<j\leq r/2}(x_i,x_j)x_1\cdot\ldots\cdot\hat{x}_i\cdot\ldots\cdot\hat{x}_j\cdot\ldots\cdot x_{r/2}, \]
where $\mathbb{L}_{r}(A)\defeq\ker\bigl(\Delta_{r/2}\bigr)$ and $\hat{x}_\bullet$ means that the component $x_\bullet$ does not appear in the product. Then, by \cite[Theorem 5.8]{BesserCM} (see also \cite[Lemma 2.1]{wang}), the sheaves $\mathcal{F}_A$ and $\mathbb{L}_r(A)$ are isomorphic.
\end{remark}  

The Kuga--Sato variety $W_{k,d}$ can be equipped with several projectors, which we briefly introduce. First, there are a projector $\epsilon_k$ (see, \emph{e.g.}, \cite[\S2.1]{BDP} when $B=\M_2(\Q)$ and \cite[Theorem 5.8]{BesserCM}, \cite[\S6.1]{Brooks} when $B\not\simeq \M_2(\Q)$) and group isomorphisms 
\begin{equation} \label{projectors} 
\epsilon_k H^*_\et\bigl(\overline{W}_{k,d},\Z_p\bigr)\overset\simeq\longrightarrow \epsilon_k H^{k-1}_\et\bigl(\overline{W}_{k,d},\Z_p\bigr)\overset\simeq\longrightarrow H^1_\et\bigl(\overline{X}_d,\mathcal{F}_{\Z_p}\bigr) 
\end{equation}
(see, \emph{e.g.}, \cite[Lemma 2.2]{BDP} when $B=\M_2(\Q)$ and \cite[Theorem 6.1]{Brooks} when $B\not\simeq\M_2(\Q)$). Let $G_d$ be the Galois group of the covering $X_d\rightarrow X$; there are a projector $\epsilon_d\defeq\frac{1}{|G_d|}\cdot\sum_{g\in G_d}g$ and an isomorphism 
\[ \epsilon_dH^1_\et\bigr(\overline{X}_d,\mathcal{F}(\Z_p)\bigr)\overset\simeq\longrightarrow H^1_\et\bigr(\overline{X},\mathcal{F}_{\Z_p}\bigr) \]
of groups, where we still denote by $\mathcal{F}$ the push-forward sheaf of $\mathcal{F}_{\Z_p}$ to $X$. Finally, if $\epsilon\defeq\epsilon_d\epsilon_k$, then there is an isomorphism 
\[ \epsilon H^*_\et\bigl(\overline{W}_{k,d},\Z_p\bigr)\overset\simeq\longrightarrow H^1_\et\bigr(\overline{X},\mathcal{F}_{\Z_p}\bigr) \]
of groups.

%\subsection{Modular forms}Define the $A$-module  $S_k(M,D;A)$ of modular forms of weight $k$ and level $M$ on $B$ to be the cohomology group $\epsilon H^{k-1}_\et(\overline{W}_{k,d},A)$. This group is equipped with a canonical action of standard Hecke operators $T_q$ for $q\nmid MD$ and $U_q$ for $q\mid MD$ acting by correspondences, and we denote $\T_{M,D}(A)$ the Hecke algebra generated over $A$ by these operators. 
%
%\begin{remark} Alternatively, when $D>1$, we may define $S_k(M,D;A)$ to be the cohomology group $\epsilon_dH^1_\et\bigr(\overline{X}_d,\mathcal{F}_{A}\bigr)$, while if $D=1$ we denote $S_k(M,D;A)$ the parabolic cohomology of $\epsilon_dH^1_\mathrm{par}\bigr(\overline{X},\mathcal{F}_{A}\bigr)$, defined to be the image of the compactly supported cohomology $H^1_\mathrm{cpt}\bigr(\overline{Y},\mathcal{F}_{A}\bigr)$ in $H^1_\et\bigr(\overline{X},\mathcal{F}_{A}\bigr)$ (when $D=1$ we use a slight abuse of notation, identifying the \'etale cohomology with singular cohomology by the comparison theorem to make sense to compactly supported \'etale cohomology; see for example \cite[\S2.4]{LV-TNC} for a more clear discussion of this point.)\end{remark} 

\subsection{CM points on Shimura curves} \label{sec:CM} 

Shimura curves come equipped with a family of special points, whose description we briefly recall. Let $K$ be an imaginary quadratic field of discriminant $-D_K$ in which all the primes dividing $M$ split and all the primes dividing $D$ are inert. Define the set $\mathrm{CM}_K$ of \emph{CM points} of $\mathrm{Sh}_R$ by $K$ as 
\[ \mathrm{CM}_K\defeq
 B^\times\big\backslash\bigl(\widehat{ B}^\times\times\Hom(K, B)\bigr)\big/\widehat{ R}^\times, \] where $\Hom$ denotes homomorphisms of $\Q$-algebras, which we view in $\mathcal{H}^\pm$ by scalar extension to $\R$. 
Let $\cO_K$ be the ring of integers of $K$ and for any integer $c\geq 1$ denote $\mathcal{O}_c=\Z+c\mathcal{O}_K$ the order of $K$ of conductor $c$. If a point $x\in\mathrm{CM}_K$ is represented by a pair $(g,\varphi)$ such that $\varphi(K)\cap g^{-1}\widehat{ R}^\times g=\varphi(\cO_c)$, then we say that $\varphi$ is an \emph{optimal embedding} of $\cO_c$ into $ R$ and that $x$ is a \emph{Heegner point} of conductor $c$; in this case, by Shimura's reciprocity, $x$ is defined over the ring class field $H_c$ of $K$ conductor $c$. 

Let $K^\mathrm{ab}$ be the maximal abelian extension of $K$. The set $\mathrm{CM}_K$ is equipped with the following action of $\Gal(K^\mathrm{ab}/K)$: if $\sigma\in\Gal(K^\mathrm{ab}/K)$ is represented by $\mathfrak{a}\in\hat{K}^\times$ via 
the geometrically normalized Artin map and $x$ is represented by $(g,\varphi)$, then $x^\sigma\in\mathrm{CM}_K$ is the point of $\mathrm{CM}_K$ represented by $\bigl(\hat{\varphi}(\mathfrak{a})g,\varphi\bigr)$, where $\hat{\varphi}:\widehat K\rightarrow \widehat{ B}$ is the map obtained from $\varphi$ by extending scalars to $\widehat{\Z}$. By Shimura's reciprocity law, this (algebraic) action coincides with the usual (geometric) Galois action on points.

From now on, $x_n=[(g_n\tau,\psi_K)]$ will be a Heegner point of conductor $p^n$ for some integer $n\geq0$. Given $a\in \widehat{K}^\times$, we abbreviate $x_n^\sigma=\bigl[\bigl(\widehat{\psi}_K(a)g_n\tau,\psi_K\bigr)\bigl]$, where $\sigma$ corresponds to $a$ under the Artin map as before, by $x_n(a)=\bigl[(ag_n\tau,\psi_K)\bigr]$.

\subsection{Cycles in the N\'eron--Severi group} \label{sec:NS}

As at the end of \S\ref{sec:CM}, fix an odd prime number $p\nmid MD$ and an integer $n\geq 0$. Let $E_{p^n}$ be an elliptic curve with CM by $\cO_{p^n}$ equipped with a complex uniformization giving an identification $E_{p^n}(\C)=\C/\cO_{p^n}$. 
Recall the Heegner point $x_n$ of conductor $p^n$ defined in \S\ref{sec:CM} and 
let $\tilde{x}_n$ be a lift of $x_n$ to ${X}_d$. At the cost of suitably replacing the elements $g_n$ chosen before, we can assume that $\tilde{x}_n$ is represented by a triple $(E_{p^n},C_{p^n},\nu_d)$ if $D=1$ and by a quadruplet $(E_{p^n}\times E_{p^n},\iota_{p^n},C_{p^n},\nu_d)$ if $D>1$.

Set $D_{p^n}\defeq p^{2n}D_K$ and let
\[ \Gamma_{n}\defeq\Bigl\{\bigl(\sqrt{-D_{p^n}}\cdot z,z\bigr)\;\Big|\;z\in E_{p^n}\Bigr\} \]
be the transpose of the graph of $\sqrt{-D_{p^n}}$ in $E_{p^n}\times E_{p^n}$. Let $Z_n$ be the image of the divisor $\Gamma_{n}-(E_{p^n}\times\{0\})-{p^n}D_K(\{0\}\times E_{p^n})$ in the N\'eron--Severi group $\mathrm{NS}(E_{p^n}\times E_{p^n})$ of $E_{p^n}\times E_{p^n}$; recall that $\mathrm{NS}(E_{p^n}\times E_{p^n})$ is a free $\Z$-module of rank $4$ and that (by \cite[Ch. 2, (3.6)]{Nek2} if $D=1$ and \cite[Proposition 4.2]{E-dVP} if $D>1$) $Z_n$ is also equal to the image of the divisor $\Gamma_{n}$ in $\mathrm{NS}(E_{p^n}\times E_{p^n})$. The abelian group $\mathrm{NS}(E_{p^n}\times E_{p^n})$ is equipped with a canonical intersection pairing and $Z_n$ belongs to the orthogonal complement of the free $\Z$-submodule of rank $3$ generated by $E_{p^n}\times\{0\}$, $\{0\}\times E_{p^n}$ and the diagonal in $E_{p^n}\times E_{p^n}$. Furthermore, the $\Q$-vector space $\mathrm{NS}(E_{p^n}\times E_{p^n}{)}_\Q\defeq\mathrm{NS}(E_{p^n}\times E_{p^n})\otimes_\Z\Q$ is equipped with a right action of $B^\times$ defined by $\mathcal{L}\cdot b\defeq\iota(b)^*\mathcal{L}$ for all $\mathcal{L}\in\mathrm{NS}(E_{p^n}\times E_{p^n})_\Q$; here $\iota$ is induced (via extension of scalars from $\Z$ to $\Q$) by the quaternionic action $\iota_{p^n}:R\hookrightarrow \End(E_{p^n}\times E_{p^n})$ when $D>1$, while if $D=1$, then $\iota$ is induced (again via extension of scalars from $\Z$ to $\Q$) by the standard left action of $\M_2(\Z)$ on $E_{p^n}\times E_{p^n}$ given by $\smallmat abcd\cdot\binom{x}{y}=\binom {ax+by}{cx+dy}$. The element $Z_n$ is uniquely characterized (up to sign) in $\mathrm{NS}(E_{p^n}\times E_{p^n})$ by the following two conditions: 
\begin{enumerate}
\item the action of $b$ is through the norm map, \emph{i.e.}, $Z_n\cdot b=\Norm(b)Z_n$ where $\Norm:B\rightarrow \Q$ is the reduced norm map (equal to the determinant if $D=1$); 
\item the self-intersection number of $Z_n$ is $-2p^{2n}D_K$. 
\end{enumerate} 
To prove this result, one can use the fact that the self-intersection is non-degenerate on the line where $Z_n$ lives, which is characterized as the unique line on which $B^\times$ acts through the norm: see, \emph{e.g.}, \cite[Proposition 8.2]{IS} for details. 

\subsection{Heegner cycles} 

As before, fix an odd prime number $p\nmid MD$ and an integer $n\geq 0$. Choose also an integer $d\geq1$ such that $p\nmid d\phi(d)$. We use the symbol $i_{\tilde{x}_n}$ for the inclusion of the fiber over $\tilde{x}_n$ in the universal object; in particular, $i_{\tilde{x}_n}:E_{p^n}\rightarrow \mathcal{E}_d$ for $D=1$, while $i_{\tilde{x}_n}:E_{p^n}\times E_{p^n}\rightarrow \mathcal{A}_d$ for $D>1$. We consider the canonical inclusion $i_{\tilde{x}(n)}\Bigl(Z_n^{\frac{k-2}{2}}\Bigr)\hookrightarrow W_{k,d}$; by \eqref{projectors}, the cycle $\epsilon i_{\tilde{x}_n}(Z_n^\frac{k-2}{2})$ belongs to the kernel of the cycle class map.

\begin{definition} \label{defHeegnercycle} 
The \emph{Heegner cycle} of conductor $p^n$ is 
\[ \Delta_n\defeq\epsilon i_{\tilde{x}_n}\Bigl(Z_n^\frac{k-2}{2}\Bigr)\in\epsilon\mathrm{CH}^{k/2}_0(W_{k,d}/H_{p^n}{)}_{\Z_p}. \]
\end{definition}

We write $\Delta_{n,M,D} $ in place of $\Delta_n$ when we need to specify $M$ and $D$. Let us set $A_n\defeq E_{p^n}\times E_{p^n}$. For the integer $k\geq 4$ fixed before, we also sometimes put $r\defeq k-2$. 

\subsection{$p$-adic Abel--Jacobi map}

To simplify our notation, in the lines below we write $W$ for $W_{k,d}$ and $\Delta$ for $\Delta_n$; moreover, we understand that $H^\bullet=H^\bullet_\text{\'et}$. The image of $\Delta$ under the $p$-adic Abel--Jacobi map 
\[ \AJ:\epsilon \mathrm{CH}^{k/2}_0(W_{k,d}/H_{p^n}{)}_{\Z_p}\longrightarrow
H^1\Bigl(H_{p^n},\epsilon H^{k-1}_\text{\'et}\bigl(\overline{W}_{k,d},\Z_p(r/2)\bigr)\!\Bigr) \]
is represented by the extension class obtained by pull-back of the extension
\[
\begin{split}
0\longrightarrow\epsilon H^{k-1}\bigl(\overline{W},\Z_p(k/2)\bigr)&\longrightarrow\epsilon  H^{k-1}\bigl(\overline{W}-\overline{\Delta},\Z_p(k/2)\bigr)\\&\longrightarrow \epsilon  H^{k}_{\overline{\Delta}}\bigl(\overline{W},\Z_p(k/2)\bigr)\longrightarrow \epsilon H^{k}\bigl(\overline{W},\Z_p(k/2)\bigr)
\end{split}
\]
via the map $c\ell:\Z_p\rightarrow\epsilon H^{k}_{\overline{\Delta}}\bigl(\overline{W},\Z_ p (k/2)\bigr)$ that takes $1$ to the class $c\ell(\Delta)$ of $\Delta$. 

Following \cite[II, Proposition 2.14, (2)]{Nek2} for $D=1$ and \cite[\S8.1]{Chida} for $D>1$, the extension class above can be rewritten as follows. Using the notation introduced for varieties, set $\bar{\tilde{x}}_n\defeq\tilde{x}_n\otimes_{H_{p^n}}\overline\Q$; let $\overline{W}_{\bar{\tilde{x}}_n}$ be the fiber of $\overline{W}$ at $\bar{\tilde{x}}_n$. By purity, there is an isomorphism  
\begin{equation}\label{purity}
H^{k}_{\overline{\Delta}}\bigl(\overline{W},\Z_ p(k/2)\bigr)\simeq 
H^{k-2}\bigl(\overline{W}_{\bar{\tilde{x}}_n},\Z_ p(r/2)\bigr).
\end{equation}
Then we use Scholl's result in \cite{Scholl} (the argument is detailed in \cite[II, Proposition 2.14, (1)]{Nek2} in the modular curve case and extended in \cite[\S10.1]{IS} to a general quaternionic setting) to obtain isomorphisms (which will be viewed as identifications) 
\begin{equation} \label{pur}
\begin{split}
\epsilon H^{k-2}\bigl(\overline{W}_{\bar{\tilde{x}}_n},\Z_ p (r/2)\bigr)
&\simeq H^2_{\bar{\tilde{x}}_n}\bigl(\overline{X},\mathcal{F}_{\Z_p}(1)\bigr),\\
%\simeq H^0(\bar{\tilde{x}}_n,\mathcal{F})\simeq L_r(\Z_p),\\
\epsilon H^{k-1}\bigl(\overline{W},\Z_ p (k/2)\bigr)&\simeq H^1\bigl(\overline{X},\mathcal{F}_{\Z_p}(1)\bigr),\\
\epsilon H^{k-1}\bigl(\overline{W}-\overline{\Delta},\Z_ p (k/2)\bigr)&\simeq 
H^1\bigl(\overline{X} \smallsetminus \bar{\tilde{x}}_n,,\mathcal{F}_{\Z_p}(1)\bigr),
\end{split}
\end{equation}
where we have set $X\defeq X_{k,d}$. As a consequence, $\mathrm{AJ}(\Delta)$ is represented by the extension class obtained by pull-back of the extension 
\begin{equation}\label{extclass}
0\longrightarrow H^1\bigl(\overline{X},\mathcal{F}_{\Z_p}(1)\bigr)\longrightarrow H^1\bigl(\overline{X}\smallsetminus \bar{\tilde{x}}_n,,\mathcal{F}_{\Z_p}(1)\bigr)\longrightarrow 
H^2_{\bar{\tilde{x}}_n}\bigl(\overline{X} ,\mathcal{F}_{\Z_p}(1)\bigr)\longrightarrow 0
\end{equation}
via the map $c\ell:\Z_p\rightarrow H^2_{\bar{\tilde{x}}_n}\bigl(\overline{X} ,\mathcal{F}_{\Z_p}(1)\bigr)$ that, under the identifications in \eqref{purity} and in the first line of \eqref{pur}, takes $1$ to the class $c\ell(\Delta)$ of $\Delta$ (see again \cite[II, Proposition 2.14]{Nek2} for details). The fiber $\overline{W}_{\bar{\tilde{x}}_n}$ of $W$ over $\bar{\tilde{x}}_n$ is $(E_n\times E_n)^{r/2}$; applying the projector $\epsilon$, we obtain an isomorphism of Galois representations 
\[ \epsilon H^{k-2}\bigl(\overline{W}_{\bar{\tilde{x}}_n},\Z_ p (r/2)\bigr)\simeq \mathrm{Sym}^r \bigl(H^1(\overline{E}_n,\Z_p)\bigr)(r/2). \]
Set $K_p\defeq K\otimes_\Q\Q_p$. If $p$ splits in $K$ as $p\cO_K=\wp\cdot\bar\wp$, where $\wp$ is the prime ideal of $\cO_K$ corresponding to the fixed embedding $\bar\Q\hookrightarrow\bar\Q_p$, then $K_p=K_\wp\oplus K_{\bar\wp}\simeq\Q_p\oplus\Q_p$; in this case, we set $L\defeq K_\wp$. If $p$ is inert in $K$, then $K_p$ is the unramified quadratic extension of $\Q_p$; in this case, we set $L\defeq K_p$. Let
\begin{equation} \label{intersection-eq}
\smile\,:H^1\bigl(\overline{E}_n,L\bigr)\times H^1\bigl(\overline{E}_n,L\bigr)\longrightarrow L
\end{equation}
be the intersection pairing. Since $E_n$ has CM by $\mathcal{O}_{p^n}$, we can choose a basis $\{x_1,x_2\}$ of $H^1(\overline{E}_n,L)$ over $L$ such that
\begin{itemize}
\item the action of $\alpha\in K$ on $x_1$ is given by $x_1\mapsto \alpha x_1$ and the action of $\alpha\in K$ on $x_2$ is given by $x_2\mapsto \bar\alpha x_2$;
\item the pairing in \eqref{intersection-eq} is represented by the matrix $\smallmat 01{-1}0$ with respect to $\{x_1,x_2\}$. 
\end{itemize} 
Let us identify $\mathrm{Sym}^{k-2}\bigl(H^1(\overline{E}_n,L)\bigr)$ with $L_{k-2}(L)$ by sending $x_1$ to $X$ and $x_2$ to $Y$; then the action of $g\in\GL_2(L)$ on $\mathrm{Sym}^{k-2}\bigl(H^1(\overline{E}_n,L)\bigr)$ is given by $\det(g)\rho_k(g)$, while the intersection pairing is the symmetrization of the intersection pairing $2^{k-2}{\langle\cdot,\cdot\rangle}_k$ (see, \emph{e.g.}, \cite{Bry}). 

The next auxiliary result is well known, but we add some details for the convenience of the reader (see also \cite[Lemma 7.2]{Chida}, \cite[Lemma 8.4]{IS}, \cite[Lemma 3.3]{wang}). Recall that $\mathbf{w}_n=p^\frac{n(k-2)}{2}\mathbf{v}_0^*$. 

\begin{lemma} \label{ERLbase}
The equality $c\ell(\Delta_n)=\rho_k^{-1}(\gamma_\p)(\mathbf{w}_n)$ holds in $L_r(\Z_p)$ up to sign. 
\end{lemma}

\begin{proof} Given $t\in K$, the action of $\rho_{k,\infty}(t)$ on $\mathbf{v}_m$ is via $(t/\bar{t})^m$, and $\rho_{k,\infty}=\rho_{p,\infty}$ on $K^\times$, so the action of $\rho_{k,p}(t)$ on $\mathbf{v}_m$ is via $(t/\bar{t})^m$, where if $p$ splits in $K$, then $(t,\bar{t})$ is the diagonal $(t,t)$ under the isomorphism $K_p\simeq\Q_p\oplus\Q_p$, while if $p$ is inert in $K$, then $\bar{t}$ denotes the action of the non-trivial element of $\Gal(K_p/\Q_p)$. Now recall that $\rho_{k,p}(\gamma_{\mathfrak{p}}^{-1}g\gamma_\mathfrak{p})=\rho_k(g)$, so that, in particular, $\rho_{k,p}(\gamma_\p)=\rho_k(\gamma_\p)$ and $\rho_k(\gamma_\p^{-1})(\mathbf{v}_m)=\rho_{k,p}(\gamma_\p^{-1})(\mathbf{v}_m)$; thus, $\rho_k(\gamma_\p^{-1})(\mathbf{v}_m)$
generates the subspace on which $\rho_{k}(t)$ acts via $(t/\bar{t})^m$. For $t\in K$, we see that $\det(t)\rho_k(t)$ acts on $\rho_k^{-1}(\gamma_\p)(\mathbf{v}_0^*)$ as the norm map, and then the same is true for $\rho_k^{-1}(\gamma_\p)(\mathbf{w}_n)$; since $B^\times$ acts on $\Delta_n$ by the norm map, the comment before the statement of the lemma shows that $c\ell(\Delta_n)$ and $\rho_k^{-1}(\gamma_\p)(\mathbf{w}_n)$ belong to the same line. 
Now the self-intersection of $\Delta_n$ is, up to sign, $2^{k-2}\big\langle\rho_k^{-1}(\gamma_\p)(\mathbf{w}_n),\rho_k^{-1}(\gamma_\p)(\mathbf{w}_n)\big\rangle=2^{k-2}\langle\mathbf{w}_n,\mathbf{w}_n\rangle$, which concludes the proof of the lemma. \end{proof}

\subsection{Heegner classes} \label{classes} 

The Shimura curve $X(M,D)$ is naturally equipped with an action (via correspondences) of Hecke operators $T_q$ for primes $q\nmid MD$ and $U_q$ for primes $q\,|\,MD$ (see, \emph{e.g.}, \cite[\S2.4]{LV-MM} for details). Let $\T_{M,D}$ denote the $\Z$-algebra generated by these Hecke operators and for any ring $A$ set $\T_{M,D}(A)\defeq\T_{M,D}\otimes_\Z A$. Now let $A$ be either the valuation ring of a finite extension of $\Q_p$ or a quotient of such a ring. Let $\phi:\T_{M,D}(A)\twoheadrightarrow A$ be an $A$-algebra homomorphism; we assume that $\phi$ is \emph{ordinary at $p$}, \emph{i.e.}, $\phi(T_p)\in A^\times$. There is a natural action of $\T_{M,D}(A)$ on $\epsilon_d H^1_\et\bigl(\overline{X}_d,\mathcal{F}_A(k/2)\bigr)$: taking the $\phi$-eigenspace (which can be done for all rings $A$ as before, thanks to \cite[Lemma 2.2]{Nek}) allows us to define 
\[ T_\phi\defeq\epsilon_d H^1_\et\bigr(\overline{X}_d,\mathcal{F}_A(k/2)\bigr)[\phi]. \]
Thus, for any number field $L$ we obtain an Abel--Jacobi map
%\begin{equation} \label{AJ1}
\[ \mathrm{AJ}_{W,\phi,L}:\epsilon\CH^{k/2}_0(W_{k,d}\otimes_\Q L{)}_A\longrightarrow H^1(L,T_{\phi}). \]
%\end{equation}
If the field $L$ is clear from the context, then we omit it from the notation.

\begin{definition}\label{defHC} 
The \emph{Heegner class} of conductor $p^n$ attached to $\phi$ is 
\[ y_{n}(\phi)\defeq\mathrm{AJ}_{W,\phi,H_{p^n}}(\Delta_{n})\in H^1(H_{p^n},T_{\phi}). \]
\end{definition}

Again, when we need to specify $M$ and $D$, we use the complete notation $y _{n,M,D}(\phi)$ for $y _{n}(\phi)$. On the other hand, if $\phi$ is understood, then we simply write $y_n$ for $y_n(\phi)$. 

\subsection{Regularized Heegner classes} \label{reg-heeg}

As before, let $\phi:\T_{M,D}(A)\twoheadrightarrow A$ be an $A$-algebra homomorphism; let $\alpha_p\in A^\times$ be the unit root of the Hecke polynomial $X^2-\phi(T_p)X-p^{k-1}$ at $p$. For all integers $n\geq1$, define 
\[ 
z_{n}(\phi)\defeq\frac{1}{\alpha_p^n}\cdot y_{p^n}(\phi)-\frac{p^{k-2}}{\alpha_p^{n+1}}\cdot y_{p^{n-1}}(\phi)\in H^1(H_{p^n},T_\phi). 
\] 
Set $u\defeq\#\cO_K^\times/2$; if $p$ splits in $K$, then write $p\cO_K=\wp\cdot\bar\wp$ and denote by $\sigma_\star$ a Frobenius element at $\star\in\{\wp,\bar\wp\}$. Define
\[ 
z_0(\phi)\defeq\begin{cases}\displaystyle{\frac{1}{u}\cdot\Biggl( 1-\frac{p^\frac{k-2}{2}}{\alpha_p}\cdot\sigma_\wp\Biggr)\cdot\Biggl( 1-\frac{p^\frac{k-2}{2}}{\alpha_p}\cdot\sigma_{\bar\wp}\Biggr)\cdot y_1(\phi)}\quad &\text{if $p\cO_K=\wp\cdot\bar\wp$},\\[7mm]
\displaystyle{\frac{1}{u}\cdot\Biggl(1-\frac{p^{k-2}}{\alpha_p^2}\Biggr)\cdot y_1(\phi)}\quad &\text{if $p$ is inert in $K$}.
\end{cases}
\]
When $\phi$ is understood, we simply write $z_n$ for $z_n(\phi)$.

The classes $z_n$ enjoy the following trace-compatibility relations. 

\begin{proposition} \label{prop-compatibility}
The equality $\cores_{H_{p^n}/H_{p^{n-1}}}(z_n)=z_{n-1}$ holds for all $n\geq 1$. 
\end{proposition}

\begin{proof} With notation as above, class field theory gives the equality
\[  
[H_p:H]=\begin{cases}\displaystyle{\frac{p-1}{u}}&\text{if $p$ splits in $K$},\\[3mm]\displaystyle{\frac{p+1}{u}}&\text{if $p$ is inert in $K$}. \end{cases}
\]
Then a direct computation shows that to prove the proposition it suffices to prove that
\begin{equation} \label{formuletta}
T_p(y_{p^{n-1}})=p^{k-2}y_{p^{n-2}}+\cores_{H_{p^n}/H_{p^{n-1}}}(y_{p^n})
\end{equation}
for all $n\geq2$ and 
\begin{equation} \label{formuletta2}
T_p(y_1)=\begin{cases}p^{\frac{k-2}{2}}(\sigma_\wp+\sigma_{\bar\wp})y_1+u\cdot\cores_{H_p/H}(y_p)&\text{if $p$ splits in $K$},\\[3mm]u\cdot\cores_{H_p/H}(y_p)&\text{if $p$ is inert in $K$}.
\end{cases}
\end{equation}
For all $n\geq 1$, the action of $T_p$ on $y_{p^{n-1}}$ is given by
\begin{equation} \label{T_p-eq}
T_p(y_{p^{n-1}})=\sum_{\substack{\mathcal{L}\subset \cO_{p^{n-1}}\times \cO_{p^{n-1}}\\ [\cO_{p^{n-1}}\times \cO_{p^{n-1}}: \mathcal{L}]=p^2}} \psi_\mathcal{L}^*(y_{p^{n-1}}),
\end{equation}
where $\mathcal{L}$ runs over the sublattices of $\cO_{p^{n-1}}\times \cO_{p^{n-1}}$ of index $p^2$ and $\psi_\mathcal{L}\colon A_\mathcal{L}\defeq\C^2/\mathcal{L}\rightarrow A_{p^{n-1}}$ is the isogeny of degree $p^2$ that is induced by the inclusion $\mathcal{L}\hookrightarrow \cO_{p^{n-1}}\times \cO_{p^{n-1}}$. 

We first assume $n\geq2$. We distinguish two cases:
\begin{enumerate} 
\item $A_\mathcal{L}$ represents a Heegner point of conductor $p^n$;
\item $\mathcal{L}=p\cO_{p^{n-2}}\times p\cO_{p^{n-2}}$.
\end{enumerate}
In case (1), by \cite[Section 2.2]{E-dVP} we can write $\mathcal{L}$ as $\cO_{p^n}\times L$, where $L=e\mathfrak{a}$ for some $e\in \mathcal{B}$ and $\mathfrak{a}$ is a proper fractional $\cO_{p^n}$-ideal that represents the trivial class in $\Pic (\cO_{p^{n-1}})$. We can assume that $\mathfrak{a}^{-1}$ is an integral $\cO_{p^n}$-ideal. If we set $E_\mathfrak{a}\defeq\C/\cO_{p^n}\times\C/\mathfrak{a}$, then there are an isomorphism $A_\mathcal{L}\simeq E_{p^n}\times E_\mathfrak{a}$ and an equality $\psi_\mathcal{L}=\psi_{p^n}\times \psi_\mathfrak{a}$, where $\psi_{p^n}\colon E_{p^n}\rightarrow E_{p^{n-1}}$ and $\psi_\mathfrak{a}\colon E_\mathfrak{a}\rightarrow E_{p^{n-1}}$ are the isogenies induced by the inclusions $\cO_{p^n}\hookrightarrow\cO_{p^{n-1}}$ and $\mathfrak{a}\hookrightarrow\cO_{p^{n-1}}$, respectively; both maps have degree $p$. Let $\sigma_\mathfrak{a}\in\Gal(H_{p^n}/H_{p^{n-1}})$ be the image of $\mathfrak{a}$ under the geometrically normalized Artin map. By the theory of complex multiplication, one has $E_\mathfrak{a}=E_{p^n}^{\sigma_\mathfrak{a}^{-1}}$. Observing that $E_\mathfrak{a}=E_{p^n}/E_{p^n}[\mathfrak{a}^{-1}]$, there is an equality
\begin{equation} \label{comp}
[p]\circ\varphi_{p^{n-1}}\circ\psi_{p^n}=\psi_\mathfrak{a}\circ\lambda_\mathfrak{a}\circ \varphi_{p^n},
\end{equation}
where $\varphi_c$ denotes the multiplication-by-$\sqrt{-D_c}$ map on $E_c$ and $\lambda_\mathfrak{a}\colon E_{p^n}\rightarrow E_\mathfrak{a}$ is the quotient map. Now consider the graph
\[ 
\Gamma_\mathfrak{a}\defeq\Bigl\{\Bigl(x,\lambda_\mathfrak{a}\bigl(\varphi_{p^n}(x)\bigr)\!\Bigr)\;\Big|\;x\in E_{p^n}\Bigr\}\subset E_{p^n}\times E_\mathfrak{a}.
\]
If $T_x$ is the translation-by-$x$ map for $x\in E_{p^n}\times E_\mathfrak{a}$, then equality \eqref{comp} yields an equality
\[ \bigsqcup_{z\in\ker(\psi_\mathfrak{a})}T_{(0,z)}^*(\Gamma_\mathfrak{a})=([p]\times \mathrm{id})^*(\psi_{p^n}\times \psi_\mathfrak{a})^*(\Gamma_{p^{n-1}}) \]
in $E_{p^n}\times E_\mathfrak{a}$. In fact, there are equalities
\begin{equation*}
\begin{split}
\bigsqcup_{z\in\ker(\psi_\mathfrak{a})}T_{(0,z)}^*(\Gamma_\mathfrak{a})&=\Bigl\{(x,y)\in E_{p^n}\times E_\mathfrak{a}\;\Big|\;\psi_\mathfrak{a}(y)=\psi_\mathfrak{a}\Bigl(\lambda_\mathfrak{a}\bigl(\varphi_{p^n}(x)\bigr)\!\Bigr)\Bigr\}\\
& =\bigl\{(x,y)\in E_{p^n}\times E_\mathfrak{a}\;\big|\;\psi_\mathfrak{a}(y)=\varphi_{p^{n-1}}\bigl(p\psi_{p^n}(x)\bigr)\bigr\}\\
& = ([p]\times\mathrm{id})^*(\psi_{p^n}\times \psi_\mathfrak{a})^*(\Gamma_{p^{n-1}}).
\end{split}
\end{equation*}
Thus, the equality $p\cdot\Gamma_\mathfrak{a}=p\cdot(\psi_{p^n}\times \psi_\mathfrak{a})^*(\Gamma_{p^{n-1}})$ holds in the N\'eron--Severi group of the product $E_{p^n}\times E_\mathfrak{a}$, and hence
\begin{equation} \label{eq 1}
y_{p^n}^{\sigma_\mathfrak{a}^{-1}}=\psi_\mathcal{L}^*(y_{p^{n-1}}).
\end{equation}
In case (2), we have $A_\mathcal{L}\simeq \C/p\cO_{p^{n-2}}\times \C/p\cO_{p^{n-2}}\simeq E_{p^{n-2}}\times E_{p^{n-2}}$ and $\psi_\mathcal{L}=\psi_{p^{n-2}}\times \psi_{p^{n-2}}$. It follows that
\[ \psi_{p^{n-2}}\circ[p]\circ\varphi_{p^{n-2}}=\varphi_{p^{n-1}}\circ \psi_{p^{n-2}}. \]
As above, one can easily see that
\[ 
\bigsqcup_{z\in \ker(p\psi_{p^{n-2}})}T_{(0,z)}^*(\Gamma_{p^{n-2}})=(\psi_{p^{n-2}}\times\psi_{p^{n-2}})^*(\mathrm{id}\times [p])^*(\Gamma_{p^{n-1}}), 
\]
which gives the equality $p^3\cdot\Gamma_{p^{n-2}}=p\cdot\psi_\mathcal{L}^*\Gamma_{p^{n-1}}$ in the N\'eron--Severi group of the product $E_{p^{n-2}}\times E_{p^{n-2}}$. Then we get
\begin{equation} \label{eq 2}
(p^2)^{\frac{k-2}{2}} y_{p^{n-2}}=\psi_\mathcal{L}^*(y_{p^{n-1}}).
\end{equation}
Finally, combining \eqref{T_p-eq}, \eqref{eq 1} and \eqref{eq 2} we obtain \eqref{formuletta}.

Now assume $n=1$. Then 
\[ 
T_p(y_1)=\sum_{\substack{\mathcal{L}\subset \cO_K\times \cO_K\\ [\cO_K\times \cO_K:\mathcal{L}]=p^2}}\psi_\mathcal{L}^*(y_1). 
\]
If splits in $K$ as $p\cO_K=\wp\cdot\bar\wp$, we distinguish the following three cases:
\begin{itemize}
\item[(1')] $A_\mathcal{L}$ represents a Heegner point of conductor $p$;
\item[(2')] $\mathcal{L}=\bar\wp\times \bar\wp$;
\item[(3')] $\mathcal{L}=\wp\times \wp$.
\end{itemize}
Case (1') can be treated like case (1) above: now $\mathfrak{a}$ will vary over all fractional $\cO_p$-ideals that are trivial in $\Pic(\cO_K)$. By class field theory, every such $\mathfrak{a}$ is represented by an element of $\Gal(H_p/H)$ up to multiplication by elements of $\cO_K^\times/\cO_p^\times$, whose cardinality is $u$. Cases (2') and (3') can be dealt with in much the same way, so we see only what happens in (2'). In this case, there are splittings $A_\mathcal{L}\simeq \C/\bar\wp\times \C/\bar\wp\simeq E_1^{\sigma_\wp}\times E_1^{\sigma_\wp}$ and $\psi_\mathcal{L}=\psi_{\bar\wp}\times\psi_{\bar\wp}$, where $\psi_{\bar\wp}\colon \C/\bar\wp\rightarrow E_1$ is induced by the inclusion $\bar\wp\hookrightarrow \cO_K$ and has degree $p$. Here we have implicitly used the fact that $\bar\wp=p\wp^{-1}$ and hence $\C/\bar\wp\simeq\C/\wp^{-1}$; furthermore, complex multiplication gives an isomorphism $\C/\wp^{-1}\simeq E_1^{\sigma_\wp}$. It turns out that
\[ 
\psi_{\bar\wp}\circ \varphi_1^{\sigma_\wp}=\varphi_1\circ \psi_{\bar\wp},
\]
where $\varphi_1^{\sigma_\wp}$ is the multiplication-by-$\sqrt{-D_K}$ map on $E_1^{\sigma_\wp}$. Computations as above imply that $p\cdot\Gamma_{\varphi_1^{\sigma_\wp}}=(\psi_{\bar\wp}\times\psi_{\bar\wp})^*(\Gamma_{\varphi_1})$ in the N\'eron--Severi group of $E_1^{\sigma_\wp}\times E_1^{\sigma_\wp}$, so there is an equality
\[ 
p^{\frac{k-2}{2}}y_1^{\sigma_\wp}=\psi_\mathcal{L}^*(y_1). 
\]
Collecting all contributions from cases (1'), (2') and (3'), we get \eqref{formuletta2}.

Finally, assume that $p$ is inert in $K$. The only sublattices $\mathcal{L}\subset \cO_K\times\cO_K$ with the desired properties are those for which $A_\mathcal{L}$ represents a Heegner point of conductor $p$, and we end up in the same case as (1'). Thus, formula \eqref{formuletta2} follows. \end{proof}

%Thanks to Proposition \ref{prop-compatibility}, taking the inverse limit with respect to corestriction maps yields a class \[z_\infty(\phi)=\invlim_{n\geq 0}z_n(\phi)\in H^1_\mathrm{Iw}(H_{p^\infty}, T_\phi)\defeq\invlim_{n\geq 0}H^1(H_{p^n},T_\phi).\]
For all integers $n\geq0$, corestrictions give classes 
\[
\mathbf{z}_n(\phi)\defeq\cores_{H_{p^{n+1}}/K_n}\bigl(z_{n+1}(\phi)\bigr)\in H^1(K_n,T_\phi) 
\]
that, by Proposition \ref{prop-compatibility}, are compatible under corestrictions. Then we can define a class 
\begin{equation}\label{Heegnerclass} \mathbf{z}_\infty(\phi)\defeq\invlim_n\mathbf{z}_n(\phi)\in H^1_\mathrm{Iw}(K_\infty, T_\phi)\defeq\invlim_n H^1(K_n,T_\phi). \end{equation}
When we want to stress dependence on $M$ and $D$, we write $\mathbf{z}_{n,M,D}(\phi)$ and $\mathbf{z}_{\infty,M,D}(\phi)$ for $\mathbf{z}_n(\phi)$ and $\mathbf{z}_\infty(\phi)$, respectively.

\begin{remark} \label{shapiro}
Define $\mathbf{T}_\phi\defeq T_\phi\otimes A\llbracket G_\infty\rrbracket$. 
By Shapiro's lemma, there is an isomorphism
\[ H^1_\mathrm{Iw}(K_\infty, T_\phi)\simeq H^1(K,\mathbf{T}_\phi). \] 
Thus, we may view $\mathbf{z}_\infty(\phi)$ as an element of $H^1(K,\mathbf{T}_\phi)$. 
\end{remark}
%We view $z_\infty(\phi)$ as an element in $H^1(H,T_\phi\otimes R\llbracket\Gal(H_{p^\infty}/K)\rrbracket)$ via the map induced in cohomology by the inclusion $\Gal(H_{p^\infty}/H)\hookrightarrow\Gal(H_{p^\infty}/K)$.
%Let\[\tilde z_\infty(\phi)=\mathrm{cores}_{H/K}(z_\infty(\phi))\in H^1(K,T_\phi\otimes_A A\llbracket\Gal(H_{p^\infty}/K)\rrbracket),\]and define $\mathbf{z}(\phi)$ as the image of $\tilde z_\infty(\phi)$ under the map induced in cohomology by the projection $\Gal(H_{p^\infty}/K)\twoheadrightarrow\Gal(K_\infty/K)$, where $K_\infty$ is the anticyclotomic $\Z_p$-extension of $K$. Similarly, for any integer $n\geq 0$, again Shapiro's lemma we have 
%$H^1(H_{p^{n+1}}, T_\phi)\simeq H^1(H,T_\phi\otimes A[\Gal(H_{p^{n+1}}/H)])$, let 
%\[\tilde z_{n+1}(\phi)=\mathrm{cores}_{H/K}(z_{n+1}(\phi))\in H^1(K,T_\phi\otimes A[\Gal(H_{p^{n+1}}/K)]),\]
%and define $\mathbf{z}_n(\phi)$ as the image of $\tilde z_{n+1}(\phi)$ under the map induced in cohomology by the projection $\Gal(H_{p^{n+1}}/K)\twoheadrightarrow\Gal(K_{n}/K)$, where $K_n\subset K_\infty$ has Galois group $\Gal(K_{n}/K)\simeq\Z/p^n\Z$, for all $n\geq 0$. 

\section{Explicit reciprocity laws}

Our purpose in this section is to extend the explicit reciprocity laws of \cite{wang} to the case of Heegner cycles with $p$-power conductors.

\subsection{Galois representations} \label{Gal}

Fix square-free coprime integers $M,D\geq1$; let $B_D$ denote the quaternion algebra of discriminant $D$ and fix an Eichler order $R_M$ of $B_D$ of level $M$ (here 
$B_D$ can be either definite or indefinite). Fix a prime number $p\nmid MD$ and let $A$ be either the valuation ring of a finite extension of $\Q_p$ or a quotient of it. 
Let $\T_{M,D}(A)$ denote the Hecke algebra acting on modular forms $S_k(M,D;A)$ as in \S\ref{mod} if $B_D$ is definite, and acting as in \S\ref{classes} by correspondences on the \'etale cohomology of the Shimura curve $X(M,D)$ with coefficients in $\mathcal{F}_A$ if $B_D$ is indefinite. Fix a homomorphism $\phi:\T_{M,D}(A)\rightarrow A$ of $A$-algebras and assume that $\phi(T_p)\in A^\times$. We also assume  
\begin{itemize}
\item $p>k+1$ and $\#{(\F_p^\times})^{k-1}>5$. 
\end{itemize}
Let $\rho_\phi:G_\Q\rightarrow\GL_2(A)$ be the representation uniquely characterized by requiring that a geometric Frobenius $F_q$ at primes $q\nmid MDp$ satisfies 
\[ \det\bigl(1-\rho_\phi(F_q)X\bigr)=1-\frac{\phi(T_q)}{q^{k/2}}X+\frac{X^2}{q}. \] 
Let $T_\phi=T_\phi(A)$ be a free $A$-module of rank $2$ affording the representation $\rho_\phi$; when $B_D$ is indefinite, this representation is equivalent to the representation constructed in \S\ref{classes}. 

The $A$-module $T_\phi$ is \emph{ordinary at $p$} if there is a short exact sequence of $G_{\Q_p}$-modules 
\begin{equation}\label{ordinaryfil}
0\longrightarrow {T}_\phi ^{(p)}\longrightarrow {T}_\phi  \longrightarrow  {T}_\phi  ^{[p]}\longrightarrow 0 \end{equation}
such that ${T}_\phi^{(p)}$ and $T_\phi^{[p]}$ are both free $A$-modules of rank $1$ and there is an isomorphism of $G_{\Q_p}$-representations 
\begin{equation}\label{ordinaryfilrep}
{T}_\phi\simeq \begin{pmatrix}\eta^{-1}\chi_\cyc^{k/2}&*\\0&\eta\chi_\cyc^{-\frac{k-2}{2}}\end{pmatrix} \end{equation}
for an unramified character $\eta:G_{\Q_p}\rightarrow A^\times$ with $\eta(\Frob_p)=\alpha_p\in A^\times$; here $\Frob_p$ is an arithmetic Frobenius at $p$ and  $\chi_\cyc:G_{\Q_p}\rightarrow\Z_p^\times$ is the $p$-adic cyclotomic character.

Let $\pi$ be a generator of the maximal ideal of $A$; set $k\defeq A/\pi A$ and
$\overline{T}_\phi\defeq T_\phi\otimes_Ak$. In the list of assumptions below, let $p^*\defeq(-1)^{\frac{p-1}{2}}p$.

\begin{assumption} \label{ass} 
The morphism $\phi: \T_{M,D}(A)\rightarrow A$ satisfies the following conditions: 
\begin{enumerate} 
\item $T_\phi$ is ordinary at $p$;
\item the restriction of $\overline{T}_{\phi} $ to $\Gal\bigl(\bar\Q/\Q(\sqrt{p^*})\bigr)$ is absolutely irreducible;
\item $\overline{T}_{\phi} $ is ramified at all primes $q\,|\,M$ with $q\equiv 1\pmod{p}$;
\item $\overline{T}_{\phi} $ is ramified at all primes $q\,|\,D$ with $q\equiv\pm 1\pmod{p}$; 
\item There is a prime $q\,|\,MD$ such that $\overline{T}_{\phi}$ is ramified at $q$. 
\end{enumerate}
\end{assumption}

Observe that these conditions are imposed also in \cite{LPV}, whose main result is a proof of an analogue for $\rho_\phi$ of Kolyvagin's conjecture on the non-triviality of his $p$-adic system of derived Heegner points on elliptic curves, and are a higher weight counterpart of analogous ones appearing in \cite{Zhang}.

\subsection{Level raising} \label{levelraisingsec}

The goal of this section is to slightly extend the level raising results proved by Wang in \cite[Theorem 2.9]{wang} and \cite[Theorem 2.12]{wang} to the level of generality that is required in this paper. The proofs are variations on those by Wang in \cite{wang} and by Chida in \cite{Chida}; we include them because they will be relevant for the construction of the bipartite Euler system, and they seem not to be exposed in \emph{loc. cit.} to the level of generality required in our paper. Notice that these results are generalizations of results by Ribet (\cite{Ribet-Inv100}) and by Bertolini--Darmon (\cite{BD-IMC}), previously extended by Rajaei (\cite{rajaei}); see also \cite{ChHs2} and \cite{Longo-AIF}.

We start by recalling the notion of an \emph{$n$-admissible prime}, which was originally introduced by Bertolini--Darmon in their work on the Iwasawa main conjecture for elliptic curves over anticyclotomic $\Z_p$-extensions (\cite[p.~18]{BD-IMC}); more precisely, we review its extension to even weight modular forms provided by \cite[Definition 1.1]{ChHs2}. As before, let $K$ be an imaginary quadratic field of discriminant $-D_K$ not divisible by $p$ and assume that the coprime integers $M$ and $D$ satisfy the following condition: all the primes dividing $M$ split in $K$ and all the primes dividing $D$ are inert in $K$.   
 
\begin{definition} \label{admissible-def}
Let $\phi:\T_{M,D}(A)\rightarrow A$ be a morphism of $A$-algebras satisfying Assumption \ref{ass} and let $n\geq1$ be an integer. A prime number $\ell$ is \emph{$n$-admissible for $\phi$} if 
\begin{enumerate}
\item $\ell$ does not divide $MDp$;
\item $\ell$ is inert in $K$;
\item $p$ does not divide $\ell^2-1$;
\item $\pi^n$ divides the ideal generated by $\ell^{\frac{k}{2}}+\ell^{\frac{k-2}{2}}-\epsilon_\ell \phi(T_\ell)$ with $\epsilon_\ell\in\{\pm1\}$.
\end{enumerate}
\end{definition}

We denote by $\mathcal P_n(\phi)$ the set of square-free products of $n$-admissible primes for $\phi$. By convention, $1\in\mathcal P_n(\phi)$.

\begin{definition} Let $\phi:\T_{M,D}(A)\rightarrow A$ be a morphism of $A$-algebras and $S\in\mathcal P_n(\phi)$. A morphism $\xi:\T_{M,DS}(A)\rightarrow A$ of $A$-algebras is \emph{congruent to $\phi$ modulo $\pi^n$} if 
\begin{itemize}
\item $\phi(T_\ell)\equiv \xi(T_\ell)\pmod{\pi^n}$ for all primes $\ell\nmid MDSp$;
\item $\phi(U_\ell)=\xi(U_\ell)\pmod{\pi^n}$ for all primes $\ell\,|\,MD$.
\end{itemize} 
In this case, we write $\phi\equiv \xi\pmod{\pi^n}$. \end{definition} 

The first level raising result we are interested in, which is essentially \cite[Theorem 2.12]{wang}, deals with the definite case.

\begin{theorem}[Wang] \label{levelraising1} 
Suppose that 
\begin{itemize}
\item $A\simeq A/\pi^nA$;
\item $\phi$ satisfies Assumption $\ref{ass}$;
\item $D$ is a product of an odd number of primes.
\end{itemize}
Let $\ell$ be an $n$-admissible prime for $\phi$. There is a morphism of $A$-algebras $\phi_\ell:\T_{M,D\ell}(A)\rightarrow A$ such that $\phi\equiv \phi_\ell\pmod{\pi^n}$ and $\phi_\ell(U_\ell)=\epsilon_\ell \ell^\frac{k-2}{2}$ for all primes $\ell\,|\,S$. 
\end{theorem}

\begin{proof} We mimic the proof of \cite[Theorem 2.12]{wang}, which uses techniques from \cite{Liu}. Let $B=B_{D}$ (respectively, $B'=B_{D\ell}$) denote the definite (respectively, indefinite) quaternion algebra over $\Q$ of discriminant $D$ (respectively, $D\ell$). Attached to $B'$ we have the Shimura curves $\mathfrak{X}_\ell\defeq\mathfrak{X}_0(M,D\ell)$, $\mathfrak{X}_{d,\ell}\defeq\mathfrak{X}_d(M,D\ell)$ and the Kuga--Sato variety 
$\mathfrak{W}_{d,\ell}\defeq\mathfrak{W}_{k,d}(M,D\ell)$, which are all defined over the quadratic unramified extension $\Z_{\ell^2}$ of $\Z_\ell$. Recall that $\mathbf{X}_\ell$, $\mathbf{X}_{d,\ell}$ and $\mathbf{W}_{d,\ell}$ denote the special fibers of
$\mathfrak{X}_{\ell}$, $\mathfrak{X}_{d,\ell}$ and $\mathfrak{W}_{d,\ell}$, respectively, all defined over the field $\F_{\ell^2}$ with $\ell^2$ elements, and that $\overline{\mathbf{X}}_{\ell}$, $\overline{\mathbf{X}}_{d,\ell}$ and $\overline{\mathbf{W}}_{d,\ell}$ denote the base changes of $\mathbf{X}_{\ell}$, $\mathbf{X}_{d,\ell}$ and $\mathbf{W}_{d,\ell}$, respectively, to an algebraic closure $\overline{\F}_\ell$ of $\F_{\ell^2}$. Moreover, recall the sheaf $\mathcal{F}_A$ on $\mathfrak{X}_\ell$ which we denote by $\mathcal{F}$ to simplify our notation; writing $i:\mathbf{X}_\ell\rightarrow\mathfrak{X}_\ell$ for the natural map, we 
obtain a sheaf $i^*\mathcal{F}$, which we also denote by $\mathcal{F}$. Finally, recall the maps $a_0:\overline{\mathbf{X}}_\ell^{(0)}\rightarrow \overline{\mathbf{X}}_\ell$ and $a_1:\overline{\mathbf{X}}_\ell^{(1)}\rightarrow \overline{\mathbf{X}}_\ell$ from \S\ref{sec4.2}; we use the same symbol $\mathcal{F}$ for the restrictions of $\mathcal{F}$ to $\overline{\mathbf{X}}_\ell^{(0)}$ and 
$\overline{\mathbf{X}}_\ell^{(1)}$, so that $a_{i*}\mathcal{F}$ is a sheaf on $\overline{\mathbf{X}}_\ell$ supported on $\overline{\mathbf{X}}^{(i)}_\ell$.  
Then there are isomorphisms of $A$-modules
\begin{equation} \label{RL1}
S_k(M,D; A)^{\oplus 2}\simeq H^0\bigl(\overline{\mathbf{X}}_\ell,a_{0*}\mathcal{F}\bigr)
\end{equation}
and
\begin{equation} \label{RL2}S_k(M\ell,D;A)\simeq H^0\bigl(\overline{\mathbf{X}}_\ell,a_{1*}\mathcal{F}\bigr)
\end{equation}
(\emph{cf.} \cite[p. 2319]{wang}). While isomorphism \eqref{RL2} is equivariant with respect to the action of the Hecke algebra $\T_{M,D\ell}$, isomorphism \eqref{RL1} is equivariant for the action of Hecke operators outside $\ell$ but 
on $S_k(M,D; A)^{\oplus 2}$ there is an action of $T_\ell$ while on 
$H^0(\overline{\mathbf{X}}_\ell,a_{0*}\mathcal{F})$ there is an action of $U_\ell$, which we need to relate. As observed in \cite[p. 2319]{wang},
Poincar\'e duality induces an isomorphisms of $A$-modules 
\begin{equation} \label{RL3}
S_k(M,D; A)^{\oplus2}\simeq H^2\bigl(\overline{\mathbf{X}}_\ell,a_{0*}\mathcal{F}(1)\bigr).
\end{equation}
The composition 
\begin{equation} \label{rhotau}
H^0\bigl(\overline{\mathbf{X}}_\ell,a_{0*}\mathcal{F}\bigr)\overset{\rho}\longrightarrow H^0\bigl(\overline{\mathbf{X}}_\ell,a_{1*}\mathcal{F}\bigr)\overset{\tau}\longrightarrow H^2\bigl(\overline{\mathbf{X}}_\ell,a_{0*}\mathcal{F}(1)\bigr),
\end{equation}
where $\rho$ is restriction and $\tau$ is the Gysin map (see, \emph{e.g.}, \cite[p. 2309]{wang}), corresponds under isomorphisms \eqref{RL1} and \eqref{RL3} to the action of the matrix 
\[ \nabla\defeq\begin{pmatrix}-\ell^\frac{k-2}{2}(\ell+1)&T_\ell\\T_\ell&-\ell^\frac{k-2}{2}(\ell+1)\end{pmatrix} \]
on $S_k(M,D; A)^{\oplus 2}$. The action of the Hecke operator $U_\ell$ on $H^0\bigl(\overline{\mathbf{X}}_\ell,a_{0*}\mathcal{F}\bigr)$ is given by the rule $(x,y)\mapsto\bigl(-\ell^\frac{k}{2}y,\ell^\frac{k-2}{2}x+T_\ell y\bigr)$. A direct computation (see, \emph{e.g.}, \cite[p. 2320]{wang}) shows that 
\begin{equation} \label{RL7bis}
\mathrm{coker}(\nabla)=\frac{S_k(M,D;A)^{\oplus 2}}{\Bigl(\mathcal{I}^{(\ell)}_\phi,U_\ell^2-\ell^{k-2}\Bigr)},
\end{equation}
where $\mathcal{I}^{(\ell)}_\phi$ is the ideal of $\T_{M,D}(A)$ generated by all $(\phi-1)\cdot T$ for $T\neq T_\ell$. The definition of admissible primes implies that there is an isomorphism of $A$-modules
\begin{equation} \label{RL8bis}
\frac{S_k(M,D;A)^{\oplus 2}}{\Bigl(I_{\ell}^{(\ell)},U_\ell^2-\ell^{k-2}\Bigr)}\overset\simeq\longrightarrow
\frac{S_k(M,D;A)^{\oplus 2}}{\Bigl(I_{\ell}^{(\ell)},U_\ell-\epsilon_\ell\ell^\frac{k-2}{2}\Bigr)},
\end{equation}
and the right-hand side in \eqref{RL8bis} turns out to be free of rank $1$ (see, \emph{e.g.}, \cite[p. 2320]{wang}). The action of $\T_{M,D\ell}(A)$ on this $A$-module then gives the searched-for morphism $\phi_{\ell}$. 
\end{proof}

\begin{remark} \label{rem1}
Theorem \ref{levelraising1} is also obtained by Chida in \cite[Theorem 5.11]{Chida}. The proofs by Wang, which we sketched above, and by Chida are different: Wang works directly with the weight spectral sequence of the Kuga--Sato variety over $\mathcal{X}_d(M,D\ell)$, whereas Chida's proof follows closely the arguments in the proof of \cite[Theorem 5.15]{BD-IMC} (which, in turn, is inspired by the arguments in \cite[Section 3]{Ribet-Inv100}) and proceeds with an analysis of the character and cocharacter groups (in the sense of \cite{rajaei}) of $\mathcal{X}_d(Mm,D/m)$ with coefficients in $\mathcal{F}=\mathcal{F}_A$ for a prime divisor $m$ of $D$ (such a divisor exists because, by assumption, $D$ is a product of an odd number of primes). The proofs by Wang and by Chida of Theorem \ref{levelraising1} give also slightly different proofs of the first reciprocity law (which corresponds to Theorem \ref{ERL1}); \emph{cf.} also Remark \ref{rem2}. In this paper, we find it more convenient to follow Wang's strategy, but we wish to remark that there are no clear advantages in any of the two choices over the other, except perhaps for the fact that Wang's approach, which builds on work of Liu (\cite{Liu}), offers a shorter path to the desired result.  
\end{remark}

The second level raising result, which corresponds to \cite[Theorem 2.9]{wang}, takes care of the indefinite case.

\begin{theorem}[Wang] \label{levelraising2} 
Suppose that
\begin{itemize}
\item $A\simeq A/\pi^nA$;
\item $\phi$ satisfies Assumption $\ref{ass}$;
\item $D$ is a product of an even number of primes.
\end{itemize}
Let $\ell$ be an $n$-admissible prime for $\phi$. There is a morphism of $A$-algebras $\phi_\ell:\T_{M,D\ell}(A)\rightarrow A$ such that $\phi\equiv \phi_\ell\pmod{\pi^n}$ and $\phi_\ell(U_\ell)=\epsilon_\ell \ell^\frac{k-2}{2}$ for all primes $\ell\,|\,S$. 
\end{theorem}

\begin{proof} We divide the proof into three steps. 

\subsubsection*{Step 1: Construction of the map $\gamma$}

Consider the Shimura curves $\mathfrak{X}=\mathfrak{X}(M,D)$ and  $\mathfrak{X}_d=\mathfrak{X}_d(M,D)$ and let
$\mathfrak{W}_d=\mathfrak{W}_{k,d}(M,D)$ be the Kuga--Sato variety over $\mathfrak{X}_d$, 
all defined over the quadratic unramified extension $\Z_{\ell^2}$ of $\Z_\ell$. Denote by $\mathbf{W}_d$, $\mathbf{X}_d$, $\mathbf{X}$ the special fibers of
$\mathfrak{W}_d$, $\mathfrak{X}_d$, $\mathfrak{X}$, defined over the field $\F_{\ell^2}$ with $\ell^2$ elements; let also  
$\overline{\mathbf{W}}_d$, $\overline{\mathbf{X}}_d$, $\overline{\mathbf{X}}$ denote the base change of $\mathbf{W}_d$, $\mathbf{X}_d$, $\mathbf{X}$ to the algebraic closure $\overline{\F}_\ell$ of $\F_{\ell^2}$. Let $\pi:\overline{\mathbf{W}}_d\rightarrow \overline{\mathbf{X}}_d$ be the canonical projection, write $\overline{\mathbf{X}}_d^\mathrm{ss}$ for the set of supersingular points in $\overline{\mathbf{X}}_d$, and define $\overline{\mathbf{W}}_d^\mathrm{ss}=\pi^{-1}(\overline{\mathbf{X}}_d^\mathrm{ss})$. Set also $\overline{\mathbf{W}}_d^\mathrm{ord}\defeq\overline{\mathbf{W}}_d\smallsetminus\overline{\mathbf{W}}_d^\mathrm{ss}$. Let us consider the exact sequence of cohomology with compact support
\[
\dots\longrightarrow 
H^{k-1}\bigl(\overline{\mathbf{W}}_d,A(k/2)\bigr)
\longrightarrow H^{k-1}\bigl(\overline{\mathbf{W}}_d^\mathrm{ord},A(k/2)\bigr)\longrightarrow
H^{k}_{\overline{\mathbf{W}}^\mathrm{ss}_d}\bigl(\overline{\mathbf{W}}_d,A(k/2)\bigr)\longrightarrow\cdots.
\]
Observe that, by purity, there is an isomorphism 
\[ H^{k}_{\overline{\mathbf{W}}_d^\mathrm{ss}}\bigl(\overline{\mathbf{W}}_d,A(k/2)\bigr)\simeq 
\bigoplus_{x\in \overline{\mathbf{X}}_d^\mathrm{ss}}H^{k-2}\bigl(\overline{\mathbf{W}}_{d,x},A(k/2)\bigr), 
\] 
where $\overline{\mathbf{W}}_{d,x}\defeq\pi^{-1}(x)$ is the fiber of $\pi$ over $x$. Moreover, applying the projector $\epsilon_k$ gives an isomorphism
\[ \epsilon_kH^{k-2}\bigl(\overline{\mathbf{W}}_{d,x},A(k/2)\bigr)\overset\simeq\longrightarrow L_r(A). \]
Since ${\mathfrak{W}}_d$ has good reduction at $\ell$, it follows from 
\eqref{projectors} that $\epsilon_kH^j\bigl(\overline{\mathbf{W}}_{k,d},A(k/2)\bigr)=0$ for $j\neq k-1$; furthermore, since also ${\mathfrak{X}}_d$ has good reduction at $\ell$, there is an isomorphism of $A$-modules $\epsilon_kH^{k-1}\bigl(\overline{\mathbf{W}}_{k,d},A(k/2)\bigr)\simeq H^1\bigl(\overline{\mathbf{X}}_d,\mathcal{F}(1)\bigr)$ where, as before, $\mathcal{F}=\mathcal{F}_A$ is viewed as a sheaf on $\overline{\mathbf{X}}_d$. Thus, setting $\overline{\mathbf{X}}_d^\mathrm{ord}\defeq\overline{\mathbf{X}}_d\smallsetminus\overline{\mathbf{X}}_d^\mathrm{ss}$, we get a short exact sequence 
\[ 0\longrightarrow H^1\bigl(\overline{\mathbf{X}}_d,\mathcal{F}(1)\bigr)\longrightarrow
H^1\bigl(\overline{\mathbf{X}}^\mathrm{ord}_d,\mathcal{F}(1)\bigr)\longrightarrow
H^0\bigl(\overline{\mathbf{X}}^\mathrm{ss}_d,\mathcal{F}\bigr)\longrightarrow 0. \]
Taking $G_{\F_{\!\ell^2}}$-cohomology, we obtain a connecting homomorphism 
\[ H^0\bigl(\overline{\mathbf{X}}^\mathrm{ss}_d,\mathcal{F}\bigr)^{G_{\F_{\!\ell^2}}}\longrightarrow H^1\Bigl(\F_{\ell^2},H^1\bigl(\overline{\mathbf{X}}_d,\mathcal{F}(1)\bigr)\!\Bigr).
\] 
Supersingular elliptic curves have models defined over $\F_{\ell^2}$, so $H^0\bigl(\overline{\mathbf{X}}^\mathrm{ss}_d,\mathcal{F}\bigr)^{G_{\F_{\!\ell^2}}}=H^0({\mathbf{X}}^\mathrm{ss}_d,\mathcal{F})$. Furthermore, using the projector $\epsilon_d$ we finally obtain a homomorphism 
\begin{equation} \label{RL3+1/2}
\gamma: 
H^0({\mathbf{X}}^\mathrm{ss},\mathcal{F} )\longrightarrow H^1\Bigl(\F_{\ell^2},H^1\bigl(\overline{\mathbf{X}},\mathcal{F}(1)\bigr)\!\Bigr),
\end{equation}
where $\mathcal{F}$ is viewed as a sheaf on $\overline{\mathbf{X}}$.

\subsubsection*{Step 2: Surjectivity of $\bar\gamma$}

The crucial input of the proof is that the cokernel of the map $\gamma$ in \eqref{RL3+1/2} is Eisenstein. This is shown in \cite[pages 2315-2316]{wang} by appealing 
to Ihara's Lemma, which we review. 
Consider the Shimura curve $\mathfrak{X}_\ell=\mathfrak{X}(M\ell,D)$ (defined over $\Z_{\ell^2})$, its special fiber $\mathbf{X}_\ell$ and the base change $\overline{\mathbf{X}}_\ell$ of ${\mathbf{X}}_\ell$ to $\overline{\F}_\ell$. Set ${X}\defeq\mathfrak{X}\otimes_{\Z_{\ell^2}}\Q_{\ell^2}$, ${X}_\ell\defeq\mathfrak{X}_\ell\otimes_{\Z_{\ell^2}}\Q_{\ell^2}$,
$\overline{{X}}\defeq{X}\otimes_{\Q_{\ell^2}}\overline{\Q}_\ell$, $\overline{{X}}_\ell\defeq{X}_\ell\otimes_{\Z_{\ell^2}}\overline{\Q}_\ell$.
We use the same symbol $\mathcal{F}=\mathcal{F}_A$ for the corresponding sheaf on $\mathfrak{X}_\ell$ and denote by $\widetilde{\mathfrak{W}}_\ell$ the semistable model
of $\mathfrak{W}_\ell\defeq\mathfrak{W}_{k,d}(M\ell,D)$ constructed in \cite[Lemma 7.1]{Liu}. 
Using results of Liu (\cite{Liu}), Wang describes in \cite[\S2.4]{wang} the 
weight spectral sequence for $H^1(\overline{\mathbf{X}}_\ell,\mathcal{F})$ in terms of the weight spectral sequence  
\[ H^{p+q}\bigl({\widetilde{\mathbf{W}}}_\ell\otimes_{\F_{\ell^2}}\overline{\F}_\ell,
\mathrm{gr}_{-p}\R\Psi(A)\bigr)\Longrightarrow H^{p+q}\bigl({\widetilde{W}}_\ell\otimes_{\Q_{\ell^2}}\overline{\Q}_{\ell},A\bigr) \] 
of $\widetilde{\mathfrak{W}}_\ell$ with trivial coefficients, ${\widetilde{\mathbf{W}}}_\ell$ (respectively, ${\widetilde{W}}_\ell$) being the special (respectively, generic) fiber of $\widetilde{\mathfrak{W}}_\ell$ and $\mathrm{gr}_{\bullet}\R\Psi(A)$ denotes the graded pieces of the monodromy filtration 
$\Fil^\bullet\R\Psi(A)$ of the complex of sheaves of nearby cycles $\R\Psi(A)$ of the constant sheaf $A$ over $\widetilde{\mathfrak{W}}_\ell$.
If we denote $\R\Phi(\mathcal{F})$ the complex of vanishing cycles of the sheaf $\mathcal{F}$ on $\mathfrak{X}_\ell$ (which is concentrated in degree 1 and supported on the set $\mathbf{X}^\mathrm{ss}$ of supersingular points of $\mathbf{X}$), then there is an isomorphism 
\[ 
\bigoplus_{x\in \mathbf{X}^\mathrm{ss}}\bigl(\R^1\Phi(\mathcal{F})(1)\bigr)_x\simeq H^0\bigl(\overline{\mathbf{X}}_\ell,a_{1*}\mathcal{F}\bigr)=\bigoplus_{x\in\mathbf{X}^\mathrm{ss}}L_k(A)
\] 
and the monodromy filtration on $H^1(\overline{X}_\ell,\mathcal{F})(1)$ is given by 
the short exact sequence 
\begin{equation}\label{MF}
0\longrightarrow H^1\bigl(\overline{\mathbf{X}}_\ell,a_{0*}\mathcal{F}\bigr)(1)\longrightarrow
H^1\bigl(\overline{X}_\ell,\mathcal{F}\bigr)(1)\longrightarrow X(\mathcal{F})\longrightarrow 0,
\end{equation}
where 
\[ X(\mathcal{F})\defeq\ker\Bigl(H^0\bigl(\overline{\mathbf{X}}_\ell,a_{1*}\mathcal{F}\bigr)\overset{\tau}\longrightarrow H^2\bigl(\overline{\mathbf{X}}_\ell,a_{0*}\mathcal{F}\bigr)\!\Bigr) \]
and $\tau$ is the Gysin map. Now consider the map 
\begin{equation} \label{MF1}
H^1\bigl(\overline{\mathbf{X}},\mathcal{F}\bigr)(1)^{\oplus2}\xrightarrow{(i_{1*},i_{2*})} H^1\bigl(\overline{\mathbf{X}}_\ell,a_{0*}\mathcal{F}\bigr)(1),
\end{equation}
 where $i_1:\mathbf{X}\rightarrow\mathbf{X}_\ell$ and $i_2:\mathbf{X}\rightarrow\mathbf{X}_\ell$ are closed immersions satisfying the matrix equality
\[ \begin{pmatrix}\pi_1\circ i_1&\pi_1\circ i_2\\ \pi_2\circ i_2&\pi_2\circ i_2\end{pmatrix}=\begin{pmatrix}\mathrm{id}&\Frob_\ell\\ \langle\ell\rangle^{-1}\Frob_\ell&\mathrm{id}\end{pmatrix} \] 
and $\langle\ell\rangle$ is the diamond operator at $\ell$ (\emph{cf.} \cite[p. 2308]{wang}).
We obtain a commutative diagram with exact rows
\begin{equation} \label{diagram}
\xymatrix{H^1\bigl(\overline{\mathbf{X}},\mathcal{F}\bigr)(1)^{\oplus2}\ar[r]\ar[d]^-\simeq & H^1\bigl(\overline{{X}}_\ell,\mathcal{F}\bigr)(1)
\ar[r] \ar[d]^-{(\pi_{1*},\pi_{2*})}& X(\mathcal{F})\ar[d]\ar[r]&0\\
H^1\bigl(\overline{{X}},\mathcal{F}(1)\bigr)^{\oplus2}\ar[r]^-\nabla & H^1\bigl(\overline{{X}},\mathcal{F}(1)\bigr)^2\ar[r]&\coker(\nabla)\ar[r]&0}
\end{equation}
in which
\begin{itemize}
\item the top horizontal line is obtained by composing \eqref{MF} and \eqref{MF1};
\item the leftmost vertical arrow is the canonical isomorphism;
\item $\pi_1$, $\pi_2$ are the natural degeneracy maps, ordered in such a way that $\pi_1$ corresponds to forgetting the isogeny of degree $\ell$ in the moduli problem;     
\item the rightmost vertical arrow is induced by $(\pi_{1*},\pi_{2*})$; 
\item $\nabla$ makes \eqref{diagram} commutative and $\coker(\nabla)\simeq H^1\Bigl(\F_{\ell^2},H^1\bigl(\overline{\mathbf{X}},\mathcal{L}_k(A)(1)\bigr)\!\Bigr)$. 
\end{itemize}
The cokernel of the map $(\pi_{1*},\pi_{2*})$ in \eqref{diagram}
is Eisenstein: this is due to Ihara's Lemma (\cite[Lemma 3.2]{Ihara}) when $D=1$ and to its generalization to the quaternionic case in \cite[Theorem 4]{DT} (see \cite[Theorem 2.5]{wang}; note also that in this passage the condition $p>k+1$ is required). Moreover, 
$H^2\bigl(\overline{\mathbf{X}}_\ell,a_{0*}\mathcal{F}\bigr)$ is also Eisenstein (\cite[p. 2312]{wang}). Now, the maps in \eqref{diagram} are equivariant for the action of the Hecke operators $T$ of $\T_{M,D}$ and of $\T_{M\ell,D}$ for $T\neq U_\ell$ (here we slightly abuse notation and adopt the same symbol $T$ for the Hecke operators $T_q$ or $U_q$, $q\neq \ell$, belonging to different Hecke algebras); this follows from \cite[Lemma 9.1]{BD-IMC} and standard properties of Hecke operators. Thus, if we set $\mathcal{I}_\phi\defeq\ker(\phi)$ 
and denote by $\mathcal{I}_\phi^{(\ell)}$ the ideal of $\T_{M\ell,D}(A)$ generated by $T-\phi(T)$ for all $T\neq U_\ell$, then the rightmost vertical arrow in \eqref{diagram} induces a surjection 
\begin{equation} \label{RL10}
\bar{\gamma}: H^0(\mathbf{X}^\mathrm{ss},\mathcal{F})\big/\mathcal{I}^{(\ell)}_\phi\longepi H^1\Bigl(\F_{\ell^2},H^1\bigl(\overline{\mathbf{X}},\mathcal{F}(1)\bigr)\big/\mathcal{I}_\phi\Bigr).
\end{equation}
As the notation suggests, $\bar\gamma$ is the map $\gamma$ in \eqref{RL3+1/2} composed with the canonical projections: this is detailed in \cite[p. 2316]{wang}. It follows from a result of Ribet (\cite[Proposition 3.8]{Ribet-Inv100}) that the action of the Frobenius automorphism on $H^0\bigl(\overline{\mathbf{X}}^\mathrm{ss},\mathcal{F}\bigr)$ is by $U_\ell\big/\ell^\frac{k-2}{2}$. On the other hand, it follows from \eqref{RL5} that Frobenius acts on the right-hand side of \eqref{RL10} by $\epsilon_\ell$; we conclude that $\bar\gamma(U_\ell x)=\epsilon_\ell \ell^\frac{k-2}{2}\bar\gamma(x)$ for all $x\in H^0({\mathbf{X}}^\mathrm{ss},\mathcal{F})$. Therefore, we obtain a surjection 
\begin{equation}\label{bargamma}
\bar\gamma:\frac{H^0(\mathbf{X}^\mathrm{ss},\mathcal{F})}{\bigl(\mathcal{I}^{(\ell)}_\phi,U_\ell-\epsilon_\ell\ell^{\frac{k-2}{2}}\bigr)}\longepi
H^1\Bigl(\F_{\ell^2},H^1\bigl(\overline{\mathbf{X}},\mathcal{F}(1)\bigr)\big/\mathcal{I}_\phi\Bigr).
\end{equation}

\subsubsection*{Step 3: Multiplicity one} 

As a preliminary observation, note that, since $\ell$ is $n$-admissible, we have $T_\phi\simeq A(1)\oplus A$ as $G_{\Q_{\ell^2}}$-modules; it follows that there are isomorphisms 
\begin{equation} \label{RL5}
H^1(\F_{\ell^2},T_\phi)\simeq H^1\bigl(\F_{\ell^2},A(1)\oplus A\bigr)\simeq A.
\end{equation}
Let $\mathfrak{m}_\phi$ be the maximal ideal of $\T_{M,D}\otimes_\Z A$ containing $\mathcal{I}_\phi$ and set $\overline{T}_\phi\defeq T_\phi/\mathfrak{m}_\phi$. By the Brauer--Nesbitt theorem, there is an isomorphism
\begin{equation} \label{BN}
H^1\bigl(\overline{{X}},\mathcal{F}(1)\bigr)\big/\mathfrak{m}_\phi\simeq \overline{T}_\phi^r \end{equation} 
for some integer $r\geq1$. Put $k\defeq A/\pi A$. Fixing \emph{any} of these $r$ components and using \eqref{RL5}, we obtain from \eqref{bargamma} a surjective map $\tilde\gamma$ given by the composition
\begin{equation}\label{bargamma1}
\xymatrix@C=30pt{\displaystyle{\frac{H^0(\mathbf{X}^\mathrm{ss},\mathcal{F})}{\bigl(\mathcal{I}^{(\ell)}_\phi,U_\ell-\epsilon_\ell\ell^{\frac{k-2}{2}}\bigr)}}\ar@{->>}[r]^-{\bar\gamma}\ar@{-->>}[rrdd]^-{\tilde\gamma}& H^1\Bigl(\F_{\ell^2},H^1\bigl(\overline{\mathbf{X}},\mathcal{F}(1)\bigr)\big/\mathcal{I}_\phi\Bigr)\ar@{->>}[r]&H^1\Bigl(\F_{\ell^2},H^1\bigl(\overline{\mathbf{X}},\mathcal{F}(1)\bigr)\big/\mathfrak{m}_\phi\Bigr)\ar@{->>}[d]^-{\eqref{BN}}\\
&&H^1\bigl(\F_{\ell^2},\overline{T}_\phi\bigr)\ar[d]_\simeq^-{\eqref{RL5}}\\
&&k,}
\end{equation}
where the right horizontal arrow is induced by reduction modulo $\mathfrak m_\phi$. Then one gets an isomorphism 
\[ H^0(\mathbf{X}^\mathrm{ss},\mathcal{F})/\ker(\tilde\gamma)\overset\simeq\longrightarrow k \] 
that, using the isomorphism of $k$-vector spaces $S_k(M,D\ell;A)\simeq H^0(\mathbf{X}^\mathrm{ss},\mathcal{F})$, corresponds to a modular form $\bar\phi_\ell$ as in the statement when $n=1$. This concludes the proof for $n=1$. 

Let $\xi_\ell:\T_{M\ell,D}(\cO)\rightarrow\cO$ be a lift of $\bar\phi_\ell$, where $\cO$ is the valuation ring of a finite extension of $\Q_\ell$, whose maximal ideal will be denoted by $\mathfrak{m}_\cO$, such that $\cO/\mathfrak{m}_\cO\simeq A/\pi A$: such a lift exists by \cite[Lemma 6.1]{DS}. Thanks to \cite[Proposition 6.1]{Chida} (see also \cite[Proposition 6.8]{ChHs1}), the localization $S_k(M\ell,D;\cO)_{\mathfrak{m}_{\xi_\ell}}$ of $S_k(M\ell,D;\cO)$ at the maximal ideal $\mathfrak{m}_{\xi_\ell}$ of $\T_{M\ell,D}(\cO)$ containing the kernel of $\xi_\ell$ is $\T_{\mathfrak{m}_{\xi_\ell}}$-cyclic, where $\T_{\mathfrak{m}_{\xi_\ell}}$ is the localization of $\T_{M\ell,D}(\cO)$ at $\mathfrak{m}_{\xi_\ell}$. By Nakayama's lemma, we conclude that there is an isomorphism
\begin{equation} \label{RL6} 
H^0(\mathbf{X}^\mathrm{ss},\mathcal{F})/\mathfrak{m}_{\xi_\ell}\simeq k.
\end{equation}
Let us consider the ideal $\mathcal{I}_{\phi_\ell}\defeq\bigl(\mathcal{I}^{(\ell)}_\phi,U_\ell-\epsilon_\ell\ell^{\frac{k-2}{2}}\bigr)$ of $\T_{M\ell,D}(A)$. Since $\mathfrak{m}_{\xi_\ell}$ is the maximal ideal containing $\mathcal{I}_{\phi_\ell}$, and $\pi^n\in \mathcal{I}_{\phi_\ell}$, it follows from \eqref{RL6} and Nakayama's lemma that there is a surjection 
\begin{equation} \label{RL7}
A/\pi^nA\longepi H^0(\mathbf{X}^\mathrm{ss},\mathcal{F})/\mathcal{I}_{\phi_\ell}.
\end{equation}
Moreover, it also follows from \eqref{RL6} that $r=1$ in \eqref{BN}, so, since $T_\phi$ is a subquotient of $H^1\bigl(\overline{X},\mathcal{F}\bigr)(1)/\mathcal{I}_\phi$, Nakayama's lemma shows that 
\begin{equation} \label{RL8}
H^1(\overline{X},\mathcal{F})(1)/\mathcal{I}_\phi\simeq T_\phi.
\end{equation}
Finally, combining \eqref{RL5}, \eqref{RL7} and \eqref{RL8} yields an isomorphism 
\begin{equation} \label{RL11}
\frac{H^0(\mathbf{X}^\mathrm{ss},\mathcal{F})}{\bigl(\mathcal{I}^{(\ell)}_\phi,U_\ell-\epsilon_\ell\ell^{\frac{k-2}{2}}\bigr)}\simeq A,
\end{equation}
and the result we are looking for follows from the isomorphism of $A$-modules $S_k(M,D\ell;A)\simeq H^0(\mathbf{X}^\mathrm{ss},\mathcal{F})$. \end{proof}

From here on, we write $\mathcal P_n(\phi)=P_n(\phi,\pi,K)$ for the set of square-free products of $n$-admissible primes relative to $(\phi,\pi,K)$, with the convention that $1\in\mathcal P_n(\phi)$. Decompose $\mathcal P_n(\phi)$ as a disjoint union 
\[ \mathcal{P} _n(\phi)=\mathcal{P}_n ^\mathrm{def}(\phi)\cup\mathcal{P}_n ^\mathrm{indef}(\phi) \] 
by requiring that $S$ belong to $\mathcal{P}^\mathrm{def}_n(\phi)$ (respectively, $\mathcal{P}_n^\mathrm{indef}(\phi)$) if and only if $\mu(DS)=-1$ (respectively, $\mu(DS)=1$), where $\mu$ is the usual M\"obius function.

\begin{theorem} \label{levelraising} 
Let $\phi:\T_{M,D}(A)\rightarrow A$ satisfy Assumption $\ref{ass}$ and let $S\in \mathcal{P}_n(\phi)$. There exists $\phi_S:\T_{M,DS}(A)\rightarrow A/\pi^nA$ such that $\phi\equiv \phi_S\pmod{\pi^n}$ and $\phi_S(U_\ell)=\epsilon_\ell \ell^\frac{k-2}{2}$ for all primes $\ell\,|\,S$. 
\end{theorem}

\begin{proof} Consider the reduction $\bar\phi:\T_{M,D}(A)\rightarrow A/\pi^nA$ of $\phi$ modulo $\pi^n$. Pick a prime $\ell\,|\,S$. If $\ell\in\mathcal{P}_n^\mathrm{indef}(\phi)$, then apply Theorem \ref{levelraising1}, whereas if $\ell\in\mathcal{P}_n^\mathrm{def}(\phi)$, then apply Theorem \ref{levelraising2}: in both cases, we obtain a quaternionic eigenform $\phi_\ell:\T_{M,D\ell}(A)\rightarrow A/\pi^nA$ as in the statement, which satisfies Assumption \ref{ass} as well (all conditions are clearly the same except those at $\ell$, but $\ell\not\equiv \pm 1\pmod{p}$ since $\ell$ is $n$-admissible). We may now apply the same procedure to a prime $\ell'\,|\,(S/\ell)$ (observe that $\ell'\in\mathcal{P}_n^\mathrm{def}(\phi_\ell)$ if and only if $\ell\in\mathcal{P}_n^\mathrm{indef}(\phi)$ and $\ell'\in\mathcal{P}_n^\mathrm{indef}(\phi_\ell)$ if and only if $\ell\in\mathcal{P}_n^\mathrm{def}(\phi)$), obtaining a quaternionic eigenform $\phi_{\ell\ell'}:\T_{M,D\ell\ell'}(A)\rightarrow A/\pi^nA$ as in the statement, which also satisfies Assumption \ref{ass}. Applying recursively this procedure for all the primes dividing $S$ concludes the proof. \end{proof}

\subsection{Selmer groups} \label{Selmer-subsec}

In this subsection, $A$ denotes the valuation ring of a finite extension $F$ of $\Q_p$. Fix $\phi:\T_{M,D}(A)\rightarrow A$ and $S\in \mathcal{P}_n(\phi)$ as in Theorem \ref{levelraising}. For each prime $v$ of $K_m$, let $K_{m,v}$ be the completion of $K_m$ at $v$, put $G_{m,v}\defeq\Gal(\overline{K}_{m,v}/K_{m,v})$ and write $I_{m,v}$ for the inertia subgroup of $G_{m,v}$. Finally, $\ell$ is the residual characteristic at $v$.

For any prime $v$ of $K_m$, set
\[ 
H^1_\unr(K_{m,v},{T}_\phi )\defeq H^1\bigl(K_{m,v}^\unr/K_v,T_\phi^{I_{m,v}}\bigr)=\ker\Bigl(H^1(K_{m,v},{T}_\phi)\longrightarrow H^1\bigl(I_{m,v}, {T}_\phi)\Bigr).
\]
If $T_\phi$ is unramified at $v$, then set also 
\[ 
H^1_\mathrm{fin}(K_{m,v},T_\phi)\defeq H^1_\unr(K_{m,v},{T}_\phi)
\] 
and define $H^1_\mathrm{sing}(K_{m,v},{T}_\phi)$ via the exactness of the sequence 
\[ 0\longrightarrow H^1_\mathrm{fin}(K_{m,v}, {T}_\phi )\longrightarrow H^1(K_{m,v}, {T}_\phi)\longrightarrow H^1_\mathrm{sing}(K_{m,v}, {T}_\phi )\longrightarrow 0. \] 
Let $\loc_v:H^1(K_m, {T}_\phi)\rightarrow H^1(K_{m,v}, {T}_\phi)$ be the localization (\emph{i.e.}, restriction) map and if $T$ is unramified at $v$, then consider the composition 
\[ \partial_v:H^1(K_m,{T}_\phi)\longrightarrow H^1_\mathrm{sing}(K_{m,v},{T}_\phi) \] 
of $\loc_v$ and the canonical quotient map. If $\partial_\ell(x)=0$, then we write $v_\ell(x)\in H^1_\mathrm{fin}(K_{m,v},{T}_\phi )$ for the image of $x$ under restriction. 

If $v\,|\,p$, then recall the filtration \eqref{ordinaryfil} and set 
\[ H^1_\ord(K_{m,v},{T}_\phi)\defeq\im\Bigl(H^1(K_{m,v},{T}_\phi^{(p)})\longrightarrow H^1(K_{m,v},{T}_\phi)\Bigr). \] 
Now let $v\,|\,D$. As explained, \emph{e.g.}, in \cite{carayol}, there is a unique line $T_{\phi}^{(\ell)}\subset T_\phi$ on which $G_{\Q_\ell}$ acts either by $\chi_\cyc$ or by $\chi_\cyc\tau_K$, where $\tau_K$ is the non-trivial unramified quadratic character of $G_{\Q_\ell}$. Define $T_\phi^{(v)}$ to be the restriction of $T_\phi^{(\ell)}$ at $v$ and $T_{\phi}^{[v]}$ by the exactness of the sequence 
\[ 
0\longrightarrow T_{\phi}^{(v)}\longrightarrow T_\phi\longrightarrow T_{\phi}^{[v]}\longrightarrow 0. 
\]
Then set
\[ 
H^1_\ord(K_{m,v}, {T}_\phi)\defeq\im\Bigl(H^1(K_{m,v}, {T}_{\phi}^{(v)})\longrightarrow H^1(K_{m,v}, {T}_{\phi,})\Bigr). 
\] 
If $v\,|\,S$, then $v=\ell\cO_K$ for a prime number $\ell$ that is inert in $K$. In this case, $T_{\phi,n}\defeq T_\phi/\pi^nT_\phi$ contains a unique free $A/\pi^nA$-submodule $T_{\phi,n}^{(\ell)}$ of rank $1$ such that $G_v$ acts on $T_{\phi,n}^{(\ell)}$ as multiplication by $\epsilon_\ell\ell$ and on the quotient $T_{\phi,n}^{[\ell]}\defeq T_{\phi,n}/T_{\phi,n}^{(\ell)}$, which is free of rank $1$ over $A/\pi^nA$, as multiplication by $\epsilon_\ell$. Set
\[ H^1_\mathrm{sing}(K_{m,v},{T}_{\phi,n})\defeq \im\Bigl(H^1(K_{m,v}, {T}_{\phi,n}^{(\ell)})\longrightarrow H^1(K_{m,v}, {T}_{\phi,n})\Bigr).\]
It follows that there is a splitting 
\begin{equation} \label{split}
H^1(K_{m,v}, {T}_{\phi,n})\simeq H^1_\mathrm{fin}(K_{m,v}, {T}_{\phi,n})\oplus H^1_\mathrm{sing}(K_{m,v}, {T}_{\phi,n})
\end{equation}
(see, \emph{e.g.}, \cite[Lemma 1.5]{ChHs2}).

\begin{definition} \label{Selmer} 
Let $\phi:\T_{M,D}(A)\rightarrow A$ be an $A$-algebra homomorphism and let $S\in\mathcal{P}_n$.
\begin{enumerate}
\item 
The \emph{Selmer group} $\Sel(K_m, {T}_\phi)$ is the group of all $s\in H^1(K_m,T_\phi)$ such that 
\begin{itemize}
\item $\loc_v(s)\in H^1_\mathrm{fin}(K_{m,v}, {T}_\phi)$ for all $v\nmid Dp$;
\item $\loc_v(s)\in H^1_\mathrm{ord}(K_{m,v}, {T}_\phi)$ for all $v\,|\,Dp$.
%\item $\loc_v(s)\in H^1_\mathrm{sing}(K_{m,v}, {T}_\phi)$ for all $v\,|\,S$.
\end{itemize}
\item The \emph{Selmer group} $\Sel_S(K_m, {T}_{\phi,n})$ is the group of all $s\in H^1(K_m,T_{\phi,n})$ such that 
\begin{itemize}
\item $\loc_v(s)\in H^1_\mathrm{fin}(K_{m,v}, {T}_{\phi,n})$ for all $v\nmid DSp$;
\item $\loc_v(s)\in H^1_\mathrm{ord}(K_{m,v}, {T}_{\phi,n})$ for all $v\,|\,Dp$;
\item $\loc_v(s)\in H^1_\mathrm{sing}(K_{m,v}, {T}_{\phi,n})$ for all $v\,|\,S$.
\end{itemize}
\end{enumerate}
\end{definition}
We also introduce Iwasawa-theoretic Galois cohomology groups
\begin{equation} \label{Iwasawa-eq}
H^1_\mathrm{Iw}(K_\infty,T_\phi)\defeq\invlim_m H^1(K_m,T_\phi),\quad H^1_\mathrm{Iw}(K_\infty,T_{\phi,n})\defeq\invlim_m H^1(K_m,T_{\phi,n}), 
\end{equation}
the inverse limits being taken with respect to the corestriction maps. Correspondingly, with notation as in Definition \ref{Selmer}, there are Selmer groups
\begin{equation} \label{Selmer-Iwasawa-eq}
\Sel(K_\infty,T_\phi)\defeq\invlim_m\Sel(K_m,T_\phi),\quad\Sel_S(K_\infty,T_{\phi,n})\defeq\invlim_m\Sel_S(K_m,T_{\phi,n}) 
\end{equation}
that sit inside $H^1_\mathrm{Iw}(K_\infty,T_\phi)$ and $H^1_\mathrm{Iw}(K_\infty,T_{\phi,n})$, respectively.

Now let us set $A_\phi\defeq T_\phi\otimes_A(F/A)$ and $A_{\phi,n}\defeq A_\phi[\pi^n]$ for each integer $n\geq1$. There are canonical $G_\Q$-equivariant isomorphisms $A_{\phi,n}\simeq T_{\phi,n}$, which will be viewed as identifications. However, the groups $T_{\phi,n}$ form naturally an inverse system, whereas the groups $A_{\phi,n}$ fit into a direct system; in fact, $A_\phi=\sideset{}{_n}\dirlim A_{\phi,n}$. The isomorphisms above allow us to introduce $\Sel(K_m,A_\infty)$ and $\Sel_S(K_m,A_{\phi,n})$ as in Definition \ref{Selmer}. Finally, recipes analogous to those in \eqref{Iwasawa-eq} and \eqref{Selmer-Iwasawa-eq} lead us to define $\Sel(K_\infty,A_\phi)\subset H^1_\mathrm{Iw}(K_\infty,A_\phi)$ and $\Sel_S(K_\infty,A_{\phi,n})\subset H^1_\mathrm{Iw}(K_\infty,A_{\phi,n})$.

\subsection{Explicit reciprocity laws} \label{explicit-subsec}

We extend the explicit reciprocity laws of \cite{Chida} and \cite{wang} to generalized Heegner cycles of $p$-power conductor. Let us fix an $A$-algebra homomorphism $\phi:\T_{M,D}(A)\rightarrow A$, where, as before, $A$ is either the valuation ring of a finite extension of $\Q_p$ or a quotient of it, and write $\pi$ for a generator of the maximal ideal of $A$. Suppose throughout that Assumption \ref{ass} is satisfied. Let $\mathcal{P}_n(\phi)$ and let $\phi_{S}$ be as in Theorem \ref{levelraising}. If $S$ belongs to $\mathcal{P}^\mathrm{def}_n(\phi)$, then we denote by the same symbol $\phi_S$ a lift to $S_k(M,DS; A)$ of a generator of $S_k(M,DS; A)/\mathcal I_{\phi_S}\simeq A/\pi^n A$ (see isomorphism \eqref{RL11}). As before, for all integers $n\geq1$ set $T_{\phi,n}\defeq T_\phi/\pi^nT_\phi$. We introduce the following Iwasawa elements. 
\begin{itemize}
\item For $S\in\mathcal{P}_{n}^\mathrm{def}(\phi)$, recall the notation from \S \ref{theta} and set 
\[ 
\begin{split}
\lambda_{\phi}^{(m)}(S)&\defeq\theta_{m,M,DS}(\phi_S)\quad\text{for all $m\geq1$ prime to $MD$},\\
\lambda_{\phi}(S)&\defeq\invlim_m\lambda_{\phi}^{(m)}(S)=\theta_{\infty,M,DS}(\phi_S)\in A[\![G_\infty]\!]. 
\end{split}
\]
\item For $S\in \mathcal{P}_{n}^\mathrm{indef}(\phi)$, recall the notation from \S \ref{reg-heeg} and set 
\[
\begin{split}
\kappa_{\phi}^{(m)}{(S)}&\defeq\mathbf{z}_{m,M,DS}(\phi_S)\quad\text{for all $m\geq1$ prime to $MD$},\\
\kappa_{\phi}{(S)}&\defeq\invlim_m\kappa_{\phi}^{(m)}(S)=\mathbf{z}_{\infty,M,DS}(\phi_S)\in H^1_\mathrm{Iw}(K_\infty,T_{\phi_{S}})\simeq H^1_\mathrm{Iw}(K_\infty,T_{\phi,n}).
\end{split}
\]
\end{itemize}
For the equivalence of representations $T_{\phi_{S}}\simeq T_{\phi,n}$, see, \emph{e.g.}, \cite[Proposition 6.8]{Chida}. Moreover, by \cite[p. 2332]{wang} (\emph{cf.} also \cite[Lemma 4.12]{LV-kyoto}), it is known that
\[ \kappa_{\phi}(S)\in\Sel_{S}(K_\infty,T_{\phi,n})=\varprojlim_m\Sel_S(K_m,T_{\phi,n}). \] 
Put $A_n\defeq A/\pi^nA$. Let $\ell\in \mathcal{P}_n(\phi)$. Since $\ell$ is inert in $K$, it splits completely in $K_m$ for each $m\geq1$, so there is a natural map 
\[ \partial_\ell:H^1_\mathrm{Iw}(K_\infty,T_{\phi,n})\longrightarrow H^1_\mathrm{sing}(K_\ell,T_{\phi,n})\otimes_{A}A[\![G_\infty]\!]. \]
Upon fixing an isomorphism $H^1_\mathrm{sing}(K_\ell,T_{\phi,n})\simeq A_n$, which exists by \cite[Lemma 1.5]{ChHs2}, we obtain from $\partial_\ell$ a map $\partial_\ell:H^1_\mathrm{Iw}(K_\infty,T_{\phi})\rightarrow A_n[\![G_\infty]\!]$, still denoted by the same symbol. 

\begin{theorem}[First reciprocity law]\label{ERL1}
Let $S\in\mathcal{P}_{n}^\mathrm{def}$ and $\ell\nmid S$ be an $n$-admissible prime.
There is a congruence   
\begin{equation} \label{first-rec-eq} 
\partial_\ell\bigl(\kappa_{\phi}{(S\ell)}\bigr)\equiv u\cdot\lambda_{\phi}{(S)}\pmod{\pi^n}, \end{equation}
where $u\in A_n^\times$ or $u\in G_\infty$.
\end{theorem} 

\begin{proof} The proof is similar to that of \cite[Theorem 3.6]{wang}, with \cite[Lemma 3.3]{wang} replaced by Lemma \ref{ERLbase}. Set $\mathcal{X}_{d,\ell}\defeq\mathcal{X}_d(M,DS\ell)$ and $\mathcal{W}_{d,\ell}\defeq\mathcal{W}_d(M,DS\ell)$, both defined over $\Z[1/MSDd]$, with canonical map $\pi:\mathcal{W}_{d,\ell}\rightarrow \mathcal{X}_{d,\ell}$; we denote by $\mathfrak{X}_{d,\ell}$ the fiber of $\mathcal{X}_{d,\ell}$ at $\ell$ and by $\mathbf{X}_{d,\ell}$ the special fiber of $\mathfrak{X}_{d,\ell}$. Recall the Heegner cycle $i_{\tilde{x}_m}\Bigl(Z_m^\frac{k-2}{2}\Bigr)$ of codimension $k/2$ in the generic fiber ${W}_{d,\ell}$ of $\mathcal{W}_{d,\ell}$, which is canonically identified with the generic fiber $W_d(M\ell,D)$ of $\mathcal{W}_d(M\ell,D)$; to ease our notation, we write $c_m$ for this Heegner cycle. Let $C_m$ be the Zariski closure of the support of $c_m$ in $\mathcal{W}_{d,\ell}$, so that there is an inclusion $C_m\subset \mathcal{W}_{d,\ell}$. Let $c_m^\sharp$ be the unique cycle of the normalization $\widetilde{\mathcal{W}}_{d,\ell}$ of ${\mathcal{W}}_{d,\ell}$ supported on $C_m$ of codimension $k/2$ whose restriction to the generic fiber is $c_m$. Then $\epsilon c_m^\sharp$ is a cycle of codimension $k/2$ of $\widetilde{\mathcal{W}}_{d,\ell}$ and the cycle class map produces a class $[\epsilon c_m^\sharp]\in\epsilon H^{k}_{C_m}\bigl(\widetilde{\mathcal{W}}_{d,\ell},\Z_p(k/2)\bigr)$. 

Now we consider the special fiber $\widetilde{\mathbf{W}}_{d,\ell}$ of $\widetilde{\mathcal{W}}_{d,\ell}$ and set $\overline{\widetilde{\mathbf{W}}}_{d,\ell}\defeq\widetilde{\mathbf{W}}_{d,\ell}\otimes_{\F_{\ell^2}}\overline{\F}_\ell$. Let $\widetilde{\overline{\mathbf{W}}}_{d,\ell}^{(0)}$ be the disjoint union of the irreducible components of $\overline{\widetilde{\mathbf{W}}}_{d,\ell}$. By \cite[Proposition 2.19]{Liu}, the image of $[\epsilon c_m^\sharp]$ in $\epsilon H^{k}\Bigl(\,\overline{\widetilde{\mathbf{W}}}^{(0)}_{d,\ell},\Z_p(k/2)\!\Bigr)$ is the cycle class of $\epsilon c_m^\sharp\otimes_{\widetilde{\mathcal{W}}_{d,\ell}}\overline{\widetilde{\mathbf{W}}}_{d,\ell}$. As explained in \cite[p. 2319]{wang} with an argument based on Saito's computation of the weight spectral sequence of the Kuga--Sato variety (\cite{Saito0}), applying $\epsilon_d$ produces a cycle $\mathbf{c}\in H^2\bigl(\overline{\mathbf{X}}_\ell, a_{0*}\mathcal{F}_A(1)\bigr)$ that is supported on $\overline{\mathbf{X}}_\ell\times \bar{x}_m^\sharp$; here the notation is similar to before: $\mathcal{X}_\ell=\mathcal{X}(M,DS\ell)$, $\mathbf{X}_\ell=\mathcal{X}_\ell\otimes\Z_{\ell^2}$ (tensor product over $\Z[1/MDS]$), $\overline{\mathbf{X}}_\ell=\mathbf{X}_\ell\otimes_{\Z_{\ell^2}}\F_{\ell^2}$, $X_\ell=\mathcal{X}_\ell\otimes_{\Z_{\ell^2}}\Q_{\ell^2}$, and $\bar{x}_m^\sharp$ denotes the reduction in $\mathbf{X}_\ell$ of the extension ${x}_m^\sharp$ of $x_m\in X_\ell$ to $\mathcal{X}_\ell$. The cohomology group $H^2\bigl(\overline{\mathbf{X}}_\ell,a_{0*}\mathcal{F}_A(1)\bigr)$ is identified with two copies of $S_k(M,DS;A)$ (\emph{cf.} \eqref{RL3}). By the \v{C}erednik--Drinfeld uniformization theorem (see, \emph{e.g.}, \cite{BC}), $\overline{\mathbf{X}}_\ell$ is the disjoint union of two line bundles over $B^\times\backslash\widehat{B}^\times/\widehat{R}^\times$, where $B$ is the (definite) quaternion algebra over $\Q$ of discriminant $DS$ and $R\subset B$ is an Eichler order of level $M$; these line bundles are permuted by the non-trivial element $\tau\in \Gal(\F_{\ell^2}/\F_\ell)$, which acts via the Atkin--Lehner involution (see, \emph{e.g.}, \cite[Section 1]{BD96} for more details). Since $\ell$ is $n$-admissible, we conclude that the action of $\tau$ is given by multiplication by $\epsilon_\ell\ell^\frac{k-2}{2}$. By Lemma \ref{ERLbase}, the cycle  $\mathbf{c}$ corresponds (up to sign) to a pair of functions $\bigl(\varphi_{\mathbf{w}_m},\varphi_{\mathbf{w}_m}'\bigr)\in S_k(M,DS;A)^2$ that are supported on $\bigl\{\bar{x}_m^\sharp\bigr\}$ and take values $\rho_k^{-1}(\gamma_\p)(\mathbf{w}_m)$ and $\epsilon_\ell\rho_k^{-1}(\gamma_\p)(\mathbf{w}_m)$, respectively. It follows from \cite[Theorem 2.18]{Liu} that $\partial_\ell$ sends $\mathbf{z}_{m,M,DS}(\phi_S)$ to the function $\varphi_{\mathbf{w}_m}$ that is supported on $\{\bar{x}_m\}$ and has value $\rho_k^{-1}(\gamma_\p)(\mathbf{w}_m)$ (\emph{cf.} also the proof of \cite[Theorem 3.6]{wang}); more generally, $\partial_\ell$ sends $\mathbf{z}_{m,M,DS}(\phi_S)^\sigma$ to the function $\varphi_{\mathbf{w}_m}^\sigma$ that is supported on $\{\bar{x}_m^\sigma\}$ and has value $\rho_k^{-1}(\gamma_\p)(\mathbf{w}_m)$.

Now, by \cite[(2.21)]{wang}, there is an isomorphism 
\[ H^1_\mathrm{sing}\Bigl(K_{m,\ell},H^1\bigl(\overline{X}_\ell,\mathcal{F}_A(1)\bigr)\!\Bigr)\overset\simeq\longrightarrow \coker(\tau\circ\rho), \]
where $\tau\circ\rho$ is the composition of the maps in \eqref{rhotau} that appears in the proof of Theorem \ref{levelraising1} (when comparing with \cite{wang}, observe that we actually have $\coker(\tau\circ\rho)\simeq \coker(\tau\circ\rho)^{G_{\F_{\ell^2}}}$ because all singular points of $\overline{X}_\ell$ are defined over $\F_{\ell^2}$). Notice that the residue field of $K_{m,\ell}$ is $\F_{\ell^2}$, as $\ell$ splits completely in $K_m$. Now, $\coker(\tau\circ\rho)$ is described in \eqref{RL7bis} and \eqref{RL8bis}, and it follows from this description and the Hecke-equivariance of the Petersson product (\emph{cf.} \cite[Proposition 6.2]{Chida}) that the map $x\mapsto\sum_{\sigma\in G_m}{\langle \sigma\cdot x,\phi_S\rangle}_B$ induces an isomorphism between $\coker(\tau\circ\rho)$ and $A[G_m]$. With this isomorphism in hand, there are equalities
\[ \begin{split}
\sum_{\sigma\in G_m}\big\langle\sigma\cdot\partial_\ell\bigl(\kappa_\phi^{(m)}(S\ell)\bigr),\phi_{S}^\sharp\big\rangle_B&=\sum_{\sigma\in G_m}\big\langle\partial_\ell\bigl(\mathbf{z}_{m,M,DS\ell}(\phi_{S\ell})^\sigma\bigr),\phi_{S}^\sharp\big\rangle_B\\
&=\pm\sum_{\sigma\in G_m}\big\langle\rho_k^{-1}(\gamma_\p)(\mathbf{w}_m),\phi_S^\sharp(x_m^\sigma)\big\rangle_k\sigma=\pm\lambda_\phi^{m-1}(S).
\end{split}\]
Passing to the limit over $m$ and keeping in mind that any other isomorphism between the singular cohomology and $A[G_m]$ differs from the one chosen above by multiplication by a unit of $A[G_m]$, congruence \eqref{first-rec-eq} follows. \end{proof}

\begin{remark}\label{rem2} 
The same result is obtained by Chida in \cite[Theorem 8.4]{Chida} with a different method that is based on the proof of Theorem \ref{levelraising1} given in \cite{Chida}; see Remark \ref{rem1}. 
\end{remark}

Let $\ell\in \mathcal{P}_n(\phi)$. As before, since $\ell$ splits completely in $K_m$, using the splitting \eqref{split} we get a projection 
\[ v_\ell:H_\mathrm{Iw}^1(K_\infty,T_{\phi,n})\longrightarrow H^1_\mathrm{fin}(K_\ell,T_{\phi,n})\otimes_AA[\![G_\infty]\!]. \]
Fixing an isomorphism $H^1_\mathrm{fin}(K_\ell,T_{\phi,n})\simeq A_n$, which exists by \cite[Lemma 1.5]{ChHs2}, we thus obtain from $v_\ell$ a map 
\[ v_\ell:H^1_\mathrm{Iw}(K_\infty,T_{\phi,n})\longrightarrow A_n[\![G_\infty]\!], \]
still denoted by the same symbol.

The proof of the result below is similar to that of \cite[Theorem 3.4]{wang}, upon replacing \cite[Lemma 3.3]{wang} with Lemma \ref{ERLbase}.

\begin{theorem}[Second reciprocity law]\label{ERL2} 
Let $S\in\mathcal{P}_{n}^\mathrm{indef}$ and let $\ell\nmid S$ be an 
$n$-admissible prime. The congruence
\[ v_\ell\bigl(\kappa_\phi{(S)}\bigr)\equiv u\cdot\lambda_\phi(S\ell)\pmod{\pi^n} \]
holds for a suitable $u\in A^\times$.
\end{theorem}

\begin{proof} Let $\mathcal{I}_{\phi_S}$ be the kernel of $\phi_S:\T_{M,DS}(A)\rightarrow A$. We use the same notation as in the proof of Theorem \ref{levelraising2}: $\mathcal{X}=\mathcal{X}(M,DS)$ over $\Z[1/MDS]$, $\mathfrak{X}=\mathcal{X}\otimes\Z_{\ell^2}$, the tensor product being taken over $\Z[1/MDS]$, $X=\mathfrak{X}\otimes_{\Z_{\ell^2}}\Q_{\ell^2}$, $\mathbf{X}=\mathfrak{X}\otimes_{\Z_{\ell^2}}\F_{\ell^2}$, $\overline{\mathbf{X}}=\mathbf{X}\otimes_{\F_{\ell^2}}\overline{\F}_{\ell^2}$; moreover, $\overline{\mathbf{X}}^\mathrm{ss}$ is the set of supersingular points of $\overline{\mathbf{X}}$. Note that, since $\ell$ is inert in $K$, the integral model of the Heegner point $x_m$ has supersingular reduction. Let $\varphi_{\mathbf{w}_m}\in S_k(M,D;\Z_p)$ take $g_m$ to $\rho_k^{-1}(\gamma_\p)(\mathbf{w}_m)$. As before, identify $S_k(M,D;\Z_p)$ with $H^0(\mathbf{X}^\mathrm{ss},\mathcal{F}_{\Z_p})$. The construction of the class $\mathrm{AJ}(\Delta_m)$ combined with Lemma \ref{ERLbase} shows that $\loc_\ell\bigl(\kappa_\phi^{(m)}(S)\bigr)$ is obtained (up to sign) by pulling back the function $\varphi_{\mathbf{w}_m}$ along the exact sequence \eqref{extclass}. On the other hand, the construction of the map $\gamma$ in \eqref{RL3+1/2} combined with the surjectivity of $\bar\gamma$ in \eqref{RL10} shows that every class in $H^1\Bigl(\F_{\ell^2},H^1\bigl(\overline{\mathbf{X}},\mathcal{F}_A(1)\bigr)\big/\mathcal{I}_{\phi_S}\Bigr)$ is obtained by pull-back from the short exact sequence 
\begin{equation} \label{ERL2-1}
0\longrightarrow H^1\bigl(\overline{\mathbf{X}},\mathcal{F}_A(1)\bigr)\big/\mathcal{I}_{\phi_S}\longrightarrow H^1\bigl(\overline{\mathbf{X}}^\mathrm{ord},\mathcal{F}_A(1)\bigr)\big/\mathcal{I}_{\phi_S}\longrightarrow
H^0\bigl(\overline{\mathbf{X}}^\mathrm{ss},\mathcal{F}_A\bigr)\big/\mathcal{I}_{\phi_S}\longrightarrow 0.
\end{equation}
It follows that $\loc_\ell\bigl(\kappa_\phi^{(m)}(S)\bigr)$ factors through \eqref{ERL2-1} and is given (up to sign) by pulling back $\varphi_{\mathbf{w}_m}$. Observe that $\partial_\ell\bigl(\kappa_\phi^{(m)}(S)\bigr)=0$ because $\ell\nmid S$, so $\loc_\ell\bigl(\kappa_\phi^{(m)}(S)\bigr)\in H^1_\mathrm{fin}(K_{m,\ell},T_{\phi,n})$. As before, the Hecke-equivariance of the Petersson inner product implies that the map $x\mapsto \sum_{\sigma\in G_m}{\langle \sigma\cdot x,\phi_{S\ell}\rangle}_B$ induces an isomorphism $S_k(M,DS\ell;A)/\mathcal{I}_{\phi_{S\ell}}\otimes_\Z A_n[G_m]\simeq A_n[G_m]$, where now $B$ is the (definite) quaternion algebra of discriminant $DS\ell$
(\emph{cf.} \cite[Proposition 6.2]{Chida}). It follows that the map $x\mapsto\sum_{\sigma\in G_m}{\langle \sigma\cdot x,\phi_{S\ell}\rangle}_B$ induces an isomorphism
\[ H^0\bigl(\overline{\mathbf{X}}^\mathrm{ss},\mathcal{F}_A\bigr)\big/\mathcal{I}_{\phi_{S\ell}}\otimes_\Z A_n[G_m]\simeq A_n[G_m]
\]
that, in light of \eqref{RL11}, yields an isomorphism
\begin{equation} \label{ERL2-2}
H^1\Bigl(\F_{\ell^2},H^1\bigl(\overline{\mathbf{X}},\mathcal{F}_A(1)\bigr)\big/\mathcal{I}_{\phi_S}\Bigr)\otimes_\Z A_n[G_m]\simeq A_n[G_m]. 
\end{equation}
On the other hand, the left-hand side of \eqref{ERL2-2} is isomorphic to $H^1_\mathrm{fin}(K_{m,\ell},T_{\phi,n})$, so the map $x\mapsto\sum_{\sigma\in G_m}{\langle\sigma\cdot x,\phi_{S\ell}\rangle}_B$ gives an isomorphism $H^1_\mathrm{fin}(K_{m,\ell},T_{\phi,n})\simeq A_n[G_m]$ upon identifying the class $\loc_\ell\bigl(\kappa_\phi^{(m)}(S)\bigr)\in H^1_\mathrm{fin}(K_{m,\ell},T_{\phi,n})$ with the function $\varphi_{\mathbf{w}_m}$ viewed as an element of $S_k(M,DS\ell;A)/\mathcal{I}_{\phi_{S\ell}}\otimes_\Z A_n[G_m]$. Therefore, there are equalities 
\[ \begin{split}
\sum_{\sigma\in G_m}\big\langle \sigma\cdot v_\ell\bigl(\kappa_\phi^{(m)}{(S)}\bigr),\phi_{S\ell}^\sharp\big\rangle_B&=\sum_{\sigma\in G_m}\big\langle v_\ell\bigl(\mathbf{z}_{m,M,DS}(\phi_S)^\sigma\bigr),\phi_{S\ell}^\sharp\big\rangle_B\\
&=\pm\sum_{\sigma\in G_m}\big\langle\rho_k^{-1}(\gamma_\p)(\mathbf{w}_m),\phi_{S\ell}^\sharp(x_m^\sigma)\big\rangle_k\cdot\sigma=\pm\lambda^{(m-1)}_\phi(S\ell).
\end{split}\]
Finally, the isomorphism $H^1_\mathrm{fin}(K_{m,\ell},T_{\phi,n})\simeq A_n[G_m]$ that was chosen above is unique up to multiplication by a unit of $A_n[G_m]$, so we conclude by passing to the limit over $m$. \end{proof}

\begin{corollary}\label{coro2.5} 
Let $S\in\mathcal{P}^\mathrm{indef}$ and let $\ell_1$, $\ell_2$ be distinct admissible primes not dividing $S$. Then 
\[ v_{\ell_1}\bigl(\kappa_{\phi}(S)\bigr)\neq 0\;\Longleftrightarrow\;\partial_{\ell_2}\bigl(\kappa_{\phi}(S\ell_1\ell_2)\bigr)\neq 0. \] 
\end{corollary}

\begin{proof} By Theorem \ref{ERL2}, there is an equality $v_{\ell_1}\bigl(\kappa_\phi(S)\bigr)=\lambda_\phi(S\ell_1)$; on the other hand, $\lambda_\phi(S\ell_1)=\partial_{\ell_2}\bigl(\kappa_\phi(S\ell_1\ell_2)\bigr)$ by Theorem \ref{ERL1}, and the corollary follows. \end{proof}

\begin{remark} 
It would be interesting to explore applications of more general versions of explicit reciprocity laws for all integers $c\geq 1$, not necessarily in the context of Iwasawa theory, with a view towards the Tamagawa number conjecture for the motive of $f$; results in this direction can be found, \emph{e.g.}, in \cite{Chida} and \cite{LV-JNT}. 
\end{remark}

\section{Anticyclotomic Iwasawa main conjectures} \label{Iwasawa}

Let $f\in S_k(\Gamma_0(N))$ be a newform of even weight $k\geq4$ and square-free level $N$, with $q$-expansion $f(q)=\sum_{n\geq1}a_n(f)q^n$. As remarked, \emph{e.g.}, in \cite[p. 34]{ribet}, the square-freeness of $N$ guarantees that $f$ has no complex multiplication in the sense of \cite[p. 34, Definition]{ribet}. Write $N$ as $N=MD$ where $M$ is a product of primes that split in $K$ and $D$ is a product of primes that are inert in $K$. Let $F\defeq\Q\bigl(a_n(f)\mid n\geq1\bigr)$ be the Hecke field of $f$, which is a totally real number field, and write $\cO_F$ for its ring of integers. Finally, let $\cO_f\defeq\Z\bigl[a_n(f)\mid n\geq1\bigr]$ be the order of $\cO_F$ generated by the Fourier coefficients of $f$ and let $c_f\defeq[\cO_F:\cO_f]\in\N$ be the index of $\cO_f$ in $\cO_F$. Let $p$ be a prime number. Denote by
\begin{equation} \label{rho-p-eq} 
\rho_p:G_\Q\longrightarrow\GL_2(\cO_F\otimes_\Z\Z_p) 
\end{equation} 
the $p$-adic Galois representation attached to $f$ and $p$ by Deligne (\cite{Del-Bourbaki}). We say that $\rho_p$ has \emph{big image} if there is an inclusion
\[ 
\bigl\{g\in\GL_2(\cO_F\otimes_\Z\Z_p)\mid\det(g)\in(\Z_p^\times)^{k-1}\bigr\}\subset\im(\rho_p). 
\]
Since $f$ is not CM, a result of Ribet ensures that $\rho_p$ has big image for all but finitely many $p$ (\cite[Theorem 3.1]{ribet2}).

Let $\p$ be the prime of $F$ above $p$ induced by the fixed embedding $\overline{\Q}\hookrightarrow\overline\Q_p$. Moreover, denote by $\cO_\p$ the completion of $\cO_F$ at $\p$ and consider the Iwasawa algebra $\Lambda\defeq\cO_\p\llbracket G_\infty \rrbracket$. 

\begin{definition} \label{anomalous}
The prime number $p$ is \emph{anomalous} if 
\[
a_p(f)\equiv\begin{cases}1\pmod{\p} & \text{if $p$ splits in $K$},\\[3mm]\pm1\pmod{\p} & \text{if $p$ is inert in $K$}. \end{cases}
\]
The prime number $p$ is \emph{non-anomalous} if $p$ is not anomalous. 
\end{definition}

From now on, we assume the following conditions. 

\begin{assumption} \label{ass1}
\begin{enumerate}
\item $p\nmid ND_Kc_f$.
\item $p>k+1$ and $\#{(\F_p^\times})^{k-1}>5$.
\item $p$ is non-anomalous.
\item Assumption $\ref{ass}$ holds for $\phi=\phi_f$.
\item The representation $\rho_p$ in \eqref{rho-p-eq} has big image.
\end{enumerate}
\end{assumption}

Notice that the non-divisibility of $c_f$ by $p$ and the big image property of $\rho_p$ are imposed because they are assumed in \cite{LPV}, some of whose results will be used below.

We distinguish two cases, which we refer to as \emph{definite} and \emph{indefinite}, according to the parity of the number of the prime factors of $D$. Our purpose in this section is to prove an anticyclotomic Iwasawa main conjecture (AIMC, for short) for $f$, both in the definite case and in the indefinite case. Notice that all the Selmer groups in the statements below were defined in \S \ref{Selmer-subsec}.

\subsection{Indefinite setting} \label{indefinite-subsec}

Assume that $D$ is a square-free product of an  \emph{even} number of primes, so $\mu(D)=+1$. Let $\mathbb{T}_{M,D}(\cO_\p)$ be the Hecke algebra acting by correspondences on the \'etale cohomology of the Shimura curve $X(M,D)$ with coefficients in $\mathcal{F}_{\cO_\p}$, and denote by $\phi=\phi_f:\mathbb{T}_{M,D}(\cO_\p)\rightarrow \mathcal{O}_\p$ the $\cO_\p$-algebra homomorphism attached to $f$. Write $T_\phi=T_\phi(\cO_\p)$ for the integral Galois representation attached to $\phi$ as in \S \ref{Gal} and suppose that $\phi$ satisfies Assumption \ref{ass}. Note that, since $\mu(D)=+1$, we can add $1$ to the set $\mathcal{P}_n^\mathrm{indef}(\phi)$ for each $n\geq 1$, and to simplify the notation, we will just write $\Sel(K_\infty,T_{\phi,n})$ for $\Sel_1(K_\infty,T_{\phi,n})$. The collection of classes 
\[ \bigl\{\kappa_\phi(1)\in \Sel(K_\infty, T_{\phi,n})\bigr\}_{n\geq 1} \] 
is compatible with respect to the projections $T_{\phi,n+1}\rightarrow T_{\phi,n}$, so the inverse limit over $n$ gives rise to an element $\kappa_\infty\in\Sel(K_\infty,T_\phi)$.
Note that $\kappa_\infty$ is the element $\mathbf{z}_\infty(\phi)$ introduced in \eqref{Heegnerclass}. 

%Let 
%\[\kappa_\infty=\varprojlim_j \kappa_j(1)\in \varprojlim_j H^1_1(K,\mathbf{T}_j)\sim \Sel(K_\infty,T).\]
%{\footnote{Da controllare le relazioni fra i vari gruppi di Selmer}}
\begin{proposition}[Burungale] \label{non zero}
If $N$ is square-free and $p$ splits in $K$, then $\kappa_\infty$ is non-zero.
\end{proposition}

\begin{proof}
This is a consequence of \cite[Theorem A]{Burungale}.
\end{proof}

\begin{lemma} \label{torsion-free}
The $\Lambda$-module $\Sel(K_\infty,T_\phi)$ is torsion-free.
\end{lemma}

\begin{proof} For each $m\geq 0$ the extension $K_m/\Q$ is solvable, so $H^0(K_m,T_\phi)=0$ by \cite[Lemma 3.10, (1)]{LV}. Since $H^0(K_\infty,T_\phi)=\varprojlim_m H^0(K_m,T_\phi)$, it follows that $H^0(K_\infty,T_\phi)=0$ as well, and the claim of the lemma follows from \cite[Lemma 1.3.3]{PR-Ast}. \end{proof}
%We refer to the case in which $N^-$ is the product of an even (resp. odd) number of primes as the \emph{indefinite setting} (resp. \emph{definite setting}).

The following is our result on the AIMC in the indefinite setting.

\begin{theorem}[AIMC, indefinite case] \label{IMC}
Suppose that Assumption $\ref{ass1}$ holds and $p$ splits in $K$. Then
\begin{enumerate}
\item $\rank_\Lambda \Sel(K_\infty,T_\phi)=\rank_\Lambda\Sel(K_\infty,A_\phi)^\vee=1$;
\item the indefinite anticyclotomic Iwasawa main conjecture holds, i.e., there is an equality 
\begin{equation}
\cchar_\Lambda\bigl(\Sel(K_\infty,A_\phi)^\vee_\tor\bigr)=\cchar_\Lambda\bigl(\Sel(K_\infty,T_\phi)/\Lambda\cdot\kappa_\infty\bigr)^2 
\end{equation}
of ideals of $\Lambda$.
\end{enumerate}
%Moreover, equality in \eqref{1-inclus} holds if the following condition is satisfied: for any height one prime $\mathfrak{P}\subset \Lambda$, there exists $k=k(\mathfrak{P})$ such that for all $j\geq k$ the set
%\[ \{\theta_\infty(\phi)\in \Lambda/\p^j\Lambda\ \vert\ \phi\in \mathbf{S}_k^{B_m}(\cO_\p/\p^j), m\in\mathcal{N}_j^-\}\]
%contains an element with nontrivial image in $\Lambda/(\mathfrak{P},\p^k)$.
\end{theorem}
This is Theorem \ref{ind-IMC-intro-thm} in the introduction. We will prove this result in \S \ref{proof}

\subsection{Definite setting}\label{definite-subsec}

Assume that $D$ is a square-free product of an \emph{odd} number of primes. Let $\phi\colon \mathbb{T}_{M,D}(\cO_\p)\rightarrow\cO_\p$ be the Jacquet--Langlands lift of $f$. Since $\mu(D)=-1$, for each $n\geq1$ we can add $1$ to the set $\mathcal{P}_n^\mathrm{def}(\phi)$. The resulting collection $\bigl\{\lambda_\phi(1)\in \Lambda/\p^n\Lambda\bigr\}_{n\geq 1}$ is compatible with respect to the natural maps $\Lambda/\p^{n+1}\Lambda\rightarrow \Lambda/\p^n\Lambda$ and so gives rise to an element $\lambda_\phi\in\Lambda$. Then the $\mathfrak{p}$-adic $L$-function $\mathcal{L}_\p(f)\in\Lambda$ of $f$, defined in \S \ref{L-funct} and denoted $\mathcal{L}(f)$ there, is just the product $\lambda_\phi\lambda_\phi^*$.

Now we state our result on the AIMC in the definite setting.

\begin{theorem}[AIMC, definite case] \label{IMC-def}
Suppose that
Assumption $\ref{ass1}$ holds. 
Then
\begin{enumerate}
\item $\rank_\Lambda \Sel(K_\infty,T_\phi)=\rank_\Lambda\Sel(K_\infty,A_\phi)^\vee=0$;
\item the definite anticyclotomic Iwasawa main conjecture holds, i.e., there is an equality 
\begin{equation}
\cchar_\Lambda\bigl(\Sel(K_\infty,A_\phi)^\vee\bigr)=\bigl(\mathcal{L}_\mathfrak{p}(f)\bigr)
\end{equation}
of ideals of $\Lambda$.
\end{enumerate}
\end{theorem}
This is Theorem \ref{A-thm} in the introduction. Like its indefinite counterpart, this result will be proved in \S \ref{proof}.

\subsection{Proof of the anticyclotomic IMCs} \label{proof}

In this section, we prove the definite and the indefinite anticyclotomic Iwasawa main conjectures for our newform $f$, following Mazur--Rubin (\cite{MR}) and Howard (\cite{howard-bipartite}. We assume throughout that Assumption \ref{ass1} is satisfied, and that $p$ splits in $K$ in the indefinite case.

\subsubsection{Selmer groups over $\Lambda$} \label{selmerlambda}

It is convenient to use the following alternative description of Selmer groups. Set 
\[ 
\mathbf{T}_\phi\defeq\varprojlim_n\mathrm{Ind}^{G_K}_{G_{K_n}}(T_\phi),
\] 
the inverse limit being taken with respect to the corestriction maps. As before, for any prime $v\nmid p$ of $K$ define
\[ H^1_\unr(K_v,\mathbf{T}_\phi)\defeq H^1\bigl(K_v^\unr/K_v,\mathbf{T}_\phi^{I_v}\bigr)=\ker\Bigl(H^1(K_v,\mathbf{T}_\phi)\longrightarrow H^1\bigl(I_v,\mathbf{T}_\phi)\Bigr). \]
If $v\,|\,p$, then set 
\[
\mathbf{T}_\phi^{(v)}\defeq \invlim_n\mathrm{Ind}^{G_K}_{G_{K_n}}\bigl(T_\phi^{(v)}\bigr)
\] 
and define 
\[ H^1_\ord(K_v,\mathbf{T}_\phi)\defeq\im\Bigl(H^1(K_v,\mathbf{T}_\phi^+)\longrightarrow H^1(K_v,\mathbf{T}_\phi)\Bigr). \] 
If $v\,|\,D$, then set 
\[
\mathbf{T}_\phi^{(v)}\defeq \invlim_n\mathrm{Ind}_{K_n/K}\bigl(T_\phi^{(v)}\bigr)
\] 
and define 
\[ H^1_\ord(K_v,\mathbf{T}_\phi)\defeq\im\Bigl(H^1(K_v,\mathbf{T}_\phi^{(v)})\longrightarrow H^1(K_v,\mathbf{T}_\phi)\Bigr). \] 

\begin{definition}
The \emph{Selmer group} $\Sel(K,\mathbf{T}_\phi)$ is the group of all $s\in H^1(K,\mathbf{T}_\phi)$ such that 
\begin{itemize}
\item $\loc_v(s)\in H^1_\unr(K_v,\mathbf{T}_\phi)$ for all $v\nmid Dp$;
\item $\loc_v(s)\in H^1_\mathrm{ord}(K_v,\mathbf{T}_\phi)$ for all $v\,|\,Dp$.
\end{itemize} 
\end{definition}

Likewise, set 
\[
\mathbf{A}_\phi\defeq\varinjlim_n\mathrm{Ind}^{G_K}_{G_{K_n}}(A_\phi).
\]
Here and in the following, the direct limit is taken with respect to the canonical restriction maps. 
As before, define 
\[ H^1_\unr(K_v,\mathbf{A}_\phi)\defeq H^1\bigl(K_v^\unr/K_v,\mathbf{A}_\phi^{I_v}\bigr)=\ker\Bigl(H^1(K_v,\mathbf{A}_\phi)\longrightarrow H^1\bigl(I_v,\mathbf{A}_\phi)\Bigr). \] 
If $v\,|\,Dp$, then set 
\[
\mathbf{A}_\phi^{(v)}\defeq\dirlim_n\mathrm{Ind}^{G_K}_{G_{K_n}}\bigl(A_\phi^{(v)}\bigr)
\] 
and define 
\[ H^1_\ord(K_v,\mathbf{A}_\phi)\defeq\im\Bigl(H^1(K_v,\mathbf{A}_\phi^+)\longrightarrow H^1(K_v,\mathbf{A}_\phi)\Bigr). \]

\begin{definition}
The \emph{Selmer group} $\Sel(K,\mathbf{A}_\phi)$ is the group of all $s\in H^1(K,\mathbf{A}_\phi)$ such that 
\begin{itemize}
\item $\loc_v(s)\in H^1_\mathrm{fin}(K_v,\mathbf{A}_\phi)$ for all $v\nmid Dp$;
\item $\loc_v(s)\in H^1_\mathrm{ord}(K_v,\mathbf{A}_\phi)$ for all $v\,|\,Dp$.
\end{itemize} 
\end{definition}
Before going on, we discuss semilocal cohomology groups. Fix a Galois extension $F/K_m$, then set $\mathcal{G}\defeq\Gal(F/K)$ and $\mathcal{G}_m\defeq\Gal(F/K_m)$. For each prime $w$ of $K_m$ above a prime $v$ of $K$, consider the decomposition groups $\mathcal{D}_v\subset \mathcal{G}$ and $\mathcal{D}_{w}\subset \mathcal{G}_m$, and let $g_w\in \mathcal{G}$ 
be such that $\mathcal{D}_{w}=g_w^{-1}\mathcal{D}_{w_0}g_w$, where $w_0$ is the prime ideal of $K_m$ over $v$ corresponding to the fixed embedding $\overline\Q\hookrightarrow\overline{\Q}_p$. Then $M_w\defeq g_w^{-1}M$ becomes a $D_{w}$-module. We also define the $\mathcal{D}_{w_0}$-module $M^{(w)}$ to be $M$ equipped 
with the action of $\mathcal{D}_{w_0}$ that is obtained by combining the obvious map $\mathcal{D}_{w_0}\rightarrow \mathcal{D}_w$ with the action of $\mathcal{D}_w$; there is an isomorphism 
\[
H^1(\mathcal{D}_w,M_w)\overset\simeq\longrightarrow H^1\bigl(\mathcal{D}_{w_0},M^{(w)}\bigr)
\]
that sends the class of a $1$-cocycle $c:\mathcal{D}\rightarrow M_w$ to the class of the $1$-cocycle $c':\mathcal{D}_{w_0}\rightarrow M^{(w)}$ such that $c'(g)\defeq g_wc(g_w^{-1}gg_w)$. By \cite[8.1.7.1]{Nek-Selmer}, there is a canonical isomorphism of $\mathcal{D}_v$-modules 
\[ 
\Ind^{\mathcal{G}}_{\mathcal{G}_m}(M)\overset\simeq\longrightarrow\bigoplus_{w\mid v}\Ind^{\mathcal{D}_v}_{\mathcal{D}_{w_0}}\bigl(M^{(w)}\bigr), 
\]
where $\mathcal{D}_v$ acts on the left-hand side by restriction from $G_K$ to $\mathcal{D}_v$ and on the right-hand side diagonally. We also have $\Ind^{\mathcal{D}_v}_{\mathcal{D}_{w_0}}(M)^{(w)}\simeq\Ind^{\mathcal{D}_v}_{\mathcal{D}_{w}}(M{)}_{w}$, where induction is taken with respect to the composition $D_w\rightarrow D_{w_0}\hookrightarrow D_v$ of conjugation by $g_w$ and the canonical inclusion. Then there are isomorphisms 
\begin{equation} \label{Sh00}
\begin{split}
H^1\bigl(\mathcal{D}_v,\Ind^{\mathcal{G}}_{\mathcal{G}_m}(M)\bigr)&\overset\simeq\longrightarrow
H^1\Bigl(\mathcal{D}_v,\bigoplus_{w\mid v}\Ind^{\mathcal{D}_v}_{\mathcal{D}_{w_0}}\bigl(M^{(w)}\bigr)\!\Bigr)\\
&\overset\simeq\longrightarrow\bigoplus_{w\mid v}H^1\bigl(\mathcal{D}_{w_0},M^{(w)}\bigr)\overset\simeq\longrightarrow\bigoplus_{w\mid v}H^1(\mathcal{D}_w,M_w)
\end{split}
\end{equation}
(\emph{cf.} also \cite[\S5]{Ru}). Identifying $\Gal(\overline{K_{m,w}}/K_{m,w})$ with $\mathcal{D}_w$ when $F=\overline{K}$, and defining the semilocal cohomology groups $H^1(K_{m,v},M)\defeq\prod_{w\mid v}H^1(K_{m,w},M_w)$, we get an isomorphism 
\begin{equation} \label{Sh0}
H^1\bigl(K_v,\Ind^{G_K}_{G_{K_m}}(M)\bigr)\simeq H^1(K_{m,v},M). 
\end{equation}
Direct limits commute with cohomology, so Shapiro's lemma gives canonical isomorphisms of $\Lambda$-modules 
\begin{equation} \label{ShapiroA} 
\begin{split}
\mathrm{Sh}:H^1(K,\mathbf{A}_\phi)&=H^1\bigl(K,\dirlim_n\Ind^{G_K}_{G_{K_n}}({A}_\phi)\bigr)\\&\overset\simeq\longrightarrow\dirlim_nH^1\bigl(K,\Ind^{G_K}_{G_{K_n}}({A}_\phi)\bigr)\overset\simeq\longrightarrow\varinjlim_n H^1(K_n,A_\phi),
\end{split}
\end{equation}
where the direct limit is taken with respect to restrictions. In fact, a result on Selmer groups is also true: this is well known, but we add a proof for lack of precise references. 

\begin{proposition} \label{comparison1}
Isomorphism $\eqref{ShapiroA}$ induces an isomorphism of $\Lambda$-modules 
\[ 
\mathrm{Sh}:\Sel(K_\infty,A_\phi)\overset\simeq\longrightarrow\Sel(K,\mathbf{A}_\phi).
\]
\end{proposition}

\begin{proof} As in \eqref{ShapiroA}, Shapiro's lemma induces canonical isomorphisms of semilocal cohomology groups 
\[
\mathrm{Sh}_v:H^1(K_v,\mathbf{A}_\phi)\overset\simeq\longrightarrow
\dirlim_mH^1\bigl(K_v,\Ind^{G_K}_{G_{K_m}}({A}_\phi)\bigr)\overset\simeq\longrightarrow\varinjlim_m H^1(K_{m,v},A_\phi),\]
where the first isomorphism follows from the commutation of direct limits and cohomology, 
while the second follows from \eqref{Sh0}. Then there is a commutative square 
\[
\xymatrix@C=35pt{
H^1(K,\mathbf{A}_\phi)\ar[r]^-{\mathrm{Sh}}\ar[d]&\dirlim_m H^1(K_m,A_\phi)\ar[d]\\
H^1(K_v,\mathbf{A}_\phi)\ar[r]^-{\mathrm{Sh}_v}&\dirlim_m H^1(K_{m,v},A_\phi)
}
\]
in which the left vertical map is the local restriction map at $v$, while the right vertical map is the limit of the products of the local restriction maps at all $w\mid v$. Then we need to show that 
$\mathrm{Sh}_v$ induces an isomorphism 
\begin{equation} \label{Shv}
\mathrm{Sh}_v:H^1_\bullet(K_v,\mathbf{A}_\phi)\overset\simeq\longrightarrow\dirlim_mH^1_\bullet(K_{m,v},A_\phi)  
\end{equation} 
for $\bullet=\unr$ if $v\nmid Dp$ and $\bullet=\ord$ if $v\,|\,Dp$, 
where we adopt the same notation as before for semilocal cohomology.

If $v\,|\,Dp$, then, by \eqref{Sh0} and the commutation betweeen direct limits and cohomology, there are canonical isomorphisms of semilocal cohomology groups 
\[ 
\mathrm{Sh}_v:H^1\bigl(K_v,\mathbf{A}_\phi^{(v)}\bigr)\overset\simeq\longrightarrow 
\dirlim_mH^1\Bigl(K_v,\Ind^{G_K}_{G_{K_m}}\bigl({A}_\phi^{(v)}\bigr)\!\Bigr)\overset\simeq\longrightarrow\varinjlim_m H^1\bigl(K_{m,v},A_\phi^{(v)}\bigr). 
\]
This shows isomorphism \eqref{Shv} in this case.  

If $v\nmid Dp$, then, using again the commutativity of direct limits with cohomology, we get canonical isomorphisms
\[ 
\begin{split}
H^1_\unr(K_v,\mathbf{A}_\phi)&=H^1\bigl(K_v^\unr/K_v,\mathbf{A}_\phi^{I_v}\bigr)=H^1\Bigl(K_v^\unr/K_v,\bigl(\dirlim_m\mathrm{Ind}^{G_K}_{G_{K_m}}(A_\phi)\bigr)^{I_v}\Bigr)\\
&\overset\simeq\longrightarrow H^1\Bigl(K_v^\unr/K_v,\dirlim_m\bigl(\mathrm{Ind}^{G_K}_{G_{K_m}}(A_\phi)^{I_v}\bigr)\!\Bigr)\overset\simeq\longrightarrow\dirlim_mH^1\bigl(K_v^\unr/K_v,\mathrm{Ind}^{G_K}_{G_{K_m}}(A_\phi)^{I_v}\bigr).
\end{split}
\]
Set $G_{v}\defeq G_{K_v}$. Since $v\nmid p$ and $K_\infty/K$ is unramified at the residual characteristic $\ell$ of $v$, the action of $I_v$ on the $G_{v}$-module $\cO_\p[G_m]$ is trivial; thus, we get 
\[ 
\Ind^{G_K}_{G_{K_m}}(A_\phi)^{I_v}=A_\phi^{I_v}\otimes_{\cO_\p}\cO_\p[G_m]=\mathrm{Ind}^{G_K}_{G_{K_m}}\bigl(A_\phi^{I_v}\bigr) 
\] 
as $G_{v}$-modules. By \eqref{Sh00} with $F$ the maximal extension of $K$ unramified outside $v$, and using the fact that the inertia subgroup $I_{m,w}$ of $G_{K_{m,w}}$ coincides with that of $G_{K_v}$ because the extension $G_{K_{m,w}}/G_{K_v}$ is unramified, we obtain canonical isomorphisms 
\[ 
H^1_\unr(K_v,\mathbf{A}_\phi)\overset\simeq\longrightarrow\dirlim_m\prod_{w\mid v}H^1\bigl(K_{m,w}^\unr/K_{m,w},A_\phi^{I_{m,w}}\bigr)\overset\simeq\longrightarrow\dirlim_mH^1_\unr(K_v,A_\phi).
\]
This proves isomorphism \eqref{Shv} in this case as well, as desired. \end{proof}

Now we discuss the relation between the compact $\Lambda$-modules $\Sel(K_\infty,T_\phi)$ 
and $\Sel(K,\mathbf{T}_\phi)$. 
%Let $S$ be the set of prime ideals of $K$ dividing $p$ and where $T$ is ramified, and denote 
%$K_S$ the maximal extension of $K$ which unramified outside $S$. Let also $G_S=\Gal(K_S/K)$.  
%By definition, 
%\[\Sel(K,\mathbf{T}_\phi)\subset H^1(G_S,\mathbf{T}_\phi).\]
Since $T_\phi=\sideset{}{_n}\invlim T_{\phi,n}$, there are isomorphisms
\[
\mathbf{T}_\phi=\invlim_m\Ind^{G_K}_{G_m}(T_{\phi})
\simeq\invlim_m\Ind^{G_K}_{G_m}\bigl(\sideset{}{_n}\invlim T_{\phi,n}\bigr)\simeq
\invlim_{m}\Bigl(\!\bigl(\sideset{}{_n}\invlim T_{\phi,n}\bigr)\otimes_{\cO_\p}\cO_\p[G_m]\Bigr).
\]
Since the $\Lambda$-modules $T_{\phi,n}$ are all finite and the transition maps are surjective, the inverse system ${\{T_{\phi,n}\}}_{n\geq1}$ satisfies the Mittag-Leffler condition; therefore, there are isomorphisms 
\[ 
\begin{split}
\invlim_{m}\Bigl(\!\bigl(\sideset{}{_n}\invlim T_{\phi,n}\bigr)\otimes_{\cO_\p}\cO_\p[G_m]\Bigr)&\simeq 
\invlim_{m}\sideset{}{_n}\invlim\bigl(T_{\phi,n}\otimes_{\cO_\p}\cO_\p[G_m]\bigr)\\&\simeq 
\bigl(\sideset{}{_n}\invlim T_{\phi,n}\bigr)\otimes_{\cO_\p}\cO_\p[G_n]=
\invlim_{n}\Ind^{G_K}_{G_{K_n}}(T_{\phi,n}),
\end{split}
\]
the second of which follows, as before, from the inverse system  
$\bigl\{T_{\phi,n}\otimes_{\cO_\p}\cO_\p[G_n]\bigr\}_{n\geq1}$ being cofinal to the inverse 
system $\bigl\{T_{\phi,n}\otimes_{\cO_\p}\cO_\p[G_m]\bigr\}_{n\geq1,m\geq1}$. Thus, we obtain an isomorphism
\begin{equation} \label{T}
\mathbf{T}_\phi \overset\simeq\longrightarrow \invlim_{n}\Ind^{G_K}_{G_{K_n}}(T_{\phi,n}).
\end{equation}
On the other hand, each $\Ind^{G_K}_{G_{K_n}}(T_{\phi,n})$ is a finite $\Lambda$-module, so \cite[Proposition B.2.3]{Ru} ensures that isomorphism \eqref{T} induces an isomorphism of $\Lambda$-modules 
\begin{equation} \label{T2}
H^1(K,\mathbf{T}_\phi)\overset\simeq\longrightarrow\invlim_n H^1\bigl(K,\Ind^{G_K}_{G_n}(T_{\phi,n})\bigr). 
\end{equation}
Finally, since $\sideset{}{_n}\invlim H^1(K_n,T_{\phi,n})\simeq H^1_\mathrm{Iw}(K_\infty,T_\phi)$ by \cite[Lemma B.3.1]{Ru}, Shapiro's lemma and isomorphism \eqref{T2} yield an isomorphism of $\Lambda$-modules 
\begin{equation} \label{Sh-eq}
\mathrm{Sh}:H^1(K,\mathbf{T}_\phi)\overset\simeq\longrightarrow H^1_\mathrm{Iw}(K_\infty,T_\phi).
\end{equation}
It turns out that $\mathrm{Sh}$ induces an isomorphism of $\Lambda$-module, as is shown by the following proposition (\emph{cf.} \cite[\S II.1]{Colmez} and \cite[Lemma 5.3.1]{MR} for related results).

\begin{proposition} \label{comparison2}
Isomorphism $\mathrm{Sh}$ in \eqref{Sh-eq} induces an isomorphism of $\Lambda$-modules 
\[
\mathrm{Sh}:\Sel(K,\mathbf{T}_\phi)\overset\simeq\longrightarrow\Sel(K_\infty,T_\phi).
\]
\end{proposition}

\begin{proof} Let $v$ be a prime of $K$. Isomorphism \eqref{T} gives isomorphisms of $\Lambda$-modules 
\begin{equation}\label{T-iso0}
H^1(K_v,\mathbf{T}_\phi)\overset\simeq\longrightarrow\invlim_n H^1\bigl(K_v,\Ind^{G_K}_{G_n}(T_{\phi,n})\bigr)\overset\simeq\longrightarrow\invlim_n H^1(K_{n,v},T_{\phi,n}), 
\end{equation} 
the first following from \cite[Proposition B.2.3]{Ru} and the second from \eqref{Sh0}. We have then to check that this isomorphism induces an isomorphism 
\begin{equation} \label{Tproof}
H^1_\bullet(K_v,\mathbf{T}_\phi)\overset\simeq\longrightarrow\invlim_n H^1_\bullet(K_{n,v},T_{\phi,n})
\end{equation}
for $\bullet\in\{\ord,\unr\}$. 

As before, $\mathbf{T}_\phi^{(v)}=\sideset{}{_n}\invlim\Ind^{G_K}_{G_n}\bigl(T_{\phi,n}^{(v)}\bigr)$ for each prime $v\,|\,Dp$,
where $T_{\phi,n}^{(v)}=T_\phi^{(v)}\big/\mathfrak{p}^nT_\phi^{(v)}$. Therefore, since $\Ind^{G_K}_{G_n}\bigl(T_{\phi,n}^{(v)}\bigr)$ is a finite $\Lambda$-module, we get isomorphisms 
\begin{equation} \label{Tv-iso0}
H^1\bigl(K_v,\mathbf{T}_\phi^{(v)}\bigr)\overset\simeq\longrightarrow\invlim_n H^1\Bigl(K_v,\Ind^{G_K}_{G_n}\bigl(T_{\phi,n}^{(v)}\bigr)\!\Bigr)\overset\simeq\longrightarrow\invlim_n H^1\bigl(K_{n,v},T_{\phi,n}^{(v)}\bigr), 
\end{equation}
the first of which is a consequence of \cite[Proposition B.2.3]{Ru}, while the second stems from \eqref{Sh0}. It follows that there is a commutative square
\[ 
\xymatrix@C=35pt{H^1\bigl(K_v,\mathbf{T}_\phi^{(v)}\bigr)\ar[d]^-\simeq\ar[r]& H^1(K_v,\mathbf{T}_\phi)\ar[d]^-\simeq\\\invlim_n H^1\bigl(K_{n,v},T_{\phi,n}^{(v)}\bigr)\ar[r]&\invlim_n H^1(K_{n,v},T_{\phi,n}),}
\]
and we deduce that \eqref{Tproof} holds true for $\bullet=\ord$. 

For $v\nmid MDp$, by \eqref{Sh00} applied with $\mathcal{G}$ the maximal extension of $K$ unramified outside $v$, there is a canonical isomorphism 
\begin{equation} \label{T-iso2}
H^1_\unr\bigl(K_v,\mathrm{Ind}^{G_K}_{G_{K_n}}(T_{\phi,n})\bigr)\overset\simeq\longrightarrow H^1_\unr(K_{n,v},T_{\phi,n})\defeq\prod_{w\mid v}H^1_\unr\bigl(K_{n,w},(T_{\phi,n}{)}_w\bigr),
\end{equation}
so there are isomorphisms
\[ 
H^1_\unr(K_v,\mathbf{T}_\phi)\overset\simeq\longrightarrow\invlim_nH^1_\unr\bigl(K_v,\mathrm{Ind}^{G_K}_{G_{K_n}}(T_{\phi,n})\bigr)\overset\simeq\longrightarrow\invlim_nH^1_\unr(K_{n,v},T_{\phi,n}),    
\] 
where the first one is, as before, a combination of \eqref{T} and \cite[Proposition B.2.3]{Ru}. This establishes \eqref{Tproof} for $\bullet=\unr$ in this case. 

Finally, suppose $v\,|\,M$. 
If $w_\infty$ is a prime of $K_\infty$ above $w$, then $K_{\infty,w_\infty}/K_v$ is the unique $\Z_p$-extension of $K_v$, so by \cite[Proposition B.3.3]{Ru} there is an isomorphism
\begin{equation} \label{T-iso3}
\invlim_m H^1_\unr(K_{m,v},T_\phi)\overset\simeq\longrightarrow\invlim_m H^1(K_{m,v},T_\phi).
\end{equation}
By restriction to the inertia subgroup, the isomorphism \eqref{Sh00} 
induces an isomorphism 
\begin{equation} \label{Sh000}
H^1\bigl(I_v,\mathrm{Ind}^{G_K}_{G_{K_m}}(T_\phi)\bigr)\overset\simeq\longrightarrow\prod_{w\mid v}H^1\bigl(I_{m,w},(T_\phi{)}_w\bigr),
\end{equation}
where $I_{m,w}$ is the inertia subgroup of $\mathcal{D}_w$. By \eqref{T} and \cite[Proposition B.2.3]{Ru} (respectively, \eqref{Sh000}) for the first (respectively, second) isomorphism below, we have 
\[ 
H^1(I_v,\mathbf{T}_\phi)^{G_v/I_v}\overset\simeq\longrightarrow\biggl(\invlim_nH^1\bigl(I_v,\Ind^{G_K}_{G_{K_n}}(T_{\phi,n})\bigr)\!\biggr)^{G_v/I_v}\overset\simeq\longrightarrow\Biggl(\prod_{w\mid v}H^1\bigl(I_{m,w},(T_\phi{)}_w\bigr)\!\Biggr)^{G_v/I_v}.
\]
The rightmost group is contained in $\prod_{w\mid v}H^1\bigl(I_{m,w},(T_\phi{)}_w\bigr)^{G_{m,w}/I_{m,w}}$; notice that each of the local cohomology groups $H^1\bigl(I_{m,w},(T_\phi{)}_w\bigr)^{G_{m,w}/I_{m,w}}$ is trivial by \cite[Proposition B.3.3]{Ru}. Thus, $H^1(I_v,\mathbf{T}_\phi)^{G_v/I_v}$ is trivial and, in light of the equality 
\[
H^1_\unr(K_v,\mathbf{T}_\phi)=\ker\Bigl(H^1(K_v,\mathbf{T}_\phi)\longrightarrow 
H^1(I_v,\mathbf{T}_\phi)^{G_v/I_v}\!\Bigr),
\]
we conclude that 
\begin{equation} \label{T-iso5}
H^1_\unr(K_v,\mathbf{T}_\phi)=H^1(K_v,\mathbf{T}_\phi).
\end{equation} 
Finally, combining \eqref{T-iso0}, \eqref{T-iso3} and \eqref{T-iso5}, we obtain isomorphisms 
\[
H^1_\unr(K_v,\mathbf{T}_\phi)\overset\simeq\longrightarrow H^1 (K_v,\mathbf{T}_\phi)\overset\simeq\longrightarrow\invlim_mH^1(K_{v,m},T_\phi)\overset\simeq\longrightarrow\invlim_m
H^1_\unr(K_{v,m},T_\phi).
\]
This proves \eqref{Tproof} for $\bullet=\unr$ in this case as well, as desired. \end{proof}   

\subsubsection{Selmer groups over DVRs}\label{selcondsec} 

We introduce Selmer groups over discrete valuation rings and artinian local rings, following again Mazur--Rubin (\cite{MR}) and Howard (\cite{Ho1}, \cite{howard-bipartite}). 

Let $\mathfrak{P}\neq \p\Lambda$ be a height $1$ prime ideal of $\Lambda$ and denote by $\cO_\mathfrak{P}$ the integral closure of $\Lambda/\mathfrak{P}$ in $\overline{\mathbb{Q}}_p$. It follows that $\cO_\mathfrak{P}$ is the valuation ring of a finite extension of $F_\p$, say $\Phi_\mathfrak{P}$. Equip $T_\mathfrak{P}\defeq T_\phi\otimes_{\cO_\p}\cO_\mathfrak{P}$ with the diagonal action of $G_K$, where $G_K$ acts on $\cO_\mathfrak{P}$ via the canonical character $G_K\hookrightarrow\Lambda^\times$. Then define $V_\mathfrak{P}\defeq T_\mathfrak{P}\otimes_{\cO_\mathfrak{P}}\Phi_\mathfrak{P}$ and $A_\mathfrak{P}\defeq V_\mathfrak{P}/T_\mathfrak{P}$. For each prime $v\,|\,Dp$ of $K$, set 
$V_\mathfrak{P}^{(v)}\defeq T_\phi^{(v)}\otimes_{\cO_\p} \Phi_\mathfrak{P}$ and 
$V_\mathfrak{P}^{[v]}\defeq V_\phi/V_\phi^{(v)}$. Now we introduce the following subgroups of $H^1(K_v,V_\mathfrak{P})$: 
\begin{itemize}
\item if $v\nmid Dp$, then set
\[
H^1_\unr(K_v,V_\mathfrak{P})\defeq H^1\bigl(K_v^\unr/K_v,T_\mathfrak{P}^{I_v}\bigr)=\ker\Bigl(H^1(K_v,V_\mathfrak{P})\longrightarrow 
H^1(I_v,V_\mathfrak{P})\Bigr);
\]
\item if $v\,|\,Dp$, then set 
\[
\begin{split}
H^1_\ord(K_v,V_\mathfrak{P})&\defeq\im\Bigl(H^1(K_v,V_\mathfrak{P}^+)\longrightarrow H^1(K_v,V_\mathfrak{P})\Bigr)\\&\,\,=\ker\Bigl(H^1(K_v,V_\mathfrak{P})\longrightarrow H^1(K_v,{V}^-_\mathfrak{P})\Bigr).
\end{split}
\]
\end{itemize}
Define $H^1_\bullet(K_v,T_\mathfrak{P})$ and $H^1_\bullet(K_v,A_\mathfrak{P})$ for $\bullet\in\{\ord,\unr\}$ by propagating local conditions; namely, we let $H^1_\bullet(K_v,T_\mathfrak{P})$ (respectively, $H^1_\bullet(K_v,A_\mathfrak{P})$) be the inverse image (respectively, the image) of $H^1_\bullet(K_v,V_\mathfrak{P})$ in $H^1(K_v,T_\mathfrak{P})$ (respectively, in $H^1(K_v,A_\mathfrak{P})$) under the map induced in (local) cohomology by the natural injection $T_\mathfrak{P}\hookrightarrow V_\mathfrak{P}$ (respectively, surjection $V_\mathfrak{P}\twoheadrightarrow A_\mathfrak{P}$). 

\begin{definition}
Let $M\in\{V_{\mathfrak P},T_{\mathfrak P}, A_{\mathfrak P}\}$. The \emph{Selmer group} $\Sel(K,M)$ is the group of all $s\in H^1(K,M)$ such that
\begin{itemize}
\item $\loc_v(s)\in H^1_\unr(K_v,M)$ for all $v\nmid Dp$;
\item $\loc_v(s)\in H^1_\ord(K_v,M)$ for all $v\,|\,Dp$. 
\end{itemize}
\end{definition} 

\begin{remark}
Let $\mathfrak{P}_0$ be the height $1$ prime ideal of $\Lambda$ corresponding to the trivial character $\mathds{1}$ of $\Gal(K_\infty/K)$, so that $\mathfrak{P}_0+\p\Lambda$ is the maximal ideal of $\Lambda$. Equivalently, $\mathfrak P_0=(\gamma-1)$ is the augmentation ideal of $\Lambda$. Fix $M\in \{A_\phi,T_\phi,V_\phi\}$. Then there is an isomorphism $\Lambda/{\mathfrak{P}_0}\simeq\cO_\p$, so $\cO_{\mathfrak{P}_0}=\cO_\mathfrak{p}$ and $F_{\mathfrak{P}_0}=F_\mathfrak{p}$, the fraction field of $\cO_\mathfrak{p}$; therefore, the equality $M_{\mathfrak{P}_0}=M$ holds in all cases. Then we have defined two subgroups $\Sel(K,T_\phi)$ and $\Sel(K,T_{\mathfrak{P}_0})$ of $H^1(K,T_\phi)$ and two subgroups $\Sel(K,A_\phi)$ and $\Sel(K,A_{\mathfrak{P}_0})$ of $H^1(K,A_\phi)$; 
under Assumption \ref{ass}, these two groups can be shown to be equal 
(and a similar result could also be obtained for the pairs given by $\Sel(K_n,A_\phi)$ and $\Sel(K,\mathbf{A}[{\mathfrak{P}_n}])$ and by $\Sel(K_n,T_\phi)$ and $\Sel(K,\mathbf{T}/{\mathfrak{P}_n}\mathbf{T})$, where $\mathfrak{P}_n=(\gamma^n-1)$ for some integer $n\geq 1$ as in \cite[Section 5]{BLV}). 
\end{remark}

For all integers $j\geq1$, set $T_{\mathfrak{P},j}\defeq T_\mathfrak{P}/\mathfrak{p}^jT_\mathfrak{P}$ and $A_{\mathfrak{P},j}\defeq A_\mathfrak{P}[\mathfrak{p}^j]$. Define the groups $H^1_\bullet(K_v,T_{\mathfrak{P},j})$ and $H^1_\bullet(K_v,A_{\mathfrak{P},j})$ for $\bullet\in\{\ord,\unr\}$ by propagating local conditions; namely, we let $H^1_\bullet(K_v,T_{\mathfrak{P},j})$ (respectively, $H^1_\bullet(K_v,A_{\mathfrak{P},j})$) be the image (respectively, the inverse image) of $H^1_\bullet(K_v,T_\mathfrak{P})$ in $H^1(K_v,T_{\mathfrak{P},j})$ (respectively, in $H^1(K_v,A_{\mathfrak{P},j})$) under the map induced in cohomology by the natural surjection  $T_\mathfrak{P}\twoheadrightarrow T_{\mathfrak{P},j}$ (respectively, injection 
$A_{\mathfrak{P},j}\hookrightarrow A_\mathfrak{P}$).

\begin{definition} 
Let $S\in \mathcal{P}_n$ and $M=T_\mathfrak{P}/\mathfrak{p}^nT_\mathfrak{P}$ or 
$M=A_\mathfrak{P}[\mathfrak{p}^n]$. The \emph{Selmer group} $\Sel_S(K,M)$ is the group of all $s\in H^1(K,M)$ such that
\begin{itemize}
\item $\loc_v(s)\in H^1_\unr(K_v,M)$ for all $v\nmid SDp$;
\item $\loc_v(s)\in H^1_\ord(K_v,M)$ for all $v\,|\,SDp$. 
\end{itemize}
\end{definition} 

For $S=1$, we just write $\Sel(K,M)$ for $\Sel_1(K,M)$. These groups satisfy the conditions listed by Mazur--Rubin in \cite{MR} and by Howard in \cite{Ho1} and \cite{howard-bipartite}. More precisely, one can use the pairing in \cite[Proposition 3.1]{Nek} to construct a perfect, symmetric $\cO_\p$-linear pairing 
\[
{(\cdot,\cdot)}_\mathfrak{P}:T_\mathfrak{P}\times T_\mathfrak{P}\longrightarrow \cO_\p(1)
\]
which satisfies $(x^g,y^{\tau g\tau^{-1}})_\mathfrak{P}=(x,y)^g$ for all $g\in G_K$ and all
$x,t\in T_\mathfrak{P}$ (see, \emph{e.g.}, \cite[Lemma 2.1.1]{Ho1} and \cite[\S5.1]{LV-kyoto}). The Selmer conditions are cartesian for this pairing, since they are defined by propagation (see, \emph{e.g.}, \cite[Lemma 3.7.1]{MR}), while the self-duality condition follows by combining the self-duality for the Selmer conditions on $V_\mathfrak{P}$ and the fact that the local conditions on $T_\mathfrak{P}$ are defined by propagation.  

\subsubsection{Bipartite Euler systems} \label{secbip}

For any integer $n\geq 1$ and any height $1$ prime ideal $\mathfrak{P}$ of $\Lambda$ 
with $\mathfrak{P}\neq\p\Lambda$, there are families 
\begin{equation} \label{Bipartite}
\Bigl\{{\kappa}_{\mathfrak{P}}(S)\in \Sel_S(K,T_{\mathfrak{P},n})\;\big|\; S\in \mathcal{P}_{2n}^\mathrm{indef}\Bigr\},\quad\Bigl\{{\lambda}_{\mathfrak{P}}(S)\in \cO_\mathfrak{P}/\p^n\cO_\mathfrak{P}\;\big|\; S\in\mathcal{P}_{2n}^\mathrm{def}\Bigr\}
\end{equation}
obtained by applying the reduction map $\cO/\p^{2n}[\![G_\infty]\!]\rightarrow \cO_\mathfrak{P}/\p^n\cO_\mathfrak{P}$ to the elements in \S \ref{explicit-subsec}. 
These classes satisfy the first and second reciprocity laws; therefore, in light of
\cite[Theorem 6.3]{ChHs2}, which proves the validity of \cite[Hypothesis 2.3.1]{howard-bipartite}, they form a bipartite Euler system of odd type for $T_{\mathfrak{P},n}$ in the sense of \cite[Definition 2.3.2]{howard-bipartite} (explicitly, this condition means that for all $c\in H^1(K,T_{\phi,1})\smallsetminus\{0\}$ and all integers $n\geq1$ 
there are infinitely many $n$-admissible primes $\ell$ such that $\loc_\ell(c)\neq0$). A generalization of \cite[Lemma 3.3.6]{howard-bipartite} shows that this bipartite Euler system is also free in the sense of \cite[Definition 2.3.6]{howard-bipartite}. We provide the details of this freeness result. In accord with \cite[Definition 2.3.6]{howard-bipartite}, we must show that there is a free $\cO_\mathfrak{P}/\mathfrak{p}^n\cO_\mathfrak{P}$-submodule of $\Sel_{S}(K,T_{\mathfrak{P},n})$ containing $\kappa_\mathfrak{P}(S)$. Let $e_\mathfrak{P}$ be the ramification index of $\cO_\mathfrak{P}$ over $\mathcal{O}_\mathfrak{p}$, let $\mathfrak{m}_\mathfrak{P}=(\pi_\mathfrak{P})$ be the maximal ideal of $\cO_\mathfrak{P}$ and let $k_\mathfrak{P}\defeq\cO_\mathfrak{P}/\mathfrak{m}_\mathfrak{P}$ be its residue field. Moreover, for all integers $j\geq0$ set $R_j\defeq\mathcal{O}_\mathfrak{P}/\mathfrak{p}^j\mathcal{O}_\mathfrak{P}$; then the length of $R_j$ as an $\cO_\mathfrak{P}$-module is $je_\mathfrak{P}$. By \cite[Proposition 2.2.7]{howard-bipartite} and the properties of Selmer conditions recalled in \S \ref{selcondsec}, there are isomorphisms $\Sel_{S}(K,T_{\mathfrak{P},2n})\simeq R_{2n}\oplus N\oplus N$ and $\Sel_{S}(K,T_{\mathfrak{P},k})\simeq R_n\oplus M\oplus M$ for finite $\mathcal{O}_\mathfrak{P}$-modules $N$ and $M$ (in the second case, of course, $S$ is viewed as an element in $\mathcal{P}_n^\mathrm{indef}$). Since $R_n$ has length $ne_\mathfrak{P}$, if $\mathfrak{m}^{ne_\mathfrak{P}-1}M\neq0$, then $\kappa_\mathfrak{P}(S)=0$ by \cite[Proposition 2.3.5]{howard-bipartite}, and there is nothing to prove. Therefore, assume $\mathfrak{m}^{ne_\mathfrak{P}-1}M=0$. By \cite[Lemma 2.2.6]{howard-bipartite}, there is a commutative triangle
\begin{equation} \label{triangle-eq}
\xymatrix@C=45pt@R=45pt{\Sel_{S}(K,T_{\mathfrak{P},2n})\ar[r]^-{\pi_\mathfrak{P}^{ne_\mathfrak{P}}\cdot} \ar[d]& \Sel_{S}(K,T_{\mathfrak{P},2n})\bigl[\mathfrak{m}_\mathfrak{P}^{ne_\mathfrak{P}}\bigr]\\
\Sel_{S}(K,T_{\mathfrak{P},n})\ar[ru]^-\simeq}
\end{equation}
in which the horizontal map is multiplication by $\pi_\mathfrak{P}^{ne_\mathfrak{P}}$, the vertical map is induced by the canonical projection and the diagonal map is an isomorphism by \cite[Lemma 2.2.6]{howard-bipartite}. Since $\mathfrak{m}^{ne_\mathfrak{P}-1}M=0$, the module $\mathfrak{m}^{ne_\mathfrak{P}-1}\Sel_{S}(K,T_{\mathfrak{P},n})$ is cyclic, isomorphic to $k_{\mathfrak P}$, and the diagonal isomorphism in \eqref{triangle-eq} shows that $\mathfrak{m}_\mathfrak{P}^{ne_\mathfrak{P}-1}\Sel_{S}(K,T_{\mathfrak{P},2n})[\mathfrak{m}_\mathfrak{P}^{ne_\mathfrak{P}}]$ is also cyclic. In particular, it follows that $\mathfrak{m}_\mathfrak{P}^{ne_\mathfrak{P}}N=0$, so the image of $N$ under the vertical arrow in \eqref{triangle-eq} is trivial and the image of this vertical map is free of rank $1$. Since $\kappa_\mathfrak{P}(S)$ is contained in this image, this concludes the discussion on the freeness of the bipartite Euler system \eqref{Bipartite}.  

The next result, which is commonly referred to as ``level raising and rank lowering'', is well known (see, \emph{e.g.}, \cite[\S4.2]{BD-IMC}, \cite[\S7.2.6]{BLV}, \cite[Proposition 5.4]{Zhang}).  

\begin{lemma}\label{ranklowering}
If $\Sel_S(K,T_{\mathfrak{P},1})\neq 0$, then for any integer $n\geq1$ there are infinitely many $n$-admissible primes $\ell\in\mathcal{P}_n$ such that $\dim_{k_\mathfrak{P}}\bigl(\Sel_{S\ell}(K,T_{\mathfrak{P},1})\bigr)=\dim_{k_\mathfrak{P}}\bigl(\Sel_{S}(K,T_{\mathfrak{P},1})\bigr)-1$.
\end{lemma}

\begin{proof} Pick $c\in\Sel_{S}(K,T_{\mathfrak{P},1})\smallsetminus\{0\}$ and choose an admissible prime $\ell$ such that $\loc_\ell(c)\neq 0$, then apply \cite[Corollary 2.2.10]{howard-bipartite}. \end{proof}

Recall that $\mathfrak{P}_0$ is the augmentation ideal of $\Lambda$. What follows (Propositions \ref{maximal}, \ref{maximal1}, \ref{maximal2} and Corollary \ref{coro(c)}) contains the proof of the validity of the analogue of condition (c) in \cite[Theorem 3.2.3]{howard-bipartite} in our setting; the verification of this property for elliptic curves is \cite[Lemma 3.6]{BCK} (however, notice the typo in the proof, where $\wp$ should be
replaced by the maximal ideal of the Iwasawa algebra; see also \cite[Theorem 7.5]{CHKLL}). In our proofs, we adapt arguments from \cite[\S7.2.4]{BLV}. 

\begin{proposition} \label{maximal}
Assume $\mu(D)=-1$. For any integer $n\geq 1$, the set 
\[\bigl\{\lambda_\phi(S)\in\Lambda/\p\Lambda\mid S\in\mathcal{P}_n^\mathrm{def}(\phi)\bigr\}\] contains an element whose image in $\Lambda/(\mathfrak{P}_0+\p\Lambda)$ is non-zero. 
\end{proposition}

\begin{proof} 
First of all, observe that the integer $r\defeq\dim_{\kappa_\p}\Sel(K,T_{\phi,1})$ is always even. 
To show this, assume by contradiction that $r$ is odd. Applying recursively 
Lemma \ref{ranklowering}, we can find a sequence of $r$ distinct $1$-admissible primes 
$\ell_1,\dots,\ell_r$ such that if $S=\prod_{i=1}^r\ell_i$ and $\phi_S$ is the level raising of $\phi$ at $S$, 
then $\Sel_S(K,T_{\phi,1})=\Sel(K,T_{\phi_S,1})=0$. However, since $\epsilon_K(NS)=-1$ because $r$ is odd, $\Sel(K,T_{\phi_S,1})\neq 0$ by a result of Skinner--Urban (\cite[Theorem 3.35]{SU}) and Fouquet--Wan (\cite[Corollary 1.9]{FW}), which is a contradiction. 

If $r=0$, then $\Sel(K,A_\phi)$ is trivial, and then $L(f/K,k/2)\neq 0$ by the results of Skinner--Urban and Fouquet--Wan referred to above. By \cite[Theorem 4.5]{wang}, one has $v_\p\bigl(L^\mathrm{alg}(f/K)\bigr)=0$, so $L^\mathrm{alg}(f/K)$ is a  $\p$-adic unit. Comparing with the interpolation formulas in \S\ref{L-funct}, we conclude that $\lambda_\phi$ is also a $\mathfrak{p}$-adic unit, and the proposition follows with $S=1$. 
%$L^\mathrm{alg}(f/K)\doteq\mathds{1}(\lambda_\phi\lambda_\phi^*)$ is a $\p$-adic unit (here the symbol $\doteq$ indicates that equality holds up to comparatively less important factors that can be explicitly described). Thus, 

If $r\geq 2$, then, by Lemma \ref{ranklowering}  there exist $n$-admissible primes $\ell_1$ and $\ell_2$ and such that $\dim_{\kappa_{\p}}\Sel(K,T_{\phi_{\ell_1\ell_2},1})=r-2$. Iterating this process if necessary, one shows that there is an $n$-admissible integer $S\in\mathcal{P}_n^\mathrm{def}$ such that $\Sel_S(K,T_{\phi,1})=\Sel(K,T_{\phi_S,1})=0$. 
The quaternionic form $\phi_S$ modulo $\mathfrak{p}$ can be lifted in characteristic zero 
to a modular form $h\in S_k(NS)$ of level $NS$ and weight $k$ by \cite[Lemma 6.10]{Serre-Deligne}. As in the $r=0$ case, $L(h/K,k/2)\neq0$ and $v_\p\bigl(L^\mathrm{alg}(h/K)\bigr)=0$ and the interpolation formulas it follows that $\lambda_\phi (S)$ has non-zero image in $\Lambda/(\mathfrak{P}_0+\p\Lambda)$, concluding the proof. \end{proof}

\begin{proposition}\label{maximal1} Assume $\mu(D)=1$. 
For any integer $n\geq 1$, the set \[\bigl\{\lambda_\phi(S)\in \Lambda/\p\Lambda\mid S\in\mathcal{P}_n^\mathrm{def}(\phi)\bigr\}\] contains an element whose image in $\Lambda\big/(\mathfrak{P}_0+\p\Lambda)$ is non-zero.
\end{proposition}

\begin{proof} We want to show that there exists a level raising $\phi_S$ of $f$ at a suitable $S\in \mathcal{P}_n^\mathrm{def}$ such that the reduction $\overline{\lambda_\phi(S)}\in\Lambda/(\mathfrak{P}_0+\p\Lambda)$ is non-zero. We argue by induction on the integer $r\defeq\dim_{\kappa_\p}\Sel(K,T_{\phi,1})$, with an argument similar to the one in the proof of Proposition \ref{maximal}. By \cite[Theorem 3.15]{LPV}, we can assume that $r$ is odd. 

Suppose $r=1$. Choose an $n$-admissible prime $\ell$ as in Lemma \ref{ranklowering}, 
so that $\Sel_\ell(K,T_{\phi,1})=0$. Then, as before, \cite[Theorem 4.5]{wang} 
show that $v_\p\bigl(L^\mathrm{alg}(f/K)\bigr)=0$, so $L^\mathrm{alg}(f/K)$ is a  $\p$-adic unit and comparing with the interpolation formulas in \S\ref{L-funct}, we conclude that $\lambda_\phi(\ell)$ is also a $\mathfrak{p}$-adic unit.

Now suppose $r\geq3$. Applying Lemma \ref{ranklowering} recursively, one finds $S\in \mathcal{P}_{n}$ that is a product of an even number of distinct primes such that $\Sel_S(K,T_{\phi,1})=\Sel(K,T_{\phi_S,1})$ has dimension $1$. Choose $\ell$ as in the case $r=1$, and then $\lambda_\phi(S\ell)$ does the job. \end{proof}

Let $\F\defeq\cO_\p/\p\cO_\p$ be the residue field of $\cO_\p$ and let $\pi$ be a uniformizer of $\cO_\p$. Fix an isomorphism of $\mathcal{O}_\mathfrak{p}$-algebras $\Lambda\simeq\mathcal{O}_\mathfrak{p}[\![T]\!]$ by choosing a topological generator $\gamma$ of $G_\infty$ and sending $\gamma$ to $1+T$. Propositions \ref{maximal1} and \ref{maximal} show that there exists $\lambda_\phi(S)\in \F\llbracket T\rrbracket$ with $S\in\mathcal{P}_1^\mathrm{def}$ such that $\lambda_\phi(S)=a+Tb$ for some $a\in \F^\times$ and $b\in \F\llbracket T\rrbracket$.

Now let $\mathfrak{P}$ be a height $1$ prime ideal of $\Lambda$, with residue field $k_\mathfrak{P}\defeq\cO_\mathfrak{P}/\mathfrak{m}_\mathfrak{P}$. Since $\mathfrak{P}+\p\Lambda$ is contained in the maximal ideal $\mathfrak{m}_\Lambda=\mathfrak{P}_0+\p\Lambda$ of $\Lambda$, the class $\lambda_\phi(S)$ is clearly non-zero in $\Lambda\big/(\mathfrak{P}+\p\Lambda)$. However, a stronger statement holds. 

\begin{proposition}\label{maximal2}
Let $\mathfrak{P}\neq\p\Lambda$ be a height $1$ prime ideal of $\Lambda$. The image of $\lambda_\phi(S)\in\F\llbracket T\rrbracket$ in $k_\mathfrak{P}$ is non-zero.   
\end{proposition}

\begin{proof} The ideal $\mathfrak{P}$ corresponds under the isomorphism $\Lambda\simeq\cO_\p\llbracket T\rrbracket$ to an ideal of $\cO_\p\llbracket T\rrbracket$ generated by a distinguished polynomial, say $G$. Thus, all the roots of $G$ belong to the maximal ideal of the ring of integers of $\overline{\Q}_p$; choose any such root $\xi$, so that the character $\chi:\Lambda\rightarrow\overline\Q_p$ with $\chi(T)\defeq\xi$ induces an isomorphism between $\Lambda/\mathfrak{P}$ and a subring of $\Q_p(\xi)$ whose integral closure is $\cO_\mathfrak{P}$. The map $\chi$ induces an injection $\bar\chi:\Lambda/(\mathfrak{P}+\p\Lambda)\hookrightarrow\cO_\mathfrak{P}/\p\cO_\mathfrak{P}$ and $\lambda_\phi(S)=a+Tb$ is taken to $\bar\chi\bigl(\lambda_\phi(S)\bigr)=a+\xi\bar\chi(b)\pmod{\p}$. Since $\xi$ belongs to the maximal ideal $\mathfrak{m}_\mathfrak{P}$ of $\cO_\mathfrak{P}$, we see that $\chi\bigl(\lambda_\phi(S)\bigr)\equiv a\pmod{\mathfrak{m}_\mathfrak{P}}$ is non-zero in $k_\mathfrak{P}$. \end{proof}

\begin{corollary}\label{coro(c)}
For any integer $n\geq 1$ and any height $1$ prime ideal $\mathfrak{P}$ of $\Lambda$, the set 
\[
\bigl\{\lambda_\phi(S)\mid S\in \mathcal{P}_n^\mathrm{def}(\phi)\bigr\}
\] 
contains an element having non-trivial image in $k_\mathfrak{P}$ if $\mathfrak{P}\neq\p\Lambda$ or in $\Lambda/\mathfrak P$ if $\mathfrak{P}=\p\Lambda$.
\end{corollary}

\begin{proof} For $\mathfrak{P}\neq\p\Lambda$ the claim follows from Proposition \ref{maximal2}, while the $\mathfrak{P}=\mathfrak{p}\Lambda$ case is clear from Proposition \ref{maximal}. \end{proof}

We recall the following algebraic notion, which was introduced in \cite[Section 2]{howard-bipartite}.
\begin{definition} \label{index-def}
Let $(R,\mathfrak m)$ be a principal artinian local ring, let $B$ be an $R$-module and let $b\in B$. The \emph{index of divisibility of $b$ in $B$} is 
\[ \mathrm{ind}(b,B)\defeq\sup\bigl\{i\in\N\mid b\in\mathfrak m^iB\bigr\}\in\N\cup\{\infty\}. \]
\end{definition}
For each pair of integers $k,j$ with $1\leq k\leq j$ and each height $1$ prime 
ideal $\mathfrak P$ of $\Lambda$ with $\mathfrak{P}\neq\mathfrak{p}\Lambda$, set 
\[
\delta_\mathfrak{P}(k,j)\defeq\min\Bigl\{\mathrm{ind}\bigl(\lambda_\phi(S),\cO_\mathfrak{P}/\mathfrak{p}^k\cO_\mathfrak{P}\bigr)\mid S\in\mathcal{P}_j^\mathrm{indef}\Bigr\}\leq \infty;
\]
moreover, using the fact that the sequence $\bigl(\delta_\mathfrak{P}(k,j)\bigr)_{j\geq1}$ is non-decreasing, define 
\begin{equation}\label{deltaP}
\delta_\mathfrak{P}(k)\defeq\lim_{j\rightarrow\infty}\delta_\mathfrak{P}(k,j).
\end{equation}

\begin{corollary} \label{corodelta}
$\delta_\mathfrak{P}(k,j)=0$ for all $j\geq k$ and all $k\geq 1$, whence $\delta_\mathfrak{P}(k)=0$ for all $k\geq 1$.
\end{corollary}

\begin{proof} This is an immediate consequence of Corollary \ref{coro(c)}. \end{proof}

\subsubsection{Results on Selmer groups}

We collect some results on the Selmer groups introduced before that will be used in our proof of the main conjectures. 

\begin{proposition}\label{len-def}
If $\mu(D)=-1$ and $\mathcal{L}_\p(f)$ is non-zero in $\cO_\mathfrak{P}$, then
\[ \length_{\cO_\mathfrak{P}}\bigl(\Sel(K,A_\mathfrak{P})\bigr)=\length_{\cO_\mathfrak{P}}\bigl(\cO_\mathfrak{P}/\mathcal{L}_\p(f)\cO_\mathfrak{P}\bigr). 
\]
\end{proposition}

\begin{proof} Let $k\geq1$ be the smallest positive integer such that $\mathcal{L}_\p(f)$ has non-zero image in $\cO_\mathfrak{P}/\p^k\cO_\mathfrak{P}$ and set 
$T_{\mathfrak{P},k}\defeq T_\mathfrak{P}/\p^kT_\mathfrak{P}$. 
The results recalled in \S\ref{secbip} show that the families \eqref{Bipartite} form a free bipartite Euler system for $T_{\mathfrak{P},k}$.
By \cite[Proposition 2.2.7]{howard-bipartite}, we know that
\begin{equation}\label{structure}\Sel(K,T_{\mathfrak{P},k})\simeq M\oplus M\simeq \Sel(K,A_\mathfrak{P})[\p^k]
\end{equation}
for some finite $\cO_\mathfrak{P}/\p^k\cO_\mathfrak{P}$-module $M$.
%(notice that $\mathfrak{N}^\mathrm{even}$ in \cite{howard-bipartite} corresponds to  $\mathcal{P}_n^\mathrm{def}(\phi)$ in the present paper). 
Let $\overline{\mathcal{L}_\p(f)}$ denotes the image of $\mathcal{L}_\p(f)$ in $\cO_\mathfrak{P}/\p^k\cO_\mathfrak{P}$. 
Applying \cite[Theorem 2.5.1]{howard-bipartite} with $R=\cO_\mathfrak{P}/\p^k\cO_\mathfrak{P}$, $T=T_{\mathfrak{P},k}$ and $\mathfrak{n}=1$, and noticing that the integer $\delta $ appearing in \emph{loc. cit.} is $\delta_\mathfrak{P}(k)$ introduced in \eqref{deltaP}, which is zero by Corollary \ref{corodelta}, 
gives the equality
\begin{equation} \label{lengths}
\mathrm{ind}\bigl(\overline{\mathcal{L}_\p(f)}, \cO_\mathfrak{P}/\p^k\cO_\mathfrak{P}\bigr)=2\cdot\length_{\cO_\mathfrak{P}}(M).
\end{equation}
Since $\overline{\mathcal{L}_\p(f)}\not=0$, the left hand side of \eqref{lengths} is strictly smaller than $k$, and hence by \eqref{structure} we must have 
\begin{equation}\label{selstructure}
M\oplus M\simeq \Sel(K,A_\mathfrak{P}).
\end{equation}
It is clear from Definition \ref{index-def} that $\mathrm{ind}(b,R)=\length_R(R/bR)$ for every non zero 
$b\in R$. Therefore, there is an equality
\begin{equation}\label{index}
\mathrm{ind}\bigl(\overline{\mathcal{L}_\p(f)}, \cO_\mathfrak{P}/\p^k\cO_\mathfrak{P}\bigr)=\length_{\cO_\mathfrak{P}}\bigl(\cO_\mathfrak{P}/\mathcal{L}_\p(f)\cO_\mathfrak{P}\bigr).
\end{equation} 
Combining \eqref{lengths}, \eqref{selstructure} and \eqref{index} completes the proof of the proposition. \end{proof}

The next result deals with the indefinite case.

\begin{proposition}\label{len-indef}
Assume that $\mu(D)=1$ and $\kappa_\infty$ is non-zero in $\Sel(K,T_\mathfrak{P})$.
\begin{enumerate}
\item $\Sel(K,T_\mathfrak{P})$ is a free $\cO_\mathfrak{P}$-module of rank $1$;
\item $\Sel(K,A_\mathfrak{P})$ has $\cO_\mathfrak{P}$-corank $1$;
\item If $\Sel(K,A_\mathfrak{P}{)}_\divv$ is the maximal $\cO_\mathfrak{P}$-divisible submodule of $\Sel(K,A_\mathfrak{P})$, then there is an equality
\[ 
\length_{\cO_\mathfrak{P}}\bigl(\Sel(K,A_\mathfrak{P})/\Sel(K,A_\mathfrak{P}{)}_\divv\bigr)=2\cdot\length_{\cO_\mathfrak{P}}\bigl(\Sel(K,T_\mathfrak{P})/\cO_\mathfrak{P}\cdot\kappa_\infty\bigr).
\]
\end{enumerate}
\end{proposition}

\begin{proof} 
Let $k\geq1$ be the smallest integer such that $\kappa_\infty$ has non-zero image via the composition  
\[ \Sel(K,T_\mathfrak{P})\longepi\Sel(K,T_\mathfrak{P})/\p^k\longmono\Sel(K,T_{\mathfrak{P},k}), \] 
where
$T_{\mathfrak{P},j}\defeq T_\mathfrak{P}/\p^jT_\mathfrak{P}$ for all integers $j\geq1$, the first map is the canonical projection and the second map, induced by the quotient map $T_\mathfrak{P}\rightarrow T_{\mathfrak{P},k}$, is injective by \cite[Lemma 3.3.2]{howard-bipartite}; also, sometimes, 
as done for $\Sel(K,T_\mathfrak{P})/\p^k$ above, for a $\mathcal{O}_\mathfrak{p}$-module
$M$ and an integer $j\geq 1$, we abbreviate $M/\mathfrak{p}^jM$ by $M/\mathfrak{p}^j$). As in the proof of Proposition \ref{len-def}, the families \eqref{Bipartite} form a free bipartite Euler system for $T_{\mathfrak{P},k}$. 
Since $1$ belongs to $\mathcal{P}^\mathrm{indef}$, by \cite[Proposition 2.2.7]{howard-bipartite} we know that
\begin{equation}\label{structure2}
\Sel(K,T_{\mathfrak{P},j})\simeq\cO_\mathfrak{P}/\p^j\cO_\mathfrak{P}\oplus M\oplus M\simeq \Sel(K,A_\mathfrak{P})[\p^j]
\end{equation}
for all $j\geq1$ and a certain $\cO_\mathfrak{P}/\p^j\cO_\mathfrak{P}$-module $M$, where the isomorphism on the right is induced, in light of \cite[Lemma 2.2.6]{howard-bipartite}, by the isomorphism $T_{\mathfrak{P},j}\simeq A_\mathfrak{P}[\p^j]$; moreover, these isomorphisms are compatible as $j$ varies over all positive integers. By applying \cite[Theorem 2.5.1]{howard-bipartite} with $R=\cO_\mathfrak{P}/\p^k$ and $\mathfrak{n}=1$, 
and using Corollary \ref{corodelta} as in the proof of Proposition 
\ref{len-def}, we get the equality
\begin{equation} \label{lengths-2}
\mathrm{ind}\bigl({\kappa}_\mathfrak{P}(1),\Sel(K,T_{\mathfrak{P},k})\bigr)=\length_{\cO_\mathfrak{P}}(M).
\end{equation}
%Note that if $\mathfrak{P}=\mathfrak{P}_0$, then $k=1$ and the left-hand side of \eqref{lengths-2} is $0$ by Kolyvagin's conjecture in rank $1$ (see \cite[Theorem 3.16]{LPV}).
Since $\kappa_\infty$ has non-zero image in $\Sel(K,T_\mathfrak{P})/\p^k$, the left-hand side of \eqref{lengths-2} is strictly smaller than $k$. This together with \eqref{structure2} implies that 
\[
\Sel(K,T_\mathfrak{P})\simeq\varprojlim_j \Sel(K,A_\mathfrak{P})[\p^j]
\] 
is a torsion-free $\cO_\mathfrak{P}$-module of rank $1$, which is the content of parts (1) and (2). As for part (3), note that \eqref{structure2} also implies that
\[
\length_{\cO_\mathfrak{P}}\bigl(\Sel(K,A_\mathfrak{P})/\Sel(K,A_\mathfrak{P}{)}_\divv\bigr)=\length_{\cO_\mathfrak{P}}(M\oplus M),
\]
and \eqref{lengths-2} implies that
\[ \length_{\cO_\mathfrak{P}}\bigl(\Sel(K,T_\mathfrak{P})/\cO_\mathfrak{P}\cdot\kappa_\infty\bigr)=\length_{\cO_\mathfrak{P}}(M).
\]
This completes the proof.
\end{proof}

\begin{theorem}[Control theorem]\label{control theorem}
For all but finitely many height $1$ prime ideals $\mathfrak P$ of $\Lambda$, the natural map $\mathbf{T}_\phi\rightarrow T_\mathfrak{P}$ and its dual $A_\mathfrak{P}\rightarrow \mathbf{A}_\phi [\mathfrak{P}]$ induce maps
\[ \Sel(K,\mathbf{T}_\phi)\big/\mathfrak{P}\Sel(K,\mathbf{T}_\phi)\longrightarrow\Sel(K,T_\mathfrak{P}),\quad\Sel(K,A_\mathfrak{P})\longrightarrow\Sel(K,\mathbf{A}_\phi)[\mathfrak{P}] \]
with finite kernels and cokernels, whose orders are bounded by a constant that depends on $[\cO_\mathfrak{P}:\Lambda/\mathfrak{P}]$ only. The first map is always injective.
\end{theorem}

\begin{proof} As is shown in \cite[Proposition 5.3.14]{MR}, it is enough to check the result locally at each prime $v$ of $K$. For primes $v\nmid Dp$, the relevant local results are due to Mazur--Rubin (\cite[Lemma 5.3.13]{MR}), while for primes dividing $Dp$ the proof goes as in \cite[Proposition 3.3.1]{howard-bipartite}. See also \cite[Proposition 3.4]{LV-kyoto} for details. \end{proof}

\subsubsection{Proof of Theorem $\ref{IMC-def}$}

We follow an argument due to Mazur--Rubin (\cite{MR}) and Howard (\cite{Ho1}). Let $\mathfrak{P}\neq\p\Lambda$ be any height $1$ prime ideal of $\Lambda$ such that
\begin{itemize}
\item $\mathfrak P$ is not a prime divisor of $\Sel(K_\infty,A_\phi)^\vee$;
\item Theorem \ref{control theorem} holds for $\mathfrak{P}$;
\item $\mathcal{L}_\p(f)$ has non-zero image in $\Lambda/\mathfrak{P}$.    
\end{itemize}
These three conditions exclude only finitely many height $1$ prime ideals of $\Lambda$. By the proof of Proposition \ref{len-def}, $\Sel(K,A_\mathfrak{P})$ has corank $0$ and $\Sel(K,T_\mathfrak{P})$ has rank $0$ over $\cO_\mathfrak{P}$, and then Theorem \ref{control theorem} ensures that
\[ 0=\rank_{\cO_\mathfrak{P}}\bigl(\Sel(K,A_\mathfrak{P})^\vee\bigr)=\rank_\Lambda\bigl(\Sel(K,\mathbf{A}_\phi)^\vee\bigr)
=\rank_\Lambda\bigl(\Sel(K_\infty,A_\phi)^\vee\bigr),
\]
and
\[ 0=\rank_{\cO_\mathfrak{P}}\Sel(K,T_\mathfrak{P})=\rank_\Lambda\Sel(K,\mathbf{T}_\phi)=
\rank_\Lambda\Sel(K_\infty,T_\phi). 
\]
This proves part (1) of Theorem \ref{IMC-def}.

Now fix a height $1$ prime ideal $\mathfrak{P}$ of $\Lambda$ with $\mathfrak{P}\neq \p\Lambda$, say $\mathfrak{P}=(F) 
$ for some distinguished polynomial $F$. For each integer $m\geq1$, denote by $\mathfrak{P}_m$ the ideal of $\Lambda$ generated by $F+\pi^m$, where $\pi$ is a uniformizer of $\cO_\p$. If $m$ is taken to be sufficiently large, then
\begin{itemize}
\item $\mathfrak{P}_m$ is not a prime divisor of $\Sel(K,\mathbf{A}_\phi)^\vee$;
\item Theorem \ref{control theorem} holds for $\mathfrak{P}_m$;
\item $\mathcal{L}_\p(f)$ has non-zero image in $\Lambda/\mathfrak{P}_m$;
\item $\Lambda/\mathfrak{P}\simeq \Lambda/\mathfrak{P}_m$ as rings (by Hensel's lemma).
\end{itemize}
By Proposition \ref{len-def}, there is an equality
\[ \length_{\cO_\p}\Sel(K,A_{\mathfrak{P}_m})=\length_{\cO_\p}\bigl(\cO_{\mathfrak{P}_m}/\mathcal{L}_\p(f)\bigr). 
\]
On the other hand, arguments in the proof of \cite[Theorem 5.3.10]{MR} show that
\begin{equation} \label{1}
\length_{\cO_\p}\Sel(K,A_{\mathfrak{P}_m})=m\cdot\rank_{\cO_\p}(\cO_{\mathfrak{P}_m})\cdot\ord_\mathfrak{P}\bigl(\cchar_\Lambda\Sel(K,\mathbf{A}_\phi)^\vee\bigr)+O(1),
\end{equation}
where $O(1)$ denotes an integer bounded independently of $m$; moreover, a calculation gives
\begin{equation}\label{2}
\length_{\cO_\p}\bigl(\cO_{\mathfrak{P}_m}/\mathcal{L}_\p(f)\bigr)=m\cdot\rank_{\cO_\p}(\cO_{\mathfrak{P}_m})\cdot\mathrm{ord}_\mathfrak{P}\bigl(\mathcal{L}_\p(f)\bigr).
\end{equation}
Therefore, equating \eqref{1} and \eqref{2} and letting $m$ go to infinity yields the equality
\[ \ord_\mathfrak{P}\bigl(\cchar_\Lambda\Sel(K_\infty,A_\phi)^\vee\bigr)=
\ord_\mathfrak{P}\bigl(\cchar_\Lambda\Sel(K,\mathbf{A}_\phi)^\vee\bigr)=
\ord_\mathfrak{P}\bigl(\mathcal{L}_\p(f)\bigr). \]
As in \cite{Ho1} and \cite{MR}, the case $\mathfrak{P}=\p\Lambda$ can be treated similarly, now considering the ideals $\mathfrak{P}_m\defeq\p\Lambda+(T^m)$. This completes the proof of part (2) of Theorem \ref{IMC-def}.

\subsubsection{Proof of Theorem $\ref{IMC}$}

By Lemma \ref{torsion-free}, the $\Lambda$-module $\Sel(K_\infty,T_\phi)\simeq\Sel(K,\mathbf{T}_\phi)$ is torsion-free and, by Proposition \ref{non zero}, $\kappa_\infty$ is non-zero, so there are only finitely many height $1$ prime ideals $\mathfrak{P}$ of $\Lambda$ such that the image of $\kappa_\infty$ in $\Sel(K,\mathbf{T}_\phi)/\mathfrak{P}$ is trivial. Let $\mathfrak{P}\neq\p\Lambda$ be any height $1$ prime ideal of $\Lambda$ such that
\begin{itemize}
\item $\mathfrak P$ is not a prime divisor of $\Sel(K_\infty, A_\phi)^\vee$;
\item Theorem \ref{control theorem} holds for $\mathfrak{P}$;
\item $\kappa_\infty$ has non-zero image in $\Sel(K,\mathbf{T}_\phi)/\mathfrak{P}$.
\end{itemize}
By the injectivity of the first map in Theorem \ref{control theorem}, the class $\kappa_\infty$ is non-zero in $\Sel(K,T_\mathfrak{P})$; thus, by Proposition \ref{len-indef}, $\Sel(K,T_\mathfrak{P})$ and $\Sel(K,A_\mathfrak{P})$ have rank $1$ and corank $1$ over $\cO_\mathfrak{P}$, respectively. Then there are equalities
\[ \begin{split}
   \rank_\Lambda\Sel(K_\infty,T_\phi)&=\rank_\Lambda \Sel(K,\mathbf{T}_\phi)\\
   & =\rank_{\cO_\mathfrak{P}}\bigl(\Sel(K,\mathbf{T}_\phi)/\mathfrak{P}\bigr)=\rank_{\cO_\mathfrak{P}}\Sel(K,T_\mathfrak{P})=1
   \end{split} \]
and
\[ \begin{split}
   \rank_\Lambda \Sel(K_\infty,A_\phi)^\vee&=\rank_\Lambda \Sel(K,\mathbf{A}_\phi)^\vee\\ &=\corank_\Lambda\Sel(K,\mathbf{A}_\phi)=\corank_{\cO_\mathfrak{P}}\Sel(K,A_\mathfrak{P})=1.
   \end{split} \]
This proves part (1) of Theorem \ref{IMC}.

Now fix a height $1$ prime ideal $\mathfrak{P}=(F)$ as in the proof of Theorem \ref{IMC-def}, where $F$ is a distinguished polynomial. For a sufficiently large integer $m\geq1$, the height $1$ prime ideal $\mathfrak{P}_m\defeq(F+\pi^m)$ satisfies the three conditions of the first part of this proof. Moreover, by Hensel's lemma we can also assume that $\Lambda/\mathfrak{P}\simeq\Lambda/\mathfrak{P}_m$. By part (3) of Proposition \ref{len-indef}, we know that
\[ \length_{\cO_\p}\bigl(\Sel(K,A_{\mathfrak{P}_m})/\Sel(K,A_{\mathfrak{P}_m}{)}_\divv\bigr)=2\cdot\length_{\cO_\p}\bigl(\Sel(K,T_{\mathfrak{P}_m})/\kappa_\infty\cO_{\mathfrak{P}_m}\bigr). \]
On the other hand, arguments in the proof of \cite[Theorem 5.3.10]{MR} show that
\begin{equation} \label{3}
\begin{split}
\length_{\cO_\p}\bigl(\Sel(K,A_{\mathfrak{P}_m})/\Sel(K,A_{\mathfrak{P}_m}{)}_\divv\bigr)=&\;m\cdot\rank_{\cO_\p}(\cO_{\mathfrak{P}_m})\\&\cdot\ord_\mathfrak{P}\bigl(\cchar_\Lambda\Sel(K,\mathbf{A}_\phi)^\vee_\tor\bigr)+O(1),
\end{split} 
\end{equation}
where $O(1)$ denotes an integer bounded independently of $m$. Since one also has
\begin{equation} \label{4}
\length_{\cO_\p}\bigl(\Sel(K,T_{\mathfrak{P}_m})/\kappa_\infty\cO_{\mathfrak{P}_m}\bigr)=m\cdot\rank_{\cO_\p}(\cO_{\mathfrak{P}_m})\cdot\ord_\mathfrak{P}\bigl(\Sel(K,\mathbf{T}_\phi)/\kappa_\infty\Lambda\bigr),
\end{equation}
combining equalities \eqref{3} and \eqref{4} and letting $m$ go to infinity yields the equality
\[
\begin{split}
2\cdot\ord_\mathfrak{P}\bigl(\Sel(K_\infty,{T}_\phi)/\kappa_\infty\Lambda\bigr)&=
2\cdot\ord_\mathfrak{P}\bigl(\Sel(K,\mathbf{T}_\phi)/\kappa_\infty\Lambda\bigr)\\&=\ord_\mathfrak{P}\bigl(\cchar_\Lambda\Sel(K,\mathbf{A}_\phi)^\vee_\tor\bigr)\\
&=\ord_\mathfrak{P}\bigl(\cchar_\Lambda\Sel(K_\infty,{A}_\phi)^\vee_\tor\bigr).
\end{split}
\]
As in \cite{Ho1} and \cite{MR}, the case $\mathfrak{P}=\p\Lambda$ can be treated similarly, now considering the ideals $\mathfrak{P}_m\defeq\p\Lambda+(T^m)$. This completes the proof of part (2) of Theorem \ref{IMC}.

\section{Iwasawa--Greenberg main conjecture for $\mathcal{L}_\wp^{\mathrm{BDP}}$}

Let $f\in S_k(\Gamma_0(N))$ be a newform, let $K$ be an imaginary quadratic field and let $p$ be a prime number as in Section \ref{Iwasawa}. Assume that we are in an indefinite setting, \emph{i.e.}, in the notation of previous sections, the integer $D$ in the splitting $N=MD$ is a square-free product of an \emph{even} number of primes. In addition, assume that $p$ splits in $K$ and write, as before, $p\cO_K=\p\cdot\bar\p$. The goal of this section is to replace the term $\cchar_\Lambda\bigl(\Sel(K_\infty,T_\phi)/\Lambda\cdot\kappa_\infty\bigr)^2$ in the indefinite anticyclotomic Iwasawa main conjecture with a $p$-adic $L$-function; once this replacement is made, the formulation of the indefinite anticyclotomic Iwasawa main conjecture looks more similar to the definite one. 

\subsection{The $p$-adic $L$-function $\mathcal{L}_\p^\mathrm{BDP}$} \label{BDP function}

We start with a short discussion of the $p$-adic $L$-function that will play a role in our subsequent arguments. This $p$-adic $L$-function was introduced (independently) by Bertolini--Darmon--Prasanna (\cite{BDP}) and Brako\v{c}evi\'{c} (\cite{Brako}) when $D=1$; the construction of \cite{BDP} was generalized soon after to $D>1$ by Brooks (\cite{Brooks}). These $p$-adic $L$-functions interpolate central critical values of Hecke characters of imaginary quadratic fields of fixed conductor $c$ not divisible by $p$ and infinity type $(k+j,-j)$ for an integer $j\geq0$. In order to develop Iwasawa theory in this setting, Castella--Hsieh removed when $D=1$ the assumption $p\nmid c$ in the interpolation formula (\cite{CH}); the case $D>1$ was finally developed by Magrone (\cite{Magrone}). Castella--Hsieh and Magrone define the $p$-adic $L$-function that we denote by $\mathcal{L}_\p^\mathrm{BDP}$ and is relevant for our arguments; we shall now discuss the interpolation properties of $\mathcal{L}_\p^\mathrm{BDP}$ and its relation with Heegner cycles. 

Denote by $\cO_\p^\unr $ the tensor product $\cO_\p\otimes_{\Z_p}\Z_p^\unr $, where $\Z_p^\unr $ is the valuation ring of the maximal unramified extension of $\Q_p$. 
The $p$-adic $L$-function $\mathcal{L}_\mathfrak{p}^\mathrm{BDP}$ in \cite[Definition 3.7]{CH} and \cite[Definition 4.3]{Magrone} lives the extended Iwasawa algebra $\Lambda^\unr \defeq\cO_\p^\unr \llbracket\Gal(K_\infty/K)\rrbracket$ and is characterized by the following interpolation property: there exists a $p$-adic period $\Omega_p\in\C_p^\times$ such that for all $p$-adic avatars $\hat\phi$ (in the sense of \cite[Definition 3.4]{CH}) of unramified Hecke characters $\phi$ of $K$ of infinity type $(k/2+t,-(k/2+t))$ with $t\geq 0$ there is an equality
\[ 
\left(\frac{\hat\phi(\mathcal{L}^\mathrm{BDP}_\p)}{\Omega_p^{k+2t}}\right)^2=C_{\phi}\cdot L^\alg(f,\phi,k/2),
\]
where $L^\alg(f,\phi,k/2)$ is the \emph{algebraic part} of the complex $L$-function $L(f,\phi,k/2)$ defined in \cite[p. 584]{CH} and $C_{\phi}$ is a \emph{non-zero} algebraic number that, at least when $D=1$, can be made explicit. The reader is referred to \cite[Proposition 3.8]{CH} and \cite[Theorem 4.6]{Magrone} for further discussion.
%which interpolates special values of the Rankin--Selberg convolutions of $f$ and theta series of Hecke characters of $K$ of infinity type $(k+j,-j)$ with $j\geq 0$. See \cite{BDP} and \cite{CH} for the construction of $\mathcal{L}_\p^\mathrm{BDP}$ in the case $D=1$, and \cite{Brooks} and \cite{Magrone} for our more general setting with $D>1$. More precisely, in \cite{CH} and \cite{Magrone} the authors construct a $p$-adic $L$-function $\mathcal{L}_{\p,\psi}(f)$ that depends on the choice of an auxiliary Hecke character $\psi$ of infinity type $(k/2,-k/2)$ and is closely related to $\mathcal{L}_{\p}^\mathrm{BDP}$, since, for every character $\phi$ of $\Gal(K_\infty,K)$, the equality 

%Roughly speaking, as explained in \cite[Section 4.6]{Magrone}, $\mathcal{L}_{\p,\psi}(f)$ is the projection to $\Lambda^\unr $ of the $p$-adic measure over $\Gal(H_{p^\infty}/K)$ that is attached to the Serre--Tate expansion of (the $p$-depletion of) the automorphic form on the quaternion algebra over $\Q$ of discriminant $D$ associated with $f$ by the Jacquet--Langlands correspondence. In this section, we prove the Iwasawa--Greenberg main conjecture for $\mathcal{L}_\p^\mathrm{BDP}$, which will essentially be a consequence of Theorem \ref{IMC}.

Let $\kappa_\infty^{\mathrm{BDP}}\in H^1(K,\mathbf{T}_f)$ be the element defined in \cite[(5.5)]{CH} for $D=1$ and in \cite[(7.1)]{Magrone} for $D>1$, which is the inverse limit of the compatible sequence of (the corestriction from $H_{p^{n+1}}$ to $K_n$ of) the \emph{generalized} Heegner classes of $p$-power conductor. It is known that $\kappa_\infty^{\mathrm{BDP}}\in\Sel(K,\mathbf{T}_f)$, so (since $H^0(K_\p,\mathbf{T}_f^{[p]})=0$) the localization $\loc_\p(\kappa_\infty^{\mathrm{BDP}})$ belongs to $H^2(K_\p,\mathbf{T}_f^{(p)})$. By \cite[Theorem 5.7]{CH} when $D=1$ and \cite[Theorem 7.2]{Magrone} when $D>1$, there is an injective map of $\Lambda^\unr $-modules 
\[
\mathrm{Log}_\p\colon H^1(K_\p,\mathbf{T}_f^{(p)})\longmono \Lambda^\unr 
\] 
such that the following equality of ideals of $\Lambda^\unr{}$ holds: 
\begin{equation} \label{big-log}
\Bigl(\mathrm{Log}_\p\bigl(\loc_\p(\kappa_\infty^\mathrm{BDP})\bigr)\!\Bigr)=
\bigl(\mathcal{L}_\p^\mathrm{BDP}\bigr).
\end{equation} 
The map $\mathrm{Log}_\p$ is the composition of the big Perrin-Riou map 
(see, \emph{e.g.}, \cite[Theorem 4.7]{LZ-Iwasawa} and \cite[Theorem 5.7]{CH}) with de Rham pairing on the crystalline Dieudonn\'e module of the relevant Galois representations; the injectivity of $\mathrm{Log}_\p$ follows from \cite[Proposition 4.11]{LZ-Iwasawa}. See, \emph{e.g.}, \cite[Theorem A.1]{Castella-London} for this result stated in the current form (in the weight $2$ case). 

The elements $\kappa_\infty$ and $\kappa_\infty^\mathrm{BDP}$ are closely linked, as described by the next result.

\begin{proposition} \label{comparison}
$\kappa_\infty=(2\sqrt{-D_K})^\frac{k-2}{2}\cdot\kappa_\infty^\mathrm{BDP}.$
\end{proposition}

\begin{proof} This follows from the arguments in \cite[\S4.1]{BDPconiveau}. \end{proof}

\begin{corollary} \label{loc-nontorsion-coro}
The local class $\loc_\p(\kappa_\infty)$ is not $\Lambda$-torsion.
\end{corollary}

\begin{proof} By \cite[Theorem 5.7]{Burungale}, $\mathcal{L}_\p^\mathrm{BDP}$ is non-zero at infinitely many characters; in particular, $\mathcal{L}_\p^\mathrm{BDP}\in\Lambda^\unr\smallsetminus\{0\}$. By the injectivity of $\mathrm{Log}_\p$, this fact and equality \eqref{big-log} imply that $\loc_\p(\kappa_\infty^\mathrm{BDP})$ is not $\Lambda$-torsion. The corollary now follows from Proposition \ref{comparison}. \end{proof}

\subsection{Modified Selmer groups}\label{modifiedsection} 

We introduce certain modified Selmer groups that differ from $\Sel(K,\mathbf{T}_f)$ and $\Sel(K,\mathbf{A}_f)$ only in the local conditions at the primes $\p$ and $\bar\p$ above $p$: namely, we remove any condition at $\p$ and impose the vanishing condition at $\bar\p$. More precisely, for $M\in\{\mathbf{T}_f,\mathbf{A}_f\}$ we write $\Sel_{\emptyset,0}(K,M)$ for the set of all classes $s\in H^1(K,M)$ such that
\begin{itemize}
\item $\loc_v(s)\in H^1_\mathrm{unr}(K_v,M)$ for all $v\nmid Dp$;
\item $\loc_v(s)\in H^1_\mathrm{ord}(K_v,M)$ for all $v\,|\,D$;
\item $\loc_{\bar\p}(s)=0$.
\end{itemize}
The following lemma, which is the counterpart in our higher weight setting of results from \cite[Appendix A]{Castella-London}, asserts that $\Sel_{\emptyset,0}(K,\mathbf{A}_f)$ is $\Lambda$-cotorsion.

\begin{lemma}\label{Sel_torsion}
Under the hypotheses of Theorem $\ref{IMC}$, the $\Lambda$-module $\Sel_{\emptyset,0}(K,\mathbf{A}_f)^\vee$ is torsion.
\end{lemma}

\begin{proof} To begin with, global duality gives a short exact sequence of $\Lambda$-modules
\begin{equation} \label{duality-selmer-eq}
0\longrightarrow\coker\Bigl({\loc_{\p}|}_{\Sel_{\mathrm{ord},\emptyset}(K,\mathbf{T}_f)}\Bigr)\longrightarrow\Sel_{\emptyset,0}(K,\mathbf{A}_f)^\vee\longrightarrow\Sel_{\mathrm{ord,0}}(K,\mathbf{A}_f)^\vee\longrightarrow 0,
\end{equation}
where $\Sel_{\ord,\emptyset}(K,\mathbf{T}_f)$ is defined by imposing the same local conditions as for $\Sel(K,\mathbf{T}_f)$ at all primes of $K$ except $\bar{\p}$, 
where no condition is required (thus, $\Sel_{\ord,\emptyset}(K,\mathbf{T}_f)$ contains $\Sel(K,\mathbf{T}_f)$), and $\Sel_{\ord,0}(K,\mathbf{A}_f)$ consists of those classes in $\Sel(K,\mathbf{A}_f)$ that vanish when localized at $\bar\p$.
%; also, 
%\[{\loc_\p|}_{\Sel_{\mathrm{ord},\emptyset}(K,\mathbf{T}_f)}: \Sel_{\ord,\emptyset}(K,\mathbf{T}_f)\longrightarrow H^1_\ord(K_\p,\mathbf{T}_f)\] 
%is the restriction of the localization map to $\Sel_{\mathrm{ord},\emptyset}$ (taking values in the local ordinary cohomology group).
In order to prove the lemma, we show that the second and fourth terms in \eqref{duality-selmer-eq} are $\Lambda$-torsion. 

We first consider the $\Lambda$-module
\[
\coker\Bigl({\loc_\p|}_{\Sel_{\ord,\emptyset}(K,\mathbf{T}_f)}\Bigr)=\frac{H^1_\ord(K_\p,\mathbf{T}_f)}{\im\Bigl({\loc_\p|}_{\Sel_{\ord,\emptyset}(K,\mathbf{T}_f)}\Bigr)}.
\]
The $\Lambda$-module $H^1_\ord(K_\p,\mathbf{T}_f)$ has rank $1$, since it can be identified with $H^1(K_\p,\mathbf{T}_f^{(p)})$ as $H^0(K_\p,\mathbf{T}_f^{[p]})=0$, 
and $H^1(K_\p,\mathbf{T}_f^{(p)})$ is a $\Lambda$-module of rank $1$ (apply 
\cite[Lemma 2.3]{greenberg-cetraro} and take inverse limits). Furthermore, the image of ${\loc_\p|}_{\Sel_{\ord,\emptyset}(K,\mathbf{T}_f)}$ contains $\loc_\p(\kappa_\infty)$, which generates a $\Lambda$-submodule of rank $1$ since, by Corollary \ref{loc-nontorsion-coro}, $\loc_\p(\kappa_\infty)$ is not $\Lambda$-torsion. 
It follows that the $\Lambda$-module $\coker\Bigl({\loc_\p|}_{\Sel_{\ord,\emptyset}(K,\mathbf{T}_f)}\Bigr)$ is torsion. 

Now we show that the $\Lambda$-corank of $\Sel_{\ord,\emptyset}(K,\mathbf{A}_f)$ is $1$. Consider the short exact sequence of $\Lambda$-modules
\begin{equation} \label{duality-selmer-eq2}
0\longrightarrow\coker\Bigl({\loc_\p|}_{\Sel(K,\mathbf{T}_f)}\Bigr)\longrightarrow\Sel_{\emptyset,\ord}(K,\mathbf{A}_f)^\vee\longrightarrow\Sel(K,\mathbf{A}_f)^\vee\longrightarrow 0
\end{equation}
(see again \cite[Theorem 2.3.4]{MR}). The rightmost non-trivial term in \eqref{duality-selmer-eq2} has rank $1$ by Theorem \ref{IMC}, while $\coker\Bigl({\loc_\p|}_{\mathrm{Sel}(K,\mathbf{T}_f)}\Bigr)$ has rank $0$ for the same reason as above. Thus, $\rank_\Lambda\Sel_{\emptyset,\ord}(K,\mathbf{A}_f)^\vee=1$. Complex conjugation, which interchanges $\p$ and $\bar{\p}$, induces an isomorphism of $\Lambda$-modules $\Sel_{\emptyset,\ord}(K,\mathbf{A}_f)\simeq \Sel_{\ord,\emptyset}(K,\mathbf{A}_f)$; it follows that 
\begin{equation} \label{Lambdarank}
\rank_\Lambda\Sel_{\ord,\emptyset}(K,\mathbf{A}_f)^\vee=1.
\end{equation}
Similarly as before, for a height $1$ prime ideal $\mathfrak{P}$ of $\Lambda$, let $\Sel_{\ord,\emptyset}(K,A_\mathfrak{P})$ be the submodule of $H^1(K,A_\mathfrak{P})$ consisting of classes defined by the same local conditions as $\Sel(K,\mathbf{A}_\mathfrak{P})$ at all primes different from $\bar{\p}$ and 
whose restriction at $\bar{\p}$ lies in the maximal divisible subgroup $H^1(K_{\bar{\p}},A_\mathfrak{P})_\mathrm{div}$ of $H^1(K_{\bar{\p}},A_\mathfrak{P})$ (thus, this is the Selmer group defined by propagation from $\Sel_{\ord,\emptyset}(K,V_\mathfrak{P})$ defined by taking the same local conditions as $\Sel(K,V_\mathfrak{P})$ outside $\bar{\p}$ and the local condition $H^1(K_{\bar{\p}},V_\mathfrak{P})$ at $\bar{\p}$). We also define $\Sel_{\ord,0}(K,A_\mathfrak{P})$ to be the submodule of $\Sel(K,A_\mathfrak{P})$ consisting of those classes whose restriction at $\bar{\p}$ is trivial. Finally, for each integer $i\geq 0$ 
define $\Sel_{\ord,\emptyset}(K,A_\mathfrak{P}[\mathfrak{P}^i])$ and $\Sel_{\ord,0}(K,A_\mathfrak{P}[\mathfrak{P}^i])$ by propagation from $\Sel_{\ord,\emptyset}(K,A_\mathfrak{P})$ and $\Sel_{\ord,0}(K,A_\mathfrak{P})$ via the inclusion $A_\mathfrak{P}[\mathfrak{P}^i]\subset A_\mathfrak{P}$. 
Let $F_\mathfrak{P}$ be the fraction field of $\cO_\mathfrak{P}$. A generalization of \cite[Theorem 4.1.13]{MR} (see also \cite[Proposition 1.2.3]{Agboola-Howard}) shows that there is an isomorphism
\begin{equation} \label{MR4.1.13}
\Sel_{\ord,\emptyset}(K,A_\mathfrak{P}[\mathfrak{P}^i])\simeq (F_\mathfrak{P}/\mathcal{O}_\mathfrak{P})^r[\mathfrak{P}^i]\oplus 
\Sel_{\ord,0}\bigl(K,A_{{\mathfrak{P}}^\iota}\bigl[(\mathfrak{P}^\iota)^i\bigr]\bigr)
\end{equation} 
for some integer $r\geq0$, where $\iota:\Lambda\rightarrow\Lambda$ is the standard involution acting as inversion on group-like elements.
%We have then an inclusion $\Sel_{\mathrm{ord},0}(K,A_\mathfrak{P})\subset \Sel_{\mathrm{ord},\emptyset}(K,A_\mathfrak{P})$ and an exact sequence 
%\[0\longrightarrow \Sel_{\mathrm{ord},0}(K,A_\mathfrak{P})\longrightarrow 
%\Sel_{\mathrm{ord},\emptyset}(K,A_\mathfrak{P})\xrightarrow{\loc_{\bar{\p}}}H^1(K_{\bar{\p}},A_\mathfrak{P}).
%\]
By \cite[Lemma 1.3.3]{Ho1}, there are isomorphisms  
\begin{equation} \label{CTMR}
\begin{split}
\Sel_{\ord,0}\bigl(K,A_\mathfrak{P}[\mathfrak{P}^i]\bigr)\simeq\Sel_{\ord,0}(K,A_\mathfrak{P})[\mathfrak{P}^i],\\
\Sel_{\ord,\emptyset}\bigl(K,A_\mathfrak{P}[\mathfrak{P}^i]\bigr)\simeq\Sel_{\ord,\emptyset}(K,A_\mathfrak{P})[\mathfrak{P}^i].
\end{split}
\end{equation}
Combining \eqref{MR4.1.13} and \eqref{CTMR} yields an isomorphism 
\begin{equation}\label{rank1}\Sel_{\ord,\emptyset}(K,A_\mathfrak{P})[\mathfrak{P}^i]\simeq (F_\mathfrak{P}/\mathcal{O}_\mathfrak{P})^r[\mathfrak{P}^i]\oplus 
\Sel_{\ord,0}(K,A_{\mathfrak{P}^\iota})\bigl[(\mathfrak{P}^\iota)^i\bigr].
\end{equation}
Using the Poitou--Tate exact sequence and the global Euler characteristic formula for $A_\mathfrak{P}[\mathfrak{P}^i]$ (see, \emph{e.g.}, \cite[Theorem 2.18]{DDT}), we find
\[
\begin{split}
r=\length_{\cO_\mathfrak{P}}\Bigl(H^1_\ord\bigl(K_\p,A_\mathfrak{P}[\mathfrak{P}^i]\bigr)\!\Bigr)&+
\length_{\cO_\mathfrak{P}}\Bigl(H^1\bigl(K_{\bar{\p}},A_\mathfrak{P}[\mathfrak{P}^i]\bigr)\!\Bigr)\\&-\length_{\cO_\mathfrak{P}}\Bigl(H^0\bigl(\C,A_\mathfrak{P}[\mathfrak{P}^i]\bigr)\!\Bigr).
\end{split}
\]
Now $\length_{\cO_\mathfrak{P}}\bigl(H^0(\C,A_\mathfrak{P}[\mathfrak{P}^i])\bigr)=2$, while the other two lengths are $1$ at $\p$ (as an application of \cite[Lemma 2.3]{greenberg-cetraro}) and $2$ at $\bar\p$ (by local Tate duality). It follows that $r=1$. Thus, taking direct limits over $i$ in \eqref{rank1} gives
\begin{equation} \label{rank2}
\Sel_{\ord,\emptyset}(K,A_\mathfrak{P}) \simeq (F_\mathfrak{P}/\mathcal{O}_\mathfrak{P}) \oplus\Sel_{\ord,0}(K,A_{\mathfrak{P}^\iota}).
\end{equation}
By an analogue of Theorem \ref{control theorem} for the Selmer groups $\Sel_{\ord,\emptyset}(K,\mathbf{A}_f)$ and $\Sel_{\ord,0}(K,\mathbf{A}_f)$, 
the canonical maps 
\[
\Sel_{\ord,\emptyset}(K,A_\mathfrak{P})\longrightarrow \Sel_{\ord,\emptyset}(K,\mathbf{A}_f)[\mathfrak{P}],\quad\Sel_{\ord,0}(K,A_\mathfrak{P})\longrightarrow \Sel_{\ord,0}(K,\mathbf{A}_f)[\mathfrak{P}]
\]
induced by the inclusion $A_\mathfrak{P}=\mathbf{A}_f[\mathfrak{P}]\subset \mathbf{A}_f$ are injective and have finite cokernels that are bounded independently of $\mathfrak{P}$. Therefore, from \eqref{rank2} we obtain 
\[ \rank_\Lambda\Sel_{\ord,\emptyset}(K,A_\mathfrak{P})^\vee=1+\rank_\Lambda\Sel_{\ord,0}(K,A_\mathfrak{P})^\vee, \]
and the conclusion follows from \eqref{Lambdarank}. \end{proof}

\subsection{The Iwasawa--Greenberg main conjecture for $\mathcal{L}_\p^\mathrm{BDP}$} 

We are ready to prove our result on the Iwasawa--Greenberg main conjecture for the $p$-adic $L$-function $\mathcal{L}_\p^\mathrm{BDP}$.

\begin{theorem} \label{IG-thm}
Under the hypotheses of Theorem $\ref{IMC}$, there is an equality
\[
\cchar_\Lambda\bigl(\Sel_{\emptyset,0}(K,\mathbf{A}_f)^\vee\bigr)\otimes_\Lambda \Lambda^\unr=\bigl(\mathcal{L}_\p^\mathrm{BDP}\bigr)^2
\]
of ideals of $\Lambda^\unr $.
\end{theorem}

This is Theorem \ref{C-thm} in the introduction. In the proof below, for a torsion $\Lambda$-module $M$ and a prime ideal $\mathfrak P$ of $\Lambda$, write $\ell_{\mathfrak P}(M)$ as a shorthand for $\length_{\Lambda_{\mathfrak P}}(M\otimes_\Lambda\Lambda_{\mathfrak P})$.

\begin{proof} By the exact sequence \eqref{duality-selmer-eq}, for every height $1$ prime ideal $\mathfrak{P}$ of $\Lambda$ there is an equality
\begin{equation} \label{leng_1}
%\begin{split}
\ell_\mathfrak{P}\bigl(\Sel_{\emptyset,0}(K,\mathbf{A}_f)^\vee\bigr)=\ell_\mathfrak{P}\bigl(\Sel_{\ord,0}(K,\mathbf{A}_f)^\vee\bigr)+\ell_\mathfrak{P}\Bigl(\coker\bigl({\loc_\p |}_{\Sel_{\ord,\emptyset}(K,\mathbf{T}_f)}\bigr)\!\Bigr).
%\end{split}
\end{equation}
By \cite[Lemma 2.3, (3)]{Castella-London}, there is an equality
\[
\ell_\mathfrak{P}\bigl(\Sel_{\ord,0}(K,\mathbf{A}_f)^\vee\bigr)=\ell_\mathfrak{P}\bigl(\Sel_{\ord,\emptyset}(K,\mathbf{A}_f)^\vee_\mathrm{tors}\bigr).
\]
The action of complex conjugation and short exact sequence \eqref{duality-selmer-eq2} imply that
\begin{equation}\label{Cast-2.3}
\Scale[0.92]
{\begin{split}
\ell_\mathfrak{P}\bigl(\Sel_{\ord,\emptyset}(K,\mathbf{A}_f)^\vee_\tor\bigr)&=\ell_\mathfrak{P}\bigl(\Sel_{\emptyset,\ord}(K,\mathbf{A}_f)^\vee_\tor\bigr)\\
&=\ell_\mathfrak{P}\Bigl(\coker\bigl({\loc_\p |}_{\Sel(K,\mathbf{T}_f)}\bigr)\!\Bigr)+\ell_\mathfrak{P}\bigl(\Sel(K,\mathbf{A}_f)^\vee_\tor\bigr).
\end{split}}
\end{equation}
Moreover, counting $\Lambda$-ranks in the exact sequence 
\[ 
0\longrightarrow\frac{\Sel_{\ord,\emptyset}(K,\mathbf{T}_f)}{\Sel(K,\mathbf{T}_f)}\longrightarrow \frac{H^1(K_{\bar\p},\mathbf{T}_f)}{H^1_\ord(K_{\bar\p},\mathbf{T}_f)}\longrightarrow \Sel(K,\mathbf{A}_f)^\vee\longrightarrow \Sel_{\ord,0}(K,\mathbf{A}_f)^\vee\longrightarrow 0,
\]
which is a consequence of global duality, and using the fact that $H^1(K_{\bar\p},\mathbf{T}_f)\big/H^1_\ord(K_{\bar\p},\mathbf{T}_f)$ is torsion-free of rank $1$ over $\Lambda$, one gets the equality
\begin{equation}\label{equality-Selmer}
\Sel_{\ord,\emptyset}(K,\mathbf{T}_f)=\Sel(K,\mathbf{T}_f).
\end{equation}
Thus, combining \eqref{leng_1}, \eqref{Cast-2.3} and \eqref{equality-Selmer} gives
\begin{equation} \label{step2}
\Scale[0.92]
{\begin{aligned}
\ell_\mathfrak{P}\bigl(\Sel_{\emptyset,0}(K,\mathbf{A}_f)^\vee\bigr)&=\ell_\mathfrak{P}\bigl(\Sel(K,\mathbf{A}_f)^\vee_\tor\bigr)+2\cdot\ell_\mathfrak{P}\Bigl(\coker\bigl({\loc_\p|}_{\Sel(K,\mathbf{T}_f)}\bigr)\!\Bigr)\\
&=2\cdot\biggl(\ell_\mathfrak{P}\bigl(\Sel(K,\mathbf{T}_f)/\Lambda\cdot \kappa_\infty\bigr)+\ell_\mathfrak{P}\Bigl(\coker\bigl({\loc_\p|}_{\Sel(K,\mathbf{T}_f)}\bigr)\!\Bigr)\!\biggr), 
\end{aligned}}
%\raisetag{-25pt}
\end{equation}
where the second equality follows from Theorem \ref{IMC}.

On the other hand, $\Sel(K,\mathbf{T}_f)$ and $H^1_\ord(K_\p,\mathbf{T}_f)$ have both rank $1$ and $\coker\bigl({\loc_\p |}_{\Sel(K,\mathbf{T}_f)}\bigr)$ has rank $0$ (as $\loc_\p(\kappa_\infty)$ is not $\Lambda$-torsion), so $\loc_\p\colon\Sel(K,\mathbf{T}_f)\rightarrow H^1_\ord(K_\p, \mathbf{T}_f)$ is injective; in fact, the kernel of this map is $\Sel_{0,\ord}(K,\mathbf{T}_f)$, which is trivial since it is $\Lambda$-torsion and contained in $\Sel(K,\mathbf{T}_f)$, whose freeness over $\Lambda$ is ensured by Lemma \ref{torsion-free}.

This implies that there is a short exact sequence 
\[
0\longrightarrow\frac{\Sel(K,\mathbf{T}_f)}{\Lambda\cdot\kappa_\infty}\longrightarrow \frac{H^1_\ord(K,\mathbf{T}_f)}{\Lambda\cdot\loc_\p(\kappa_\infty)}\longrightarrow\coker\bigl({\loc_\p |}_{\Sel(K,\mathbf{T}_f)}\bigr)\longrightarrow 0
\]
of $\Lambda$-modules of rank $0$, and hence \eqref{step2} becomes
\[
\ell_\mathfrak{P}\bigl(\Sel_{\emptyset,0}(K,\mathbf{A}_f)^\vee\bigr)=2\cdot \ell_\mathfrak{P}\biggl(\frac{H^1_\ord(K_\p,\mathbf{T}_f)}{\Lambda\cdot \loc_\p(\kappa_\infty)}\biggr).
\]
Let $\mathfrak{P}'$ be a prime ideal of $\Lambda^\mathrm{unr}$ above $\mathfrak{P}$ and write $\ell_{\mathfrak P'}$ in place of $\length_{\Lambda^{\mathrm{unr}}_{\mathfrak P'}}$. The injectivity of the $\mathrm{Log}_\p$ map in \S \ref{BDP function}, equality \eqref{big-log} and Proposition \ref{comparison} yield the equality
\[ 
\ell_{\mathfrak{P}'}\Bigl(\Sel_{\emptyset,0}(K,\mathbf{A}_f)^\vee\otimes_\Lambda\Lambda^\mathrm{unr}\Bigr)=2\cdot \ell_{\mathfrak{P}'}\biggl(\frac{\Lambda^\mathrm{unr}}{\Lambda^\mathrm{unr}\cdot \mathcal{L}_\p^\mathrm{BDP}}\biggr).
\]
The conclusion follows as $\mathfrak{P}$ varies over all the height $1$ prime ideals of $\Lambda$.
\end{proof}

\bibliographystyle{amsplain}
\bibliography{Perrin-Riou}

\end{document}